\documentclass[a4paper,11pt]{amsart}

\pdfoutput=1

\usepackage[text={400pt,660pt},centering]{geometry}

\usepackage[utf8]{inputenc}
\usepackage[T1]{fontenc}
\usepackage{amsthm, amssymb, amsmath, amsfonts, mathrsfs}
\usepackage{mathtools}
\usepackage{scalerel} 
\usepackage[colorlinks=true, pdfstartview=FitV, linkcolor=blue, citecolor=blue, urlcolor=blue,pagebackref=false]{hyperref}

\usepackage{esint} 
\usepackage{MnSymbol} 
\usepackage{bbm}

\usepackage{appendix}

\usepackage{color}

\usepackage{microtype}

\usepackage[notcite,notref,color]{showkeys}
\definecolor{labelkey}{gray}{.8}
\definecolor{refkey}{gray}{.8}

\definecolor{darkgreen}{rgb}{0,0.5,0}
\definecolor{darkblue}{rgb}{0,0,0.7}
\definecolor{darkred}{rgb}{0.9,0.1,0.1}

\newtheorem{proposition}{Proposition}
\newtheorem{theorem}[proposition]{Theorem}
\newtheorem{lemma}[proposition]{Lemma}
\newtheorem{corollary}[proposition]{Corollary}

\theoremstyle{remark}
\newtheorem{remark}[proposition]{Remark}

\theoremstyle{definition}

\newtheorem{hypothesis}[proposition]{Hypothesis}

\numberwithin{equation}{section}
\numberwithin{proposition}{section}

\renewcommand{\leq}{\leqslant}
\renewcommand{\geq}{\geqslant}

\renewcommand{\subset}{\subseteq}
\newcommand{\mcl}{\mathcal}

\newcommand{\A}{{\mathsf{A}}}
\newcommand{\F}{\mathcal{F}}

\newcommand{\K}{\mathcal{K}}

\newcommand{\DD}{\mathbf{D}} 
\newcommand{\cm}{\mathbf{\Sigma}} 
\newcommand{\D}{{\mathsf{D}}} 

\newcommand{\ccc}{\mathbf{c}}

\newcommand{\E}{\mathbb{E}}
\newcommand{\Er}{\mathbb{E}_{\rho}}
\renewcommand{\Pr}{\mathbb{P}_{\rho}}

\newcommand{\expec}[1]{\left\langle #1 \right\rangle}
\renewcommand{\L}{\mathcal{L}}
\newcommand{\Lb}{{\overbracket[1pt][-1pt]{\L}}}
\newcommand{\Lbssep}{{\overbracket[1pt][-1pt]{\L}}^{\mathrm{ssep}}}
\newcommand{\Pb}{\bar P}
\newcommand{\Pbssep}{{\bar P}^{\, \mathrm{ssep}}}

\newcommand{\La}{\Lambda}

\newcommand{\N}{\mathbb{N}}

\newcommand{\Ll}{\left}
\newcommand{\Rr}{\right}

\newcommand{\rhs}{right-hand side}
\newcommand{\1}{\mathbf{1}}
\newcommand{\R}{\mathbb{R}}

\newcommand{\Z}{\mathcal{Z}}
\newcommand{\Zo}{\mathbb{Z}}
\newcommand{\Zt}{\mathbb{Z}^2}
\newcommand{\Zd}{{\mathbb{Z}^d}}

\renewcommand{\P}{\mathbb{P}}
\newcommand{\ov}{\overline}
\renewcommand{\bar}{\overline}

\renewcommand{\tilde}{\widetilde}

\renewcommand{\d}{{\mathrm{d}}}
\newcommand{\var}{\mathbb{V}\!\mathrm{ar}}
\newcommand{\cov}{\mathbb{C}\mathrm{ov}}

\renewcommand{\epsilon}{\varepsilon}

\newcommand{\X}{\mathcal{X}} 

\newcommand{\g}{\mathbf{g}} 

\newcommand{\HH}{\mathcal{H}}

\newcommand{\deff}{d_{\operatorname{eff}}}

\newcommand{\cu}{{\scaleobj{1.2}{\square}}}

\renewcommand{\r}{\mathbf{r}}

\newcommand{\fil}{\mathscr{F}}

\DeclareMathOperator{\supp}{supp}

\newcommand{\norm}[1]{\left\Vert{#1}\right\Vert}
\newcommand{\bracket}[1]{\left\langle{#1}\right\rangle}

\usepackage{hyperref}

\title[Additive functional of non-gradient exclusion processes]{Scaling limit of additive functionals for reversible non-gradient exclusion process: critical cases}

\author[C. Gu]{Chenlin Gu} 
\address[Chenlin Gu]{Yau Mathematical Sciences Center, Tsinghua University, Beijing, China}
\email{gclmath@tsinghua.edu.cn}

\author[L. Yang]{Linzhi Yang} 
\address[Linzhi Yang]{Qiuzhen College, Tsinghua University, Beijing, China}
\email{ylz24@mails.tsinghua.edu.cn}

\author[L. Zhao]{Linjie Zhao} 
\address[Linjie Zhao]{ School of Mathematics and Statistics $\&$ Hubei Key Laboratory of Engineering Modeling and Scientific Computing, Huazhong University of Science and Technology, Wuhan, China.}
\email{linjie\_zhao@hust.edu.cn}

\begin{document}
	\begin{abstract}
		For the reversible speed-change exclusion process $(\eta_t)_{t \geq 0}$ in $\Zd$, we study the scaling limit of additive functionals ${\Gamma_t(f) = \int_0^t f(\eta_s) \, \d s}$. Concerning the local centered function $f$, the previous work [Commun. Math. Phys. 104, 1–19, 1986] by Kipnis and Varadhan and [Comm. Pure Appl. Math., 66: 649-677, 2013] by Gon{\c{c}}alves and Jara respectively covered the cases $d \geq 3$ and $d=1$. The present paper completes the missing part $d=2$, and also develops the theory for functions with higher degree. The novelty is a quantitative homogenization of the resolvent, which allows to overcome the obstacle of correlation function in non-gradient models. 
		
		\bigskip
		
		
		\medskip
		
		
	\end{abstract}
	\maketitle
	
	\setcounter{tocdepth}{1}
	\tableofcontents
	
	\newpage
	
	%
	%
	%
	%
	%
	%
	%
	%
	\section{Introduction}
	
	A classical problem in interacting particle system is to study the long-time behavior of the additive functional
	\begin{align*}
		\Gamma_t(f) := \int_0^{t} f(\eta_s) \, \d s.
	\end{align*}
	In the seminal work \cite{KV86}, Kipnis and Varadhan proved that, for reversible interacting particle systems $(\eta_t)_{t\geq 0}$ and  a centered function $f$ of finite $H^{-1}$ norm, then $\Ll(\frac{1}{\sqrt{N}}\Gamma_{Nt}(f)\Rr)_{t \geq 0}$ has a weak limit as a Brownian motion as $N \to \infty$. 
	
	However, not all functions have a finite $H^{-1}$ norm, thus different scaling appears to obtain a meaningful limit. Taking the simple symmetric exclusion process (SSEP) on $\Zd$ for example, Kipnis proposed in \cite{kipnis87} the normalization constants for centered local function $f$ in different \emph{dimensions}
	\begin{align}\label{eq.defa}
		a(N,d) := \Ll\{\begin{array}{cc}
			N^{\frac{3}{4}}& d = 1, \\
			\sqrt{N \log N}& d = 2, \\
			\sqrt{N} & d \geq 3,
		\end{array}\Rr.
	\end{align}
    and studied the normalized additive functional 
	\begin{align}\label{eq.defGamma}
		\Gamma^N_t(f) := \frac{1}{a(N,d)} \int_0^{Nt} f(\eta_s) \, \d s.
	\end{align}
    The scaling limit of marginal distribution was also obtained in \cite{kipnis87}. Then the limit of the whole process was proved by Sethuraman and Xu in \cite{sethuraman1996central} for $d\geq 3$, and by Sethuraman in \cite{sethuraman2000central} for $d\leq 2$: the limit is a fractional Brownian motion of Hurst index $\frac{3}{4}$ when $d = 1$, and a Brownian motion for $d = 2$ or $d \geq 3$. Similar results have also been established in \cite{CG84} for independent random walk by Cox and Griffeath, and in \cite{QJS02} for the zero-range model by Quastel, Jankowski, and Sheriff, and in \cite{sheriff2011central} for a Ginzburg--Landau model by Sheriff.

    It is possible that the scaling limit degenerates to $0$ under the normalization \eqref{eq.defa}. This is related to the \emph{degree} of the function $f$, which was observed firstly in \cite{sethuraman1996central}. In \cite{sheriff2011central}, the normalization constant $\sqrt{N \log N}$ was clarified for the case $(\text{dim},\text{deg})=(1,2)$, which is like $(\text{dim},\text{deg})=(2,1)$. As a summary from the literature, one should define \emph{the effective dimension} 
    \begin{align}\label{eq.effdimIntro}
        \deff := \text{dim} \times \text{deg},
    \end{align}
    then $\Ll(\frac{1}{a(N,\deff)} \int_0^{Nt} f(\eta_s)\,\d s\Rr)_{t \geq 0}$ has a non-trivial weak limit as $N \to \infty$ for a general local centered function $f$ on SSEP.  
	
    The scaling limit for reversible speed-change exclusion is also expected, while the rigorous results are less complete due to its nature of non-gradient process. Except a direct application of \cite{KV86, sethuraman1996central} to $d \geq 3$, the justified case is $(\text{dim},\text{deg})=(1,1)$ by Gon{\c{c}}alves and Jara \cite{GJ13}. The main result in this paper can be stated informally as follows.
    \begin{theorem}\label{thm.informal}
        For the speed-change reversible exclusion process $(\eta_t)_{t \geq 0}$ and every local centered function $f$,  the additive functional $\Ll(\frac{1}{a(N,\deff)} \int_0^{Nt} f(\eta_s)\,\d s\Rr)_{t \geq 0}$ admits a non-trivial weak limit as $N \to \infty$.
    \end{theorem}
    Especially, the technique of homogenization is developed to treat the critical case $\deff = 2$, which contains two situations $(\text{dim},\text{deg}) \in \{(1,2),(2,1)\}$.
	
	\subsection{Main result}
	Let $\Zd$ be the Euclidean lattice and $\X := \{0,1\}^{\Zd}$ stand for the configuration space of the exclusion process. We denote by $\eta = \{\eta(x): x \in \Zd \}$ the canonical element in $\X$. Here $\eta(x) = 0$ means the site $x$ is vacant and $\eta(x) = 1$ means the site is occupied. We write $y \sim x$ if $x$ and $y$ are neighbors, i.e. $\vert x - y\vert = 1$. Then $\{x,y\}$ is called an (undirected) bond. For every $\Lambda \subset \Zd$, we denote by $\Lambda^*$ the set of bonds in $\Lambda$ that 
	\begin{align}\label{eq.defBond}
		\Lambda^* := \{ \{x,y\}: x,y \in \Lambda, x \sim y\}.
	\end{align} 
	
	For $x,y \in \Zd$, the exchange operator $\eta^{x,y}$ is defined as 
	\begin{align}\label{eq.exchange}
		(\eta^{x,y})(z) := \Ll\{\begin{array}{ll}
			\eta(z), & \qquad z \neq x,y; \\
			\eta(y), & \qquad z = x; \\
			\eta(x), & \qquad z = y.
		\end{array}\Rr.
	\end{align}
	In particular, when $b = \{x,y\}$ is a bond, we also write $\eta^b$ instead of $\eta^{x,y}$, and define the Kawasaki operator $\pi_b \equiv \pi_{x,y}$ 
	\begin{align}\label{eq.Kawasaki}
		\pi_b F(\eta) := F(\eta^b) - F(\eta).
	\end{align} 
	For every $x \in \Zd$, the translation operator $\tau_x$ is defined as 
	\begin{align}\label{eq.translation1}
		(\tau_x \eta)(y) := \eta({x+y}),
	\end{align}
	and given a function $F$ on $\X$, we also define $\tau_x F$ as 
	\begin{align}\label{eq.translation2}
		(\tau_x F)(\eta) := F(\tau_x \eta).
	\end{align}
	
	The \emph{non-gradient exclusion process} on $\Zd$ is defined by the generator below
	\begin{align}\label{eq.Generator}
		\L := \sum_{b \in (\Zd)^*} c_b(\eta) \pi_b = \frac{1}{2}\sum_{x,y \in \Zd: \vert x - y\vert = 1} c_{x,y}(\eta) \pi_{x,y},
	\end{align}
	where the family of functions
	\begin{align}
		\{c_b(\eta) \equiv c_{x,y}(\eta) = c_{y,x}(\eta); b=\{x,y\} \in (\Zd)^*\},
	\end{align}
	determine the jump rate of particles on the nearest bonds. This model is also called \emph{the speed-change Kawasaki dynamics} or \emph{the lattice gas} in the literature. 
	
	We suppose the following conditions for the jump rate throughout the paper without specific explanation.
	\begin{hypothesis}\label{hyp} The following conditions are assumed for  $\{c_b\}_{b \in (\Zd)^*}$.
		\begin{enumerate}
			\item Non-degenerate and local: $c_{x,y}(\eta)$ depends only on $\{\eta(z): \vert z - x\vert \leq \r\}$ for some integer $\r > 0$, and is bounded on two sides $0 < c_{-} \leq c_{x,y}(\eta) \leq c_+$.
			\item Spatially homogeneous: for all $\{x,y\} \in (\Zd)^*$, $c_{x,y} = \tau_x c_{0,y-x}$.
			\item Detailed balance under Bernoulli measures: $c_{x,y}(\eta)$ is independent of $\{\eta(x), \eta(y)\}$.
		\end{enumerate}
	\end{hypothesis}
	This model is known to be of non-gradient type, i.e. one cannot find functions $\{h_{i,j}\}_{1\leq i,j\leq d}$ such that $c_{0,e_i} (\eta)(\eta({e_i}) - \eta(0) ) = \sum_{j=1}^d \Ll((\tau_{e_j} h_{i,j})(\eta) - h_{i,j}(\eta)\Rr)$ for general $\{c_{b}\}_{b \in (\Zd)^*}$, with $\{e_i\}_{1 \leq i \leq d}$ the canonical basis of $\Zd$. 
	
	For every $\rho \in (0,1)$, the Bernoulli product measure $\Pr = \mathrm{Bernoulli}(\rho)^{\otimes \Zd}$ is an invariant measure on $\X$. We denote it by $\Pr$ with the associated expectation $\Er$. In the following statement, we denote by
	\begin{align}\label{eq.def.frho}
		\bar{f}(\rho) := \Er[f], \qquad \bar{f}'(\rho) := \frac{\d}{\d \rho} \Er[f], \qquad  \bar{f}''(\rho) := \frac{\d^2}{\d \rho^2} \Er[f].
	\end{align}
	The quantities $\DD \in \R^{d \times d}_{sym}$ and $\chi \in \R_+$ are respectively \emph{the diffusion matrix} and \emph{the compressibility}. The $H^{-1}$-norm is defined as $\norm{f}^2_{-1} := \Er[f(-\L)^{-1}f]$; see respectively \eqref{eq.Einstein}, \eqref{eq.defCompress}, and \eqref{eq.def.InnerMinus1} for their rigorous definitions.
	
	The first main theorem of the paper studies the scaling limit of $(\Gamma^N_t(f))_{t \geq 0}$ defined in \eqref{eq.defGamma}.
	\begin{theorem}\label{thm.main_basic}
		Under Hypothesis~\ref{hyp} and under the stationary measure $\Pr$, for every centered local function $f$, the additive functional $(\Gamma^N_t(f))_{t \geq 0}$ converges weakly in $C(\R_+, \R)$ to a limit $(\Gamma^\infty_t(f))_{t \geq 0}$, which is characterized  as follows:
		\begin{itemize}
			\item $d = 1$, $(\Gamma^\infty_t(f))_{t \geq 0}$ is a $\frac{3}{4}$-fractional Brownian motion with covariance matrix
			\begin{align}\label{eq.fBM}
				\cov[\Gamma^\infty_t(f), \Gamma^\infty_s(f)] = \frac{2\chi(\rho)}{3\sqrt{\pi \DD (\rho)}}\bar{f}'(\rho)^2 \Ll(t^{\frac{3}{2}} + s^{\frac{3}{2}} - \vert t - s\vert^{\frac{3}{2}} \Rr).
			\end{align}
			\item $d = 2$, $(\Gamma^\infty_t(f))_{t \geq 0}$ is a Brownian motion with covariance
			\begin{align}\label{eq.BM_d2}
				\cov[\Gamma^\infty_t(f), \Gamma^\infty_s(f)] = \frac{\chi(\rho)\bar{f}'(\rho)^2}{2\pi\sqrt{\det[\DD(\rho)]}} (t \wedge s).
			\end{align}
			\item $d \geq 3$, $(\Gamma^\infty_t(f))_{t \geq 0}$ is a Brownian motion with covariance
			\begin{align}\label{eq.BM_d3}
				\cov[\Gamma^\infty_t(f), \Gamma^\infty_s(f)] =  2 \norm{f}^2_{-1} (t \wedge s).
			\end{align}
		\end{itemize}
	\end{theorem}
	We remark that, the case $d \geq 3$ is essentially an application of the Kipnis--Varadhan theorem, and the case $d=1$ was treated by Gon{\c{c}}alves and Jara in \cite[Theorem~2.5]{GJ13} using the local Boltzmann--Gibbs principle. Thus, our main contribution is to fill the missing part $d = 2$. 
	
	Meanwhile, Theorem~\ref{thm.main_basic} is still not the complete story. A very natural question is to ask what happens if the leading order constant $\bar{f}'(\rho)$ vanishes in Theorem~\ref{thm.main_basic}. One may expect different normalization constant, as discussed in \cite[Theorem~1.2, (ii)]{QJS02}. The normalization is determined by the flatness of the local function $f$, which can be calibrated by the \emph{degree} (under the invariant measure $\Pr$) defined as 
	\begin{align}\label{eq.degree}
		\deg(f) := \min\Ll\{n \in \N: \bar{f}^{(n)}(\rho) \neq 0\Rr\}.
	\end{align}
	Here $\bar{f}^{(n)}(\rho)$ is the $n$-th derivative of  $\bar{f}(\rho)$ defined in \eqref{eq.def.frho} with respect to $\rho$, which is well-defined when $f$ is a local function. We make the convention  $\inf\emptyset=\infty$ for functions $f$ with vanishing derivative of any finite order. In the early work \cite{sethuraman1996central}, the threshold to apply the Kipnis--Varadhan theorem was proposed for the zero-range process and simple symmetric exclusion, and it can actually be summarized in the following way. Define the \emph{effective dimension} of a function $f$ to be 
    \begin{align}\label{eq.effdim}
        \deff(f):= \text{dimension} \times \deg(f). 
    \end{align}
    Then the threshold is
	\begin{align}\label{eq.threshold}
		 \deff(f) > 2.
	\end{align}  
	In other words, both the degree and the dimension contribute to the regularity. When the effective dimension is high enough, the Kipnis--Varadhan theorem applies. The cases $\deff(f) = 2$ are thus critical. A classification for the long-range exclusion can be found in \cite{bernardin2016occupation}. The following theorem, as the counterpart of  \cite[Theorem~1.2, (ii)]{QJS02} in the non-gradient reversible setting, completes the results. An illustration of the phases can be found in Figure~\ref{fig.phase}.
	\begin{figure}[h!]
		\centering
		\includegraphics[width=0.7\textwidth]{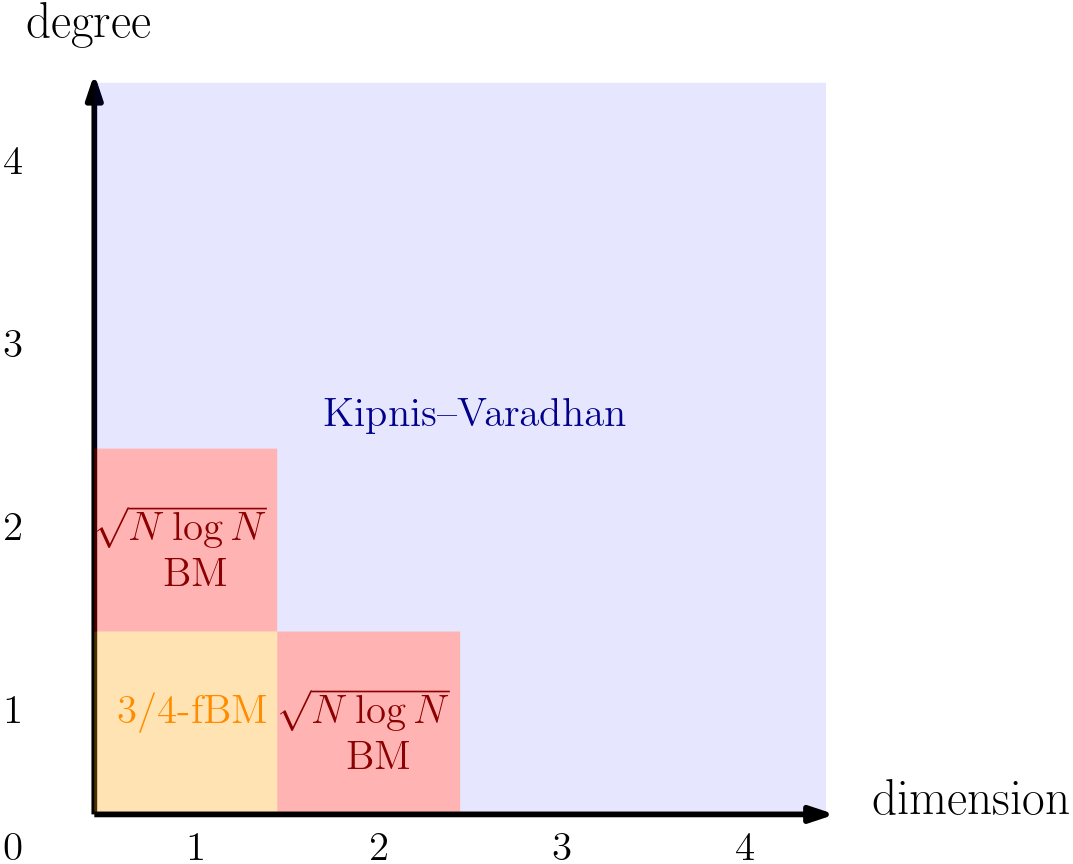}
		\caption{The diagram of the additive functional.}\label{fig.phase}
	\end{figure}
	
	\begin{theorem}\label{thm.main_next}
		Under Hypothesis~\ref{hyp} and under the stationary measure $\Pr$, for every centered local function $f$, the additive functional has a scaling limit in $C(\R_+, \R)$ for the following cases. 
		\begin{itemize}
			\item $d = 1$, $\deg(f) = 2$, then 
			\begin{align}\label{eq.def12}
				\Ll(\frac{1}{\sqrt{N \log N}} \int_0^{Nt} f(\eta_s) \, \d s\Rr)_{t \geq 0} \Longrightarrow (B_t)_{t \geq 0}, 
			\end{align}
			has the scaling limit as a Brownian motion of covariance
			\begin{align}\label{eq.BM12}
				\cov[B_t, B_s] = \frac{\chi(\rho)^2}{4\pi\DD(\rho)}\vert \bar{f}''(\rho)  \vert^2 (t \wedge s).
			\end{align}
			\item If $\deff(f) > 2$, i.e. $\deg(f)  > 2/d$, then 
			\begin{align}\label{eq.def13}
				\Ll(\frac{1}{\sqrt{N}} \int_0^{Nt} f(\eta_s) \, \d s\Rr)_{t \geq 0} \Longrightarrow (B_t)_{t \geq 0}, 
			\end{align}
			has the scaling limit as a Brownian motion of covariance
			\begin{align}\label{eq.BM13}
				\cov[B_t, B_s] = 2 \norm{f}^2_{-1} (t \wedge s).
			\end{align}
		\end{itemize}
	\end{theorem}

	\subsection{Related literature} It has a long history to investigate the additive functionals of  interacting particle systems. Besides the results discussed in the beginning, there are many other developments in this direction. Sethuraman studied the mean zero case for exclusion processes in \cite{sethuraman2000central}.  In the asymmetric case, the behavior of the additive functionals of the exclusion process depends on whether
	the density $\rho=1/2$ or not. For $\rho \neq 1/2$, one expects the occupation time is diffusive. This was proved  by Seppäläinen and Sethuraman \cite{seppalainen2003transience}   in dimension one, and was extended to dimension two by Bernardin \cite{bernardin2004fluctuations}. When $\rho = 1/2$, the variance of the occupation time grows super-linearly. In dimension one, Bernardin \cite{bernardin2004fluctuations} showed that the variance of the occupation time grows at least $N^{5/4}$ in some sense.  Li and Mao \cite{li2008upper} proved an upper bound of order $N^{3/2}$. Sethuraman \cite{sethuraman2006superdiffusivity} obtained a lower bound of order $N \log \log N$ in dimension two.  Gonçalves and
	Jara \cite{GJ13} proved invariance principles of additive functionals for a large class of particle systems in dimension one by proving the so-called  local Boltzmann--Gibbs principle, which allows to relate the  additive functionals to empirical measures of the process. Bernardin, Gonçalves and Sethuraman  
	\cite{bernardin2016occupation} studied additive functionals of the
	exclusion process with long jumps. We refer the reader to  \cite{komorowski2012fluctuations} for an excellent review on this topic. 
    
    We underline that the above literature concerns the stationary case. Only a few results concern non-equilibrium fluctuations of the additive functionals; see \cite{erhard2024nonequilibrium,xu2025nonequilibrium,fontes2021additive} for recent results.
	
	One important input of this paper is the homogenization theory; see the classical reference \cite{bensoussan1979boundary, jikov1994}. The link between the homogenization and the particle systems was revealed in the work \cite{varadhanII} by Varadhan, which applies to the non-gradient models. We also refer to \cite{bannai2024topological, bannai2021varadhan, BS25} for the further generalization of Varadhan's argument. The quantitative homogenization theory has been developed in the last decade; see the monograph \cite{AKMbook, gloria2011optimal, armstrong2022elliptic} and the proceeding \cite{armstrong2025coarse}. The quantitative homogenization is also extended to particle systems in \cite{bulk, funaki2024quantitative, gu2024quantitative, gu2025relaxation}. 
    
    Very recently, the homogenization of \emph{superdiffusion} has attracted considerable attention. Its roots, however, are distinct: it arises from long‑range jumps in \cite{bou2026quantitative, chen2026quantitative} and from acceleration in random divergence‑free fields in \cite{chatzigeorgiou2025gaussian, armstrong2026superdiffusion}. The latter is also related to \cite{yau, LQSY04} in particle systems. The $\sqrt{N \log N}$ normalization in the present paper is due to correlations of the particle density.

	\subsection{Outline of the proof}
	In this part, we give a panorama of the proof. Sections~\ref{subsec.SemiOcc} and ~\ref{subsec.MartingaleCLT} recall the main ingredients of the martingale CLT in the work \cite{KV86}. Section~\ref{subsec.HomOfReso} gives the idea how the homogenization technique works at the critical phase $\deff = 2$, and this subsection takes the case $(\text{dim},\text{deg})=(2,1)$ for example. Then Section~\ref{subsec.HigherOrder} mentions the necessary modification in the case $(\text{dim},\text{deg})=(1,2)$.
	
	\subsubsection{Semigroup and occupation time}\label{subsec.SemiOcc}
	Let us mention at first the links between the additive functionals, the occupation time, and the semigroup. The normalization constant $a(N,d)$ is closely related to the asymptotic behavior of the semigroup $P_t = e^{t\L}$. To see this, let us calculate the second moment of $\Gamma^N_t(f)$
	\begin{equation}\label{eq.double_integral}
		\begin{split}
			\Er\Ll[\Gamma^N_t(f)^2\Rr] &= \frac{2}{a(N,d)^2} \int_0^{Nt} \int_0^{s_1} \Er\Ll[f P_{s_1 - s_2} f\Rr] \, \d {s_2} \, \d {s_1}, \\
			&= \frac{2}{a(N,d)^2} \int_0^{Nt} \int_0^{s_1} \Er\Ll[\Ll(P_{(s_1-s_2)/2} f\Rr)^2\Rr]  \, \d {s_2} \, \d {s_1}.
		\end{split}
	\end{equation}
	Then $a(N,d)$ is the correct order if one assumes the following decay of the semigroup
	\begin{align}\label{eq.heat_kernel_decay}
		\var_\rho[P_t f] = C\Ll(\bar{f}'(\rho)\Rr)^2 t^{-\frac{d}{2}} + o(t^{-\frac{d}{2}}).
	\end{align}
	Moreover, this estimate allows us to reduce the study of additive functionals to the occupation time when $d \leq 2$, i.e.
	\begin{align*}
		f(\eta) =  \bar{f}'(\rho) \bar{\eta}(0) + (f(\eta) - \bar{f}'(\rho) \bar{\eta}(0)).
	\end{align*}
	Here $\bar{\eta}(0)$ is the centered random variable defined in \eqref{eq.defeta_center}.
	Since the second term has mass zero, \eqref{eq.heat_kernel_decay} will yield 
	\begin{align*}
		d \leq 2, \qquad \Er[\Gamma^N_t(f(\eta) - \bar{f}'(\rho) \bar{\eta}(0))^2] \xrightarrow{N \to \infty} 0.
	\end{align*}
	Therefore, for $d \leq 2$, it suffices to study 
	\begin{align}\label{eq.Occupation}
		\Gamma^N_t := \frac{1}{a(N,d)} \int_0^{Nt}  \bar{\eta}_s(0) \, \d s
	\end{align}
	The estimate \eqref{eq.heat_kernel_decay} is explicit for independent random walks \cite{CG84}, and requires justification in other models. It is also an important input in the previous work including \cite{QJS02, sheriff2011central}. Concerning the speed-change exclusion process, it was obtained very recently in \cite{gu2025relaxation} by the first two authors of this paper.
	
	\subsubsection{Martingale CLT}\label{subsec.MartingaleCLT}
	The classical approach for deriving the scaling limit of the occupation time \eqref{eq.Occupation} relies on the martingale argument. Consider the resolvent 
	\begin{align}\label{eq.resolvent}
		(\lambda - \L) G_\lambda = f.
	\end{align}
	It defines a martingale
	\begin{align}\label{eq.defMt}
		M_{\lambda, t} := G_\lambda(\eta_{t}) - G_\lambda(\eta_0) - \int_0^{t} \L  G_\lambda(\eta_s) \, \d s. 
	\end{align}
	Then the occupation time in \eqref{eq.Occupation} can be reformulated as 
	\begin{multline}\label{eq.Gamma_decom}
		\Gamma^N_t = \frac{1}{a(N,d)} M_{\lambda, Nt} \\
		- \frac{1}{a(N,d)} \Ll(G_\lambda(\eta_{Nt}) - G_\lambda(\eta_0)\Rr) + \frac{1}{a(N,d)} \int_0^{Nt}  \lambda G_\lambda(\eta_s) \, \d s.
	\end{multline}
	The second line is considered as the remainder. 

	We need a martingale convergence to conclude the proof. For the reasonable choice of $\lambda$, we denote by 
	\begin{align}\label{eq.defMt_new}
		M^{N}_t := \frac{1}{a(N,d)} M_{\lambda, Nt}.
	\end{align}
	The scaling limit is naturally related to the quadratic variation of $M^{N}_t$
	\begin{align}\label{eq.quadra}
		\bracket{M^{N}}_t = \frac{1}{a(N,d)^2} \int_0^{Nt} \sum_{b \in (\Zd)^*} \Ll(c_b (\pi_b G_\lambda)^2\Rr)(\eta_s) \, \d s.
	\end{align}
	Its limit requires different techniques. 
	\begin{itemize}
		\item When $d \geq 3$, then $(-\L)^{-1}$ is well-defined for all the centered local functions $f$ thanks to \eqref{eq.heat_kernel_decay}.  We can choose $\lambda = 0$, then the ergodic theorem applies to \eqref{eq.quadra}.
		\item When $d = 1$, the result was proved by relating the occupation time to the fluctuation field by proving the local Boltzmann--Gibbs principle; see \cite{GJ13}.
		\item When $d = 2$, the non-trivial choice is to let $\lambda = \frac{1}{N}$. Since $G_{1/N}$ also depends on $N$, the ergodic theorem does not apply directly to \eqref{eq.quadra}. Meanwhile, if the particle system is SEP, $G_{1/N}$ has an explicit expression that allows us to derive the convergence of quadratic variation.
	\end{itemize}
	In the speed-change model and $d = 2$, we also need to choose $\lambda = \frac{1}{N}$, but then $G_{1/N}$ does not have any explicit expression. This is the main challenge and explains why related result of the critical case was missing in the literature. One major contribution of this paper is to propose a substitution $\tilde{G}_\lambda$ for $G_\lambda$, which helps close the martingale approach.

	\subsubsection{Homogenization of resolvent}\label{subsec.HomOfReso} Thanks to the reduction by occupation time in \eqref{eq.Occupation},  we take $f = \bar{\eta}(0)$ in \eqref{eq.resolvent} for $G_\lambda$ when $d=2$, and the strategy to study it is inspired from the homogenization theory.  We propose its \emph{two-scale expansion ansatz} 
	\begin{align}\label{eq.TwoScaleExpan}
		\tilde{G}_\lambda:=\bar{G}_\lambda+\sum_{z\in\Z_m}\sum_{i=1}^{d}(\D_{e_i}\bar{g}_\lambda)_{z+\cu_m}\phi_{m,e_i}^z.
	\end{align}
	Here $\bar{G}_\lambda$ is the solution for the resolvent equation of a SEP with the same diffusion matrix:
	\begin{equation}\label{eq.ResolventSEP}
		(\lambda-\Lb)\bar{G}_\lambda=\bar{\eta}(0).
	\end{equation}
	The SEP with generator $\Lb$ is considered as a \emph{homogenized process} for the speed-change model, and the details for the construction of this SEP is discussed in  Subsection~\ref{subsec.Generator}.
	By the self-duality of the SEP, $\bar{G}_\lambda$ has an explicit expression as:
	\begin{equation*}
		\bar{G}_\lambda(\eta)=\sum_{x\in \mathbb{Z}^d}\bar{g}_\lambda(x)\bar{\eta}(x).
	\end{equation*}
	The second term in the two-scale expansion \eqref{eq.TwoScaleExpan}, $m \in \N_+$ is chosen to be a mesoscopic scale compared to $\lambda^{-1}$, and $(\D_{e_i}\bar{g}_\lambda)_{z+\cu_m}$ is the local average of the function in a cube $z+\cu_m$. The function $\phi_{m,e_i}^z$ is \emph{the local corrector}. Their rigorous definitions can be found in the beginning of Section~\ref{subsec.HomH1}.
	
	The technique of two-scale expansion comes from the homogenization theory, and is related to \emph{the replacement lemma} in the literature of particle systems. The main observation is 
	\begin{align*}
		\tilde{G}_\lambda \simeq  G_\lambda,
	\end{align*}
	while the former has a more explicit characterization. Relying on recent progress in quantitative homogenization of non-gradient particle systems \cite{bulk, gu2024quantitative, funaki2024quantitative, gu2025relaxation}, we can justify the heuristic above; see the discussion in Appendix~\ref{appendix}.
	
	Once understanding the ansatz above, a more direct idea is to consider a martingale associated with $\tilde{G}_\lambda$
	\begin{align}\label{eq.defTildeMt}
		\tilde{M}_{\lambda, t} := \tilde{G}_\lambda(\eta_{t}) - \tilde{G}_\lambda(\eta_0) - \int_0^{t} \L  \tilde{G}_\lambda(\eta_s) \, \d s, 
	\end{align}
	without touching $G_\lambda$ and $M^{N}_t$. This is also the approach taken in this paper. 
	By \eqref{eq.ResolventSEP} we can then decompose $\bar{\eta}(0)$ as:
	\begin{equation*}
		\bar{\eta}(0)=(\lambda-\Lb)\bar{G}_\lambda=-\L\tilde{G}_\lambda+\lambda\bar{G}_\lambda+(\L\tilde{G}_\lambda-\Lb\bar{G}_\lambda).
	\end{equation*}
	The occupation time in \eqref{eq.Occupation} can be reformulated as 
	\begin{equation}\label{eq.Gamma_decomTilde}
		\Gamma^N_t = \tilde{M}^{N}_t + R^N_t,
	\end{equation}
	with the martingale part defined as
	\begin{align}\label{eq.defMt_newTilde}
		\tilde{M}^{N}_t := \frac{1}{a(N,d)} \tilde{M}_{\lambda, Nt},
	\end{align}
	and a remainder  
	\begin{multline}\label{eq.remainder}
		R^N_t := \frac{1}{a(N,d)} \Ll( \tilde{G}_\lambda(\eta_0) - \tilde{G}_\lambda(\eta_{Nt}) + \int_0^{Nt}  \lambda \bar{G}_\lambda(\eta_s) + (\L\tilde{G}_\lambda-\Lb\bar{G}_\lambda)(\eta_s) \, \d s \Rr).
	\end{multline}


	Notice that the martingale $\tilde{M}^{N}_t$ also depends on $\lambda$ and $m$. With a careful choice of these parameters, one can derive the martingale CLT when $d=2$.
	
	\begin{proposition}\label{prop.CLT_M}
		For $d = 2$, let $\lambda = \frac{1}{N}$ and $m$ satisfy 
		\begin{equation}\label{eq.scaleM}
			1 \ll 3^{18m} \ll \log N,
		\end{equation}
		(e.g. $m=\lfloor\tfrac{1}{36} \log \log N\rfloor$), then the martingale $\tilde{M}^{N}_t$ defined in \eqref{eq.defMt_newTilde} converges weakly to a Brownian motion of variance $\sigma^2=\tfrac{\chi(\rho)}{2\pi\sqrt{\det[\DD(\rho)]}}$.
	\end{proposition}
	
	\subsubsection{Higher degree cases}\label{subsec.HigherOrder}
	Since the mass on the first chaos vanishes, for the special case $d=1,\ \deg(f)=2$, it suffices to study the occupation time on two different sites:
	\begin{align}\label{eq.OccupationDeg2}
		\Gamma^N_t := \frac{1}{\sqrt{N\log N}} \int_0^{Nt}  \bar{\eta}_s(0) \bar{\eta}_s(1)\, \d s.
	\end{align}
	
	Identifying the limit of \eqref{eq.OccupationDeg2} is non-trivial, even for the simple symmetric exclusion process, which is the simplest one among the exclusion models. Up to the authors' knowledge, although the limit of \eqref{eq.OccupationDeg2} for SSEP was discussed in \cite[Theorem 1.2 (ii)]{QJS02}, some details of proof were omitted there. We present a complete calculation in this paper, because it requires efforts to handle some technical difficulties hidden in the Green function of two particles with exclusion.  
	
	Similarly as before, instead of directly using the martingale defined by the resolvent
	\begin{align}\label{eq.ResolventDeg2}
		(\lambda-\mathcal{L}) G_\lambda = \bar{\eta} (0) \bar{\eta} (1),
	\end{align} 
	we utilize the two-scale expansion $\tilde{G}_\lambda$ to define the martingale \eqref{eq.defTildeMt}, where $\tilde{G}_\lambda$ is again connected to the homogenized resolvent $\bar{G}_\lambda$:
	\begin{equation}\label{eq.ResolventSSEP}
		(\lambda-\Lb)\bar{G}_\lambda=\bar{\eta}(0)\bar{\eta}(1).
	\end{equation}
	The detailed expression of $\tilde{G}_\lambda$ in terms of $\bar{G}_\lambda$ is given in \eqref{eq.TwoscaleExDeg2}. Because the averaged slope is no longer a constant, we need to re-establish some homogenization estimates in \cite[Section 4]{gu2025relaxation} (i.e. Proposition~\ref{prop.L2FluxDeg2}). Combining some precise estimates of the resolvent, we obtain the following result in parallel with Proposition~\ref{prop.CLT_M}:
	\begin{proposition}\label{prop.CLT_MDeg2}
		For $d = 1$, let $\lambda = \frac{1}{N}$ and $m$ satisfy 
		\begin{equation}\label{eq.scaleMDeg2}
			1 \ll 3^{76m} \ll \log N,
		\end{equation}
		(e.g. $m=\lfloor\tfrac{1}{100} \log \log N\rfloor$), then the martingale $\tilde{M}^{N}_t$ defined in \eqref{eq.defMt_newTilde} with respect to $\tilde{G}_\lambda$ defined in \eqref{eq.TwoscaleExDeg2} converges weakly to a Brownian motion of variance $\sigma^2=\tfrac{\chi(\rho)^2}{\pi\DD(\rho)}$.
	\end{proposition}
	
	\subsection{Organization of paper}
	We will review some results from the previous paper in Section~\ref{sec.pre}. Section~\ref{sec.KV} is devoted to the situations where the classical Kipnis--Varadhan theorem applies. Afterward, the remaining two critical cases $(\text{dim},\text{deg})=(2,1)$ and $(\text{dim},\text{deg})=(1,2)$ will be treated in Section~\ref{sec.critical_d_2} and Section~\ref{sec.critical_d_1} separately. As mentioned, the basic ideas in these two sections are similar, but the proof of the case $(\text{dim},\text{deg})=(1,2)$ is heavier.
	
	\section{Preliminaries}\label{sec.pre}
	\subsection{Notation}
	For every $\Lambda \subset \Zd$, we denote by $\fil_{\Lambda}$ the $\sigma$-algebra generated by $(\eta(x))_{x \in \Lambda}$ and we write $\fil$ for $\fil_{\Zd}$.
	We denote by $\F_0$ the local function space on $\X$ and $\F_0(\La)$ if the function is $\fil_{\La}$-measurable. The parameter $\rho \in (0,1)$ is fixed throughout the paper. For $p \geq 1$,  $L^p$ is a shorthand notation for $L^p(\X, \fil, \Pr)$, and $\ell^p$ refers to the $\ell^p$-norm on the relevant countable set.

	\subsubsection{Diffusion matrix}
	
	Let us give the detailed definition of the physical quantities in the theorem. We refer to \cite[Part II, Proposition~2.2]{spohn2012large} and \cite[(1.5)]{fuy} for the background. The diffusion matrix $\DD : (0,1) \to \R^{d \times d}$ is defined by
	\begin{align}\label{eq.Einstein}
		\DD(\rho) := \frac{\ccc(\rho)}{2 \chi(\rho)},
	\end{align}
	where $\chi(\rho)$ is \emph{the compressibility}
	\begin{align}\label{eq.defCompress}
		\chi(\rho) := \rho (1-\rho),
	\end{align}  
	and  $\ccc(\rho)$ is \emph{the effective conductivity} defined as follows. We construct a quadratic form with respect to the function $F \in \F_0^d$ 
	\begin{align}\label{eq.defQuadra}
		\xi \cdot \ccc(\rho; F) \xi = \frac{1}{2} \sum_{\vert x\vert = 1} \Er \Ll[c_{0,x}\Ll(\xi \cdot \Ll\{ x(\eta(x) - \eta(0)) - \pi_{0,x}(\sum_{y \in \Zd} \tau_y F)\Rr\}\Rr)^2\Rr],
	\end{align}
	where $\F_0^d := (\F_0)^d$. Then $\ccc(\rho)$ is the minimization of $\ccc(\rho; F)$
	\begin{align}\label{eq.defC}
		\xi \cdot \ccc(\rho) \xi := \inf_{F \in \F_0^d}\xi \cdot \ccc(\rho; F) \xi.
	\end{align}
	
	\subsubsection{Generator and semigroup}\label{subsec.Generator}
	Several generators will be used in this paper. The generator of the speed-change exclusion is defined in \eqref{eq.Generator}. We denote by $\Lbssep$ the generator of the standard SSEP
	\begin{align}\label{eq.Lbssep}
		\Lbssep := \sum_{b \in (\Zd)^*}  \pi_b,
	\end{align}
	We also denote by $\Lb$ the generator of a constant jump rate $Q$ with finite range
	\begin{align}\label{eq.Lb}
		\Lb := \frac{1}{2}\sum_{x \in \Zd} \sum_{y \in \Zd} Q_{y-x} \pi_{x,y},
	\end{align}
	such that its associated SEP has the same diffusion matrix $\DD(\rho)$ as $\L$. The existence of such $\Lb$ is an elementary exercise in linear algebra; see \cite[Corollary~2.3]{gu2025relaxation} for the justification. The notations for the semigroups are defined as follows 
	\begin{align}\label{eq.semigroup}
		P_t := e^{t\L}, \qquad \Pb_t := e^{t\bar{\L}}, \qquad \Pbssep_t := e^{t \bar{\L}^{\,\mathrm{ssep}}}.
	\end{align}
	
	\subsubsection{Fock space}
	The degree of function is related to the chaos expansion in the Fock space. The following paragraph collects some facts, and one can find their proof in \cite{Privault2008}; see also \cite[Section~3.1]{gu2025relaxation}.
	
	We define the centered field for every $x \in \Zd$	
	\begin{align}\label{eq.defeta_center}
		\bar{\eta}(x):=\eta(x)-\rho,
	\end{align}
	and for every finite subset $Y \subset \Zd$
	\begin{equation*}
		\bar{\eta}(Y):=\prod_{x\in Y}\bar{\eta}(x).
	\end{equation*}
	The collection of finite subsets of cardinal $n$ is denoted by 
	\begin{align}\label{eq.defKn}
		\K_n := \{Y \subset \Zd: \vert Y\vert = n\}.
	\end{align}
	The multi-dimensional stochastic integral for Bernoulli random variables is defined as
	\begin{equation}\label{eq.DefIn}
		\forall n \geq 1, \qquad I_n(f_n):=\sum_{Y\in\K_n}f_n(Y)\bar{\eta}(Y).
	\end{equation}
	It is well-defined for $f_n \in \ell^1(K_n)$, but then can be extended to $f_n \in \ell^2(K_n)$ thanks to the orthogonal sum. Afterward, we can define the Fock space of order $n$ 
	\begin{equation*}
		\forall n \in \N_+, \qquad \HH_n:=\{I_n(f_n):f_n\in\ell^2(\K_n)\},
	\end{equation*}
	and we keep the convention $\HH_0:=\R$.
	
	The Fock space gives the orthogonal decomposition of $L^2(\X, \fil, \Pr)$
	\begin{align}\label{eq.Fock2}
		L^2(\X, \fil, \Pr) = \bigoplus_{n=0}^{\infty}\HH_n.
	\end{align}
	More precisely, for every $f\in L^{2}(\X,\fil,\P_\rho)$, it can be decomposed as
	\begin{equation}\label{eq.chaos} 
		f =\sum_{n=0}^{\infty}I_n(T_n f),
	\end{equation}
	where the equality makes sense in $L^{2}(\X,\fil,\P_{\rho})$, and the coefficient $T_n f$ is defined by the inner product
	\begin{align}\label{eq.Tnf}
		T_n f(Y):=  \chi(\rho)^{-|Y|} \Er[\bar{\eta}(Y)f(\eta)].
	\end{align}
	The coefficient $T_n f$ is related to the degree in the following identity. For every local function $f  $, we have 
	\begin{equation}\label{eq.PertFormu}
		\frac{\d^n}{\d\rho^n}\Er[f]=n!\sum_{Y\in\K_n}T_n f(Y).
	\end{equation}
	Thus, if $\deg(f) = n$, its mass of sum vanishes in $\HH_{0 \leq k \leq n-1}$.
	
	\subsubsection{Geometry}
	The domain with various boundary conditions $\Lambda^-, \overline{\Lambda^*}, \Lambda^+$ are defined as follows. We define $\partial \Lambda$ as the boundary set of $\Lambda$, and denote by $\Lambda^-$ its interior that 
	\begin{align}\label{eq.defBoundary}
		\partial \Lambda := \{ x \in \Lambda : \exists y \notin \Lambda, y \sim x\}, \qquad \Lambda^- := \Lambda \setminus \partial \Lambda.
	\end{align}
	Recall that the bonds set of $\Lambda$ is defined as $\Lambda^*$ in \eqref{eq.defBond}. We  define its enlarged version 
	\begin{align}\label{eq.defBondLarge}
		\overline{\Lambda^*} := \{\{x,y\}: x \in \Lambda, y = x+e_i, i=1,2, \cdots, d\},
	\end{align}
	where $e_i \in \Zd$ is the $i$-th directed unit vector. We also denote by $\Lambda^+$ the vertices concerned in \eqref{eq.defBondLarge}
	\begin{align} \label{eq.defVertexLarge}
		\Lambda^+:=\Lambda\cup\bigcup_{i=1}^d(\Lambda+e_i).
	\end{align}

	For all $m \in \N$, let $\cu_m$ stand for the lattice cube of side length $3^m$ 
	\begin{align}\label{eq.defcu}
		\cu_m := \Ll(- \frac{3^m}{2}, \frac{3^m}{2}\Rr)^d \bigcap \Zd.
	\end{align}
	We also use the notation $\Z_m := 3^m \Zd$ as the centers of triadic cubes.
	
	We denote all lattices at most $\r$ away from the enlarged cube $\La^+$ by $N_\r(\La^+)$:
	\begin{equation}\label{eq.defNrLa}
		N_\r(\La^+):=\{x\in\Zd:\operatorname{dist}(x,\La^+)\leq \r\}.
	\end{equation}
	
	\subsection{Relaxation to equilibrium}
	
	The sharp asymptotic for the relaxation to equilibrium of non-gradient exclusion was recently derived in \cite{gu2025relaxation} by the first and the second author of this paper.
	\begin{proposition}[Theorem~1.2 of \cite{gu2025relaxation}]\label{prop.relaxation}
		There exists a positive exponent $\delta(d,\r, c_{-},c_+)$, such that for every local function $f$, we have 
		\begin{align}\label{eq.relaxation}
			\var_{\rho}[P_t f] = \frac{\bar{f}' (\rho)^2 \chi(\rho)}{\sqrt{(8\pi t)^{d} \det [\DD(\rho)]}} + o(t^{-\frac{d+\delta}{2}}).
		\end{align}
		Here the remainder depends on $f$ and satisfies ${\lim_{t \to \infty}t^{\frac{d+\delta}{2}}o(t^{-\frac{d+\delta}{2}}) = 0}$.
	\end{proposition}
	
	For the case of SEP, a faster decay was obtained much earlier by Bertini and Zegarlinski in \cite{berzeg} with the generalized Nash inequality.
	\begin{proposition}[Theorem~17 of \cite{berzeg}]\label{prop.fasterDecay}
		For every $n\in\N_+$, there exists a constant $C$ only depending on $n,d$ and $\rho$ such that the following estimate holds for all local function $f\in \HH_n$ with a finite triple norm $|||\cdot|||_n$:  
		\begin{align}\label{eq.fasterDecay}
			\var_{\rho}[\Pb_t f] \leq C t^{-\frac{nd}{2}} |||f|||^2_n.
		\end{align}
	\end{proposition}
	
	The precise definition for the triple norm $|||\cdot|||_n$ can be found in $|||\cdot|||_{n,1}$ in \cite[(5.5)]{berzeg}; see also \cite[(3.14)]{gu2025relaxation}.
	
	\subsection{$H^{-1}$-norm and Kipnis--Varadhan theorem}
	
	The parameter $\rho \in (0,1)$ is fixed throughout the paper. We then define the $L^2$-inner product 
	\begin{align}\label{eq.def.InnerL2}
		\bracket{f,g} := \Er[fg],
	\end{align}
	where the parameter $\rho$ is omitted. Afterward, we define the $H^1$-inner product associated to $\L$
	\begin{align}\label{eq.def.InnerH1}
		\bracket{f,f}_1 := \Er[f(-\L f)] = \bracket{f, -\L f}, 
	\end{align} 
	and $H^{-1}$-inner product 
	\begin{align}\label{eq.def.InnerMinus1}
		\bracket{f,f}_{-1} := \sup_{g \in \F_0} \Ll\{2 \bracket{f,g} - \bracket{g,g}_1 \Rr\}.
	\end{align}
	For convenience, we sometimes also refer to it as $\bracket{f, (-\L)^{-1}f}$. We also denote by $\norm{\, \cdot \,}_{1}, \norm{\, \cdot \,}_{-1}$ the associated norm, and let $H^{-1}$ refer to the function space with finite $\norm{\, \cdot \,}_{-1}$ norm.
	
	The $H^{-1}$ condition is related to the seminal Kipnis--Varadhan theorem. We refer to \cite[Corollary~1.5]{KV86} and \cite[Theorem~2.7, Theorem~2.33]{komorowski2012fluctuations}.
	
	\begin{proposition}[Kipnis--Varadhan]\label{prop.KV}
		For every function $f \in L^2 \cap H^{-1}$, as ${N \to \infty}$, the additive functional $\Ll(\frac{1}{\sqrt{N}} \int_0^{Nt} f(\eta_s) \, \d s\Rr)_{t \geq 0}$ converges in distribution to a Brownian motion with diffusive constant $\norm{f}^2_{-1}$.
	\end{proposition}

	\section{Regime of Kipnis--Varadhan theorem}\label{sec.KV}
	In this section, we discuss at first the cases where the Kipnis--Varadhan theorem applies in non-gradient exclusion. 
	\begin{proposition}\label{prop.KVnonGradient}
		Every local function $f$ satisfying $\deg(f)  > 2/d$ belongs to $H^{-1}$. Then the Kipnis--Varadhan theorem applies to the additive functional $\Ll(\frac{1}{\sqrt{N}} \int_0^{Nt} f(\eta_s) \, \d s\Rr)_{t \geq 0}$.
	\end{proposition}
	A similar discussion has already appeared in \cite[Theorems~1.1, 1.2]{sethuraman1996central} respectively for the SSEP and zero-range process; \cite[Theorem~1.3]{sethuraman1996central} covers the non-gradient process for $d=1$. The proof below relies on the chaos expansion, the relaxation estimates \eqref{eq.relaxation}, \eqref{eq.fasterDecay}, and the observation below.
	
	\begin{lemma}\cite[Proposition~3.5]{GJ13}\label{lem.H-1Es}
		For every local function $f:\X\rightarrow\R$ satisfying 
		\begin{equation}\label{eq.Equiv0}
			\forall \tilde{\rho} \in (0,1), \qquad	\E_{\tilde{\rho}}[f] \equiv 0,
		\end{equation}
		then $f \in H^{-1}$.
	\end{lemma}
	\begin{proof}[Proof of Proposition~\ref{prop.KVnonGradient}]
		The cases $d = 2$ and $d\geq 3$ can be analyzed directly using \eqref{eq.relaxation}
		\begin{align*}
			\norm{f}^2_{-1} = \int_0^\infty \bracket{f,P_t f} \, \d t = \int_0^\infty \bracket{P_{t/2} f, P_{t/2} f} \, \d t.
		\end{align*}
		We handle the integral when $\deg(f) > 2/d$:
		\begin{itemize}
			\item For $d \geq 3$, we have $\var_{\rho}[P_t f]  =  C_f t^{-\frac{d}{2}} + o(t^{-\frac{d}{2}})$, so the integral is finite.
			\item For $d = 2$ and $\deg(f) > 1$, i.e. $\deg(f) \geq 2$, then $\bar{f}'(\rho) = 0$ and the leading term in \eqref{eq.relaxation} vanishes, which yields $\var_{\rho}[P_t f] \leq C t^{-(1+\delta/2)}$. The integral thus is also finite.
		\end{itemize}
		
		We then treat the case $d = 1$ and $\deg(f) \geq 3$, which requires a more precise decay estimate. The comparison between $\L$ and $\Lbssep$ yields
		\begin{equation}\label{eq.H1comparison}
			\begin{split}
				\norm{f}^2_{-1} 	&:=\sup_{g \in \F_0}\{2\expec{f,g}-\expec{g,g}_1\}\\
				&\leq \sup_{g \in \F_0}\Ll\{2\expec{f,g}- C\expec{g,(-\Lbssep)g}\Rr\}\\
				&=C^{-1} \sup_{g \in \F_0}\left\{2\expec{f,C g}-\expec{Cg, (-\Lbssep)Cg}\right\}\\
				&=C^{-1}\expec{f, (-\Lbssep)^{-1}f}\\
				&=C^{-1}\int_{0}^{\infty}\expec{f,\Pbssep_t f} \d t.
			\end{split}
		\end{equation} 	
		Thus we only need to estimate the $H^{-1}$-norm associated to SSEP. Since $f$ is local and centered, its chaos expansion \eqref{eq.chaos} is a sum of finite terms for some $N \in \N$
		\begin{align*}
			f = I_1(T_1 f) + I_2(T_2 f) + \sum_{n=3}^{N}I_n(T_n f).
		\end{align*}
		All the three parts belong to $H^{-1}$ for the following reason:
		\begin{itemize}
			\item For the term $\sum_{n=3}^{N}I_n(T_n f)$, we apply \eqref{eq.H1comparison} and the faster decay \eqref{eq.fasterDecay}. The decay rate is at least $t^{-\frac{3}{2}}$, so the integral in \eqref{eq.H1comparison} is finite.
			\item For the terms $I_1(T_1 f), I_2(T_2 f)$, the condition $\deg(f) \geq 3$ and \eqref{eq.PertFormu} imply that 
			\begin{align*}
				\sum_{x \in \Zd}T_1 f(x) = 0, \qquad \sum_{Y\in\K_2}T_2 f(Y) = 0.
			\end{align*}
			Then for every $\tilde{\rho} \in (0,1)$, we verify the condition \eqref{eq.Equiv0}
			\begin{align*}
				\E_{\tilde{\rho}}\Ll[I_1(T_1 f)\Rr] &= (\tilde{\rho} - \rho) 	\sum_{x \in \Zd}T_1 f(x)= 0,\\
				\E_{\tilde{\rho}}\Ll[I_2(T_2 f)\Rr] &= (\tilde{\rho} - \rho)^2 	\sum_{Y\in\K_2}T_2 f(Y)= 0.
			\end{align*}
			Therefore, Lemma~\ref{lem.H-1Es} applies and $I_1(T_1 f), I_2(T_2 f) \in H^{-1}$.
		\end{itemize}
	\end{proof}

	\section{Critical case \texorpdfstring{$d=2$}{} and \texorpdfstring{$\deg(f)=1$}{}}\label{sec.critical_d_2}
	
	In this section, we focus on the critical case $d=2$ and $\deg(f)=1$. As discussed in Section~\ref{subsec.SemiOcc}, the key is to study the occupation time \eqref{eq.Occupation} in :
	\begin{align*}
		\Gamma^N_t := \frac{1}{a(N,d)} \int_0^{Nt}  \bar{\eta}_s(0) \, \d s.
	\end{align*}  
	We will follow the martingale argument to study the scaling limit of the occupation time. As discussed in Section~\ref{subsec.MartingaleCLT} and Section~\ref{subsec.HomOfReso} we will consider the martingale defined by $\tilde{G}_\lambda$ instead of $G_\lambda$.
	
	In Section~\ref{subsec.SEPH1}, we first collect several estimates about $\bar{G}_\lambda$, which is  the resolvent equation of a SEP. Its analysis relies on the explicit calculation. In Section~\ref{subsec.HomH1}, we summary the estimates about the two-scale expansion $\tilde{G}_\lambda$ from \cite{funaki2024quantitative,gu2025relaxation}, and then further develop  the results about its \emph{carr\'e du champ}. In Section~\ref{subsec.tightness} and  Section~\ref{subsec.Convergence}, we prove the tightness of $\{(\Gamma_t^N)_{t\geq 0}\}_{N\in \N}$ and justify its convergence to a Brownian motion.
	
	\subsection{Resolvent of SEP}\label{subsec.SEPH1}
	We first consider the resolvent equation \eqref{eq.ResolventSEP} of a SEP with constant jump rate whose generator is defined in \eqref{eq.Lb}:
	\begin{equation*}
		(\lambda-\Lb)\bar{G}_\lambda=\bar{\eta}(0).
	\end{equation*}
	Since $\Lb$ is closed in the Fock space, the solution of the above equation lives in the Fock space $\HH_1$ and can be expressed as
	\begin{equation}\label{eq.Gg}
		\bar{G}_\lambda(\eta)=\sum_{x\in \mathbb{Z}^2}\bar{g}_\lambda(x)\bar{\eta}(x),
	\end{equation}
	where the function $\bar{g}_\lambda : \Zt \to \R$ solves the discrete heat equation:
	\begin{equation}\label{eq.DiscreteHeat}
		\left(\lambda-\tfrac{1}{2}\Delta_Q\right)\bar{g}_\lambda=\delta_0.
	\end{equation}
	The discrete Laplacian $\frac{1}{2}\Delta_Q$ can be written explicitly in terms of the jump rate $\{Q_y\}_{y\in\mathbb{Z}^2}$ for the SEP: 
	\begin{align}\label{eq.defDeltaQ}
		\Ll(\tfrac{1}{2}\Delta_Q f\Rr)(x):= \sum_{y\in \mathbb{Z}^2}Q_{y-x}(f(y) - f(x)).
	\end{align}
	
	To study the Green function $\bar{g}_\lambda$, we consider a continuous-time symmetric random walk $(S_t)_{t \geq 0}$ on $\mathbb{Z}^2$ with the transition matrix $Q$. 
	Following the convention in the literature, the covariance matrix $\cm$ of $(S_t)_{t\geq 0}$ has the following relation with the diffusion matrix $\DD(\rho)$:
	\begin{equation*}
		\tfrac{1}{2} \cm = \DD(\rho).
	\end{equation*}
	With the transition probability $\bar{p}_t(x,y)$, we can rewrite the Green function $\bar{g}_\lambda$ as
	\begin{equation}\label{eq.GreenEx}
		\bar{g}_\lambda(x) = \int_{0}^{\infty}e^{-\lambda t}\bar{p}_t(0,x)\, \d t.
	\end{equation}
	
	Several results about $\bar{g}_\lambda$ are summarized in Proposition~\ref{prop.GreenEstimate}. They are the consequences of \eqref{eq.DiscreteHeat} and \eqref{eq.GreenEx}. Besides, one important analytic tool is the local CLT \cite[Theorem 2.1.3]{lawler2010random}: there exists a constant $c$ such that for any $t\geq 1$,
	\begin{equation}\label{eq.LocalCLT}
		\left|\bar{p}_t(0,x) - \tfrac{1}{\sqrt{(2\pi t)^2 \det[\cm]}}e^{-\frac{x\cdot\cm^{-1} x}{2t}}\right| \leq \frac{c}{t^2},\qquad \forall x\in\mathbb{Z}^{2}.
	\end{equation}
	In the statement, we denote by $\D_{h}$ the finite difference operator on $\mathbb{Z}^2$: 
	\begin{align}\label{eq.defD}
		\forall h \in \mathbb{Z}^2, x \in \mathbb{Z}^2, \qquad (\D_h f)(x) := f(x+h) - f(x).
	\end{align}
	
	\begin{proposition}\label{prop.GreenEstimate}
		The following estimates hold for the Green function $\bar{g}_\lambda$.
		\begin{enumerate}
			\item \label{eq.GreenL2}
			The $\ell^2$ norm is of order $\lambda^{-1}$:
			\begin{equation*}
				\norm{\bar{g}_\lambda}^2_{\ell^2(\mathbb{Z}^2)} \sim \lambda^{-1}.
			\end{equation*}
			\item \label{eq.GreenH1}
			The Dirichlet energy is of order $\log(\lambda^{-1})$:
			\begin{equation*}
				\expec{\bar{g}_\lambda, -\tfrac{1}{2}\Delta_Q\bar{g}_\lambda}_{\ell^2(\mathbb{Z}^2)} =  \tfrac{1}{4\pi\sqrt{\det[\DD(\rho)]}}\log(\lambda^{-1}) + O(1).
			\end{equation*}
			\item \label{eq.GreenDBound}
			The finite difference is uniformly bounded:
			\begin{equation*}
				\sup_{\lambda\in(0,1)}\sup_{i=1,2}\norm{\D_{e_i}\bar{g}_\lambda}_{\ell^\infty(\Zt)}\leq C.
			\end{equation*}
		\end{enumerate}
	\end{proposition}
	\begin{proof}
		By \eqref{eq.GreenEx} and the symmetry of the random walk, we can compute
		\begin{equation}\label{eq.gL2}
			\begin{split}
				\norm{\bar{g}_\lambda}^2_{\ell^{2}(\mathbb{Z}^2)}
				& = \sum_{x\in\mathbb{Z}^2}\int_{0}^{\infty}\int_{0}^{\infty}e^{-\lambda(s+t)}\bar{p}_s(0,x)\bar{p}_t(0,x)\, \d s\, \d t\\
				&=\int_{0}^{\infty}\int_{0}^{\infty}e^{-\lambda(s+t)}\bar{p}_{s+t}(0,0)\, \d s\, \d t\\
				&=\int_{0}^{\infty}t e^{-\lambda t}\bar{p}_{t}(0,0)\, \d t.
			\end{split}
		\end{equation}
		We can derive \eqref{eq.GreenL2} combining \eqref{eq.LocalCLT}.
		
		Since $\int_{1}^{\infty}\frac{1}{t}e^{-\lambda t} \d t=\log(\lambda^{-1}) + O(1)$, combining \eqref{eq.GreenEx} and \eqref{eq.LocalCLT} we can obtain
		\begin{equation*}
			\bar{g}_\lambda(0) = \tfrac{1}{4\pi\sqrt{\det[\DD(\rho)]}}\log(\lambda^{-1}) + O(1).
		\end{equation*} 
		Testing \eqref{eq.DiscreteHeat} with $\bar{g}_\lambda$ yields:
		\begin{equation*}
			\expec{\bar{g}_\lambda, -\tfrac{1}{2}\Delta_Q\bar{g}_\lambda}_{\ell^2(\mathbb{Z}^2)} = \bar{g}_\lambda(0) - \lambda \norm{\bar{g}_\lambda}^{2}_{\ell^2(\mathbb{Z}^2)}.
		\end{equation*}
		Combining this with estimate \eqref{eq.GreenL2}, we  obtain \eqref{eq.GreenH1}.
		
		Using mean value theorem and \eqref{eq.LocalCLT} we have for any $t\geq 1$:
		\begin{equation}\label{eq.SemigroupDBound}
			\left|\bar{p}_t(0,x)-\bar{p}_t(0,x+e_i)\right|\leq \tfrac{C}{t^2},\qquad\forall x\in\mathbb{Z}^2,\ i=1,2.
		\end{equation}  
		Since $\int_{1}^{\infty}\frac{1}{t^2}e^{-\lambda t} \d t=1 + o(1)$, combining \eqref{eq.GreenEx} and \eqref{eq.SemigroupDBound} we can obtain \eqref{eq.GreenDBound}.
	\end{proof}
	
	\begin{corollary}\label{Cor.TildeGNorm}
		We have the following estimates for $\bar{G}_\lambda$:
		\begin{itemize}
			\item $\norm{\bar{G}_\lambda}^2_{L^2}=O(\lambda^{-1})$.
			\item $\expec{\bar{G}_\lambda,-\Lb\bar{G}_\lambda}=\frac{\chi(\rho)}{4\pi\sqrt{\det[\DD(\rho)]}}\log(\lambda^{-1})+O(1)$.
			\item $\expec{-\Lb\bar{G}_\lambda,-\Lb\bar{G}_\lambda}=\norm{\Lb \bar{G}_\lambda}^2_{L^2}=O(1)$.
		\end{itemize}
	\end{corollary}
	\begin{proof}
		The first two estimates follow from \eqref{eq.GreenL2} and \eqref{eq.GreenH1} in Proposition~\ref{prop.GreenEstimate} directly. The third one can be derived from \eqref{eq.ResolventSEP}:
		\begin{equation*}
			\norm{\Lb\bar{G}_\lambda}_{L^2}=\norm{\lambda\bar{G}_\lambda-\bar{\eta}(0)}_{L^2}\leq \norm{\lambda\bar{G}_\lambda}_{L^2}+\norm{\bar{\eta}(0)}_{L^2}.
		\end{equation*}
		The calculation of $\norm{\lambda\bar{G}_\lambda}_{L^2}$ follows \eqref{eq.Gg} and \eqref{eq.gL2}
		\begin{align*}
			\norm{\lambda\bar{G}_\lambda}_{L^2}^2 = \chi(\rho)\norm{\lambda \bar{g}_\lambda}^2_{\ell^{2}(\mathbb{Z}^2)} = \chi(\rho)\lambda^2 \int_{0}^{\infty} t e^{-\lambda t}\bar{p}_{t}(0,0)\, \d  t \leq \int_{0}^{\infty} s e^{-s }\, \d s \leq 1.
		\end{align*}
	\end{proof}
	
	\subsection{Two-scale expansion}\label{subsec.HomH1}
    Recall the two-scale expansion \eqref{eq.TwoScaleExpan} yields
	\begin{align*}
		\tilde{G}_\lambda:=\bar{G}_\lambda+\sum_{z\in\Z_m}\sum_{i=1}^{d}(\D_{e_i}\bar{g}_\lambda)_{z+\cu_m}\phi_{m,e_i}^z.
	\end{align*}
	Here the difference operator $\D_{e_i}$ is defined in \eqref{eq.defD}, and $(\D_{e_i} \bar{g}_\lambda)_{z+\cu_m}$ is the local average of $\D_{e_i} \bar{g}_\lambda$ in $z+\cu_m$
	\begin{align}\label{eq.DgLocalAvg}
		(\D_{e_i} \bar{g}_\lambda)_{z+\cu_m} := \frac{1}{\vert \cu_m \vert} \sum_{x \in z + \cu_m } \D_{e_i} \bar{g}_\lambda(x),
	\end{align}
	and $\phi_{m,e_i}^z$ is the local corrector defined in \cite[Section 4.1]{gu2025relaxation}. Roughly, it is an approximation of the minimization in \eqref{eq.defQuadra}. Another related quantity is called  \emph{the centered flux}
	\begin{align}\label{eq.defFlux}
		\g_{m,e_i,b}^z := c_b  \pi_b (\ell_{e_i} + \phi^z_{m, e_i} ) - \pi_b \ell_{\DD(\rho) e_i},
	\end{align}
	where $\ell_{e_i}$ is the affine function defined as:
	\begin{equation}\label{eq.defaffinefuc}
		\ell_{e_i}:=\sum_{x \in \Zd} (e_i \cdot x) \eta(x) .
	\end{equation}
	$\g_{m,e_i, b}^z$ can create spatial cancellation, which plays a similar role as Varadhan's gradient replacement (see \cite[Chapter~7.1]{kipnis1998scaling}). We list the important properties about $\phi_{m,e_i}^z, \g_{m,e_i,b}^z$ below.
	
	\begin{lemma}\label{lem.Homogenization}
		The following estimates hold for the corrector $\phi_{m,e_i}^z$ and the flux $\g_{m,e_i, b}^z$:
		\begin{enumerate}
			\item\label{lem.Correctorlocal} \cite[Proposition~4.3~(1)]{gu2025relaxation}
			The local corrector $\phi^z_{m, e_i}$ is a local function in $\F_0({z+\cu_m^-})$.
			\item\label{lem.CorrectorMean} \cite[Proposition 4.3 (2)]{gu2025relaxation} The local corrector $\phi^z_{m,e_i}$ is centered  
			\begin{align*}
				\E_\rho \Ll[\phi^z_{m,e_i}\Rr] = 0.
			\end{align*}
			\item\label{lem.CorrectorH1L2} \cite[Proposition 4.3 (3)]{gu2025relaxation} The local corrector $\phi^z_{m,e_i}$ satisfies the following $H^1$ and $L^2$ estimates: for all ${\rho \in (0,1)}, m \in \N$ and $z \in \Z_m$, we have 
			\begin{align*}
				\sum_{b\in\ov{(z+\cu_m)^*}}\E_{\rho}\left[c_b(\eta)\left(\pi_b\phi^z_{m,e_i}\right)^2\right] &\leq 16\, (\tfrac{c_+}{c_-})\chi(\rho)3^{dm},\\
				\E_\rho\Ll[\Ll(\phi^z_{m,e_i}\Rr)^2 \Rr] &\leq 16\, (\tfrac{c_{+}}{c_{-}})\, \chi(\rho) 3^{(d+2)m}.
			\end{align*}
			\item\label{lem.FluxReplacement}\cite[Proposition 4.3 (4)]{gu2025relaxation}
			There exists an exponent $\alpha(d,c_{-},c_{+},\r) \in \Ll(0,\frac{1}{2}\Rr)$ and a positive constant $C(d,c_{-},c_{+},\r) < \infty$, such that for every $v : \X \to \R$, we have 
			\begin{align*}
				\Ll\vert \frac{1}{\vert \cu_m \vert} \sum_{b \in \ov{(z+\cu_m)^*} }\bracket{\pi_b v, \g_{m,e_i,b}^z} \Rr\vert \leq C 3^{-\alpha m}\Ll(\frac{1}{\vert \cu_m \vert} \sum_{b \in \ov{(z+\cu_m)^*} }\E_\rho\Ll[(\pi_b v)^2\Rr]\Rr)^{\frac{1}{2}}.
			\end{align*}
			\item\label{lem.CorrectorLp} \cite[Lemma 4.4]{funaki2024quantitative}
			There exists a constant $C(d,c_{-},c_{+})$ such that the $L^\infty$ and $L^{p}$ norm for the corrector $\phi_{m,e_i}^z$ satisfy the following estimates:
			\begin{align*}
				&\norm{\phi_{m,e_i}^z}_{L^\infty}\leq Cm3^{(d+2)m},\qquad\forall m\in\N,\  z\in \Z_m,\\
				&\norm{\phi_{m,e_i}^z}_{L^p}\leq C(p)3^{(d+2)m},\qquad \forall p\in[1,\infty),\quad \forall m\in\N,\  z\in \Z_m.
			\end{align*}
		\end{enumerate}
	\end{lemma}

	Based on Lemma~\ref{lem.Homogenization} above, the properties about the two-scale expansion $\tilde{G}_\lambda$ are established. Especially, the heuristic $\tilde{G}_\lambda \simeq \bar{G}_\lambda \simeq G_\lambda$ is valid. We cite the useful results developed in \cite[Section 4.2]{gu2025relaxation}. We highlight that, both (4) of Lemma~\ref{lem.Homogenization} and (2) of Proposition~\ref{prop.L2FluxErr} aim to realize the replacement argument.
	
	\begin{proposition}\label{prop.L2FluxErr}
		We have the following estimates for the two-scale expansion $\tilde{G}_\lambda$:
		\begin{enumerate}
			\item\label{prop.L2Err} \cite[Lemma 4.5]{gu2025relaxation} There exists a finite positive constant $C(d, c_{-},c_{+})$ such that 
			\begin{equation*}
				\norm{\tilde{G}_\lambda-\bar{G}_\lambda}^2_{L^2}\leq C3^{2m}\expec{\bar{G}_\lambda,-\Lb \bar{G}_\lambda}.
			\end{equation*}
			\item\label{prop.FluxErr} \cite[Lemma 4.7]{gu2025relaxation} There exists an exponent $\alpha(d,c_{-},c_{+}, \r) > 0$ and a finite positive constant $C(d,c_{-},c_{+}, \r,\rho)$ such that the following estimate holds for any $V\in L^2$:
			\begin{equation*}
				\left|\expec{V,-\L\tilde{G}_\lambda+\Lb \bar{G}_\lambda}\right|\leq C\expec{V,-\Lb V}^{\frac{1}{2}}\left(3^{-\alpha m}\expec{\bar{G}_\lambda,-\Lb \bar{G}_\lambda}^{\frac{1}{2}}+3^m \expec{-\Lb \bar{G}_\lambda,-\Lb \bar{G}_\lambda}^{\frac{1}{2}}\right).
			\end{equation*}
		\end{enumerate}
	\end{proposition}

	In the next step, we further develop the following proposition about $\sum_{b\in(\Zt)^*}\left(c_b(\pi_b\tilde{G}_\lambda)^2\right)$, which is the carr\'e du champ of $\tilde{G}_\lambda$ in dimension $d=2$. It will allow us to understand the quadratic variation of the martingale in \eqref{eq.defTildeMt}.
	
	\begin{proposition}[Homogenization of carr\'e du champ]\label{prop.H1TwoScale}
		There exists a finite positive constant $C(c_{-},c_{+},\r,\rho)$ such that for every $\lambda\in(0,\tfrac{1}{2})$, the following estimates about the expectation and variance of the carr\'e du champ of $\tilde{G}_\lambda$ hold:
		\begin{multline}\label{eq.H1TwoScaleExpect}
			\Ll|\E_\rho\Ll[\sum_{b\in(\Zt)^*}\left(c_b(\pi_b\tilde{G}_\lambda)^2\right)\left(\eta\right)\Rr]-2\expec{\bar{G}_\lambda,-\Lb\bar{G}_\lambda}\Rr|\\
			\leq C\sqrt{\log (\lambda^{-1})}(3^{-\alpha m}\sqrt{\log(\lambda^{-1})}+3^m),
		\end{multline}
		\begin{equation}\label{eq.H1TwoScaleVar}
			\var_\rho\left[\sum_{b\in(\Zt)^*}\left(c_b(\pi_b\tilde{G}_\lambda)^2\right)(\eta)\right]\leq C 3^{18m}\log(\lambda^{-1}).
		\end{equation} 
	\end{proposition}
	\begin{proof}
		\textit{Step~1: the expectation.}
		We first prove estimate \eqref{eq.H1TwoScaleExpect}. Using
		\begin{equation*}
			\E_\rho\Ll[\sum_{b\in(\Zt)^*}\left(c_b(\pi_b\tilde{G}_\lambda)^2\right)\left(\eta\right)\Rr]=2\expec{\tilde{G}_\lambda,-\L\tilde{G}_\lambda},
		\end{equation*} 
		we can decompose the left-hand side of \eqref{eq.H1TwoScaleExpect} into two terms:
		\begin{multline}\label{eq.DecomFluxL2}
			\Ll|\E_\rho\Ll[\sum_{b\in(\Zt)^*}\left(c_b(\pi_b\tilde{G}_\lambda)^2\right)\left(\eta\right)\Rr]-2\expec{\bar{G}_\lambda, -\Lb\bar{G}_\lambda}\Rr|\\
			\leq 2\Ll|\expec{\tilde{G}_\lambda, -\L\tilde{G}_\lambda+\Lb\bar{G}_\lambda}\Rr| + 2\Ll|\expec{\tilde{G}_\lambda-\bar{G}_\lambda, -\Lb\bar{G}_\lambda}\Rr|.
		\end{multline}
		We then treat the two terms on the {\rhs} respectively.
	
		For the first term, we can use \eqref{prop.FluxErr} in Proposition~\ref{prop.L2FluxErr} to obtain:
		\begin{multline}\label{eq.ExpectFlux0}
			\Ll|\expec{\tilde{G}_\lambda, -\L\tilde{G}_\lambda+\Lb\bar{G}_\lambda}\Rr|\\ 
			\leq C \expec{\tilde{G}_\lambda,-\Lb\tilde{G}_\lambda}^{\frac{1}{2}}(3^{-\alpha m}\expec{\bar{G}_\lambda,-\Lb\bar{G}_\lambda}^{\frac{1}{2}}+3^m\expec{-\Lb\bar{G}_\lambda,-\Lb\bar{G}_\lambda}^{\frac{1}{2}}).
		\end{multline}
		We can control $\expec{\tilde{G}_\lambda,-\Lb\tilde{G}_\lambda}$ by $\expec{\bar{G}_\lambda,-\Lb\bar{G}_\lambda}$. A detailed computation yields 
		\begin{equation*}
			\begin{aligned}
				\expec{\tilde{G}_\lambda-\bar{G}_\lambda,-\Lb\left(\tilde{G}_\lambda-\bar{G}_\lambda\right)}
				&\leq C\expec{\tilde{G}_\lambda-\bar{G}_\lambda, -\L\left(\tilde{G}_\lambda-\bar{G}_\lambda\right)}\\
				&=\frac{C}{2}\sum_{z\in\Z_m}\sum_{b\in\ov{(z+\cu_m)^*}}\sum_{i=1}^{2}\left(\D_{e_i}\bar{g}_\lambda\right)^2_{z+\cu_m}\E_{\rho}\left[c_b(\eta)\left(\pi_b\phi^z_{m,e_i}\right)^2\right]\\
				&\leq C\sum_{x\in\Zt}\sum_{i=1}^{2}\frac{1}{|\cu_m|}\left(\D_{e_i}\bar{g}_\lambda(x)\right)^2\cdot 3^{2m}\\
				&\leq C\expec{\bar{G}_\lambda,-\Lb\bar{G}_\lambda}.
			\end{aligned}
		\end{equation*}
		Here in the first line, we use the uniform ellipticity in (1) of Hypothesis~\ref{hyp}. In the second line, we use the expression \eqref{eq.TwoScaleExpan} and the locality of the corrector; see \eqref{lem.Correctorlocal} in Lemma~\ref{lem.Homogenization}. In the third line, we use Cauchy--Schwarz inequality and \eqref{lem.CorrectorH1L2} in Lemma~\ref{lem.Homogenization}. The last line follows the explicit calculation in \eqref{eq.Gg}. 	Therefore, by the triangular inequality we have
		\begin{equation}\label{eq.TildeGH1}
			\expec{\tilde{G}_\lambda,-\Lb\tilde{G}_\lambda}^{\frac{1}{2}}\leq \expec{\tilde{G}_\lambda-\bar{G}_\lambda,-\Lb\left(\tilde{G}_\lambda-\bar{G}_\lambda\right)}^{\frac{1}{2}}+\expec{\bar{G}_\lambda,-\Lb\bar{G}_\lambda}^{\frac{1}{2}}\leq C\expec{\bar{G}_\lambda,-\Lb\bar{G}_\lambda}^{\frac{1}{2}}.
		\end{equation}
		Combining \eqref{eq.ExpectFlux0}, \eqref{eq.TildeGH1}, and Corollary~\ref{Cor.TildeGNorm}, we obtain an upper bound for the first term in \eqref{eq.DecomFluxL2}:
		\begin{equation}\label{eq.ExpectFlux}
			\Ll|\expec{\tilde{G}_\lambda, -\L\tilde{G}_\lambda+\Lb\bar{G}_\lambda}\Rr| \leq C\sqrt{\log(\lambda^{-1})}(3^{-\alpha m}\sqrt{\log(\lambda^{-1})}+3^m).
		\end{equation}
		
		For the second term in \eqref{eq.DecomFluxL2}, combining \eqref{prop.L2Err} in Proposition~\ref{prop.L2FluxErr} and the estimate for $\expec{-\Lb\bar{G}_\lambda,-\Lb\bar{G}_\lambda}$ in Corollary~\ref{Cor.TildeGNorm}, we have
		\begin{equation}\label{eq.ExpectL2}
			\Ll|\expec{\tilde{G}_\lambda-\bar{G}_\lambda, -\Lb\bar{G}_\lambda}\Rr| \leq \norm{\tilde{G}_\lambda-\bar{G}_\lambda}_{L^2}\norm{\Lb\bar{G}_\lambda}_{L^2} \leq C 3^{m}\sqrt{\log(\lambda^{-1})}.
		\end{equation}
		We can obtain the desired estimate \eqref{eq.H1TwoScaleExpect} combining \eqref{eq.DecomFluxL2}, \eqref{eq.ExpectFlux}, and \eqref{eq.ExpectL2}.
		
		\smallskip
		\textit{Step~2: the variance.} We next prove the variance estimate \eqref{eq.H1TwoScaleVar}. We can expand $\sum_{b\in(\Zt)^*}\left(c_b(\pi_b\tilde{G}_\lambda)^2\right)\left(\eta\right)$ as:
		\begin{multline}\label{eq.H1TwoScaleCom}
			\sum_{b\in(\Zt)^*}\left(c_b(\pi_b\tilde{G}_\lambda)^2\right)\left(\eta\right) \\
			= \sum_{b\in(\Zt)^*}c_b(\eta)\Ll(\pi_b\bar{G}_\lambda(\eta)+\pi_b\left(\sum_{z\in\Z_m}\sum_{i=1}^{2}(\D_{e_i}\bar{g}_\lambda)_{z+\cu_m}\phi^z_{e_i,m}\right)\Rr)^2.
		\end{multline}
		For a bond $b=\{x,y\}$, we can compute the first term explicitly:
		\begin{equation*}
			\pi_b \bar{G}_\lambda(\eta) = (\bar{g}_\lambda(x)-\bar{g}_\lambda(y))(\bar{\eta}(y)- \bar{\eta}(x)),
		\end{equation*}
		and since the corrector is a local function (see \eqref{lem.Correctorlocal} in Lemma~\ref{lem.Homogenization}), we can simplify the second term:
		\begin{equation*}
			\pi_b\left(\sum_{z\in\Z_m}\sum_{i=1}^{2}(\D_{e_i}\bar{g}_\lambda)_{z+\cu_m}\phi^z_{e_i,m}\right)=\sum_{z\in\Z_m}\1_{\Ll\{b\in \ov{(z+\cu_m)^*}\Rr\}}\sum_{i=1}^2 (\D_{e_i}\bar{g}_\lambda)_{z+\cu_m}\pi_b\phi^z_{e_i,m}.
		\end{equation*}
		Combining the above two observations, we can rewrite \eqref{eq.H1TwoScaleCom} as:
		\begin{equation}\label{eq.CubeDecom}
			\sum_{b\in(\Zt)^*}\left(c_b(\pi_b\tilde{G}_\lambda)^2\right)\left(\eta\right)
			= \sum_{z\in\Z_m}X^z_{m}(\eta),
		\end{equation}
		where the random variable $X_{m}^z(\eta)$ equals
		\begin{multline*}
			\sum_{j=1}^2\sum_{x\in z+\cu_m}c_{x,x+e_j}(\eta)\Ll(\D_{e_j}\bar{g}_\lambda(x)(\bar{\eta}(x)-\bar{\eta}(x+e_j))+\sum_{i=1}^{2}(\D_{e_i}\bar{g}_\lambda)_{z+\cu_m}\pi_{x,x+e_j}\phi^z_{e_i,m}\Rr)^2.
		\end{multline*}
		Combining the local assumption for the jump rate $\{c_b\}_{b\in(\Zt)^*}$ in Hypothesis~\ref{hyp} and the locality of the corrector in \eqref{lem.Correctorlocal} of Lemma~\ref{lem.Homogenization}, we can show the random variable $X^z_m$ is also local:
		\begin{equation*}
			X^z_m\in\F_0(N_\r(z+\cu_m^+)).
		\end{equation*}
		This means that the random variables $X^{z_1}_m$ and $X^{z_2}_m$ are independent, when $m$ satisfies $3^m>2\r+1$ and the cubes $z_1+\cu_m, z_2+\cu_m$ are not adjacent. Therefore, we obtain that
		\begin{equation}\label{eq.CubeDecomVar}
			\var_\rho\Ll[\sum_{b\in(\Zt)^*}\left(c_b(\pi_b\tilde{G}_\lambda)^2\right)\left(\eta\right)\Rr]\leq 5\sum_{z\in\Z_m}\var_\rho[X^z_m].
		\end{equation} 
		Here the AM--GM inequality was utilized to handle the adjacent cubes in \eqref{eq.CubeDecom}.
		
		Let us just calculate one term $\var_\rho[X^z_m]$ for $z\in\Z_m$. Using the ellipticity of the jump rate $c_b(\eta)$ and Cauchy--Schwarz inequality we have
		\begin{multline*}
			\var_\rho[X^z_m]
			\leq \E_\rho\Ll[\Ll(X^z_m\Rr)^2\Rr]\\
			\leq 16c_{+}^2|\cu_m|\sum_{j=1}^2\sum_{x\in z+\cu_m}\Bigg(\E_\rho\Ll[\Ll(\D_{e_j}\bar{g}_\lambda(x)(\bar{\eta}(x)-\bar{\eta}(x+e_j))\Rr)^4\Rr]\\
			\hspace{7.5cm}+\E_\rho\Ll[\Ll(\sum_{i=1}^{2}(\D_{e_i}\bar{g}_\lambda)_{z+\cu_m}\pi_{x,x+e_j}\phi^z_{e_i,m}\Rr)^4\Rr]\Bigg).
		\end{multline*}
		The first term is easy to compute:
		\begin{equation*}
			\E_\rho\Ll[\Ll(\D_{e_j}\bar{g}_\lambda(x)(\bar{\eta}(x)-\bar{\eta}(x+e_j))\Rr)^4\Rr]
			=2\chi(\rho)\Ll(\D_{e_j}\bar{g}_\lambda(x)\Rr)^4
		\end{equation*}
		For the second term, we use Cauchy--Schwarz inequality to simplify it:
		\begin{align*}
			\E_\rho\Ll[\Ll(\sum_{i=1}^{2}(\D_{e_i}\bar{g}_\lambda)_{z+\cu_m}\pi_{x,x+e_j}\phi^z_{e_i,m}\Rr)^4\Rr]
			&\leq 8\sum_{i=1}^{2}\Ll(\D_{e_i}\bar{g}_\lambda\Rr)^4_{z+\cu_m}\E_{\rho}\Ll[(\pi_{x,x+e_j}\phi^z_{e_i,m})^4\Rr]\\
			&\leq C3^{14m}\sum_{i=1}^2\sum_{x\in z+\cu_m}(\D_{e_i}\bar{g}_\lambda(x))^4,
		\end{align*}
		where in the second line we use the $L^4$ estimate for the corrector; see \eqref{lem.CorrectorLp} in Lemma~\ref{lem.Homogenization}.	Combining the above two estimates, we conclude that
		\begin{equation}\label{eq.OneBlockL4}
			\var_\rho[X^z_m]\leq C3^{18 m}\sum_{i=1}^{2}\sum_{x\in z+\cu_m}(\D_{e_i}\bar{g}_\lambda(x))^4.
		\end{equation}
	
		We combine \eqref{eq.OneBlockL4} and \eqref{eq.CubeDecomVar} to obtain the desired result \eqref{eq.H1TwoScaleVar}.
		\begin{align*}
			\var_\rho\Ll[\sum_{b\in(\Zt)^*}\left(c_b(\pi_b\tilde{G}_\lambda)^2\right)\left(\eta\right)\Rr] &\leq C3^{18 m}\sum_{x\in\Zt}\sum_{i=1}^2(\D_{e_i}\bar{g}_\lambda(x))^4 \\
			& \leq C3^{18 m}\sum_{x\in\Zt}\sum_{i=1}^{2}(\D_{e_i}\bar{g}_\lambda(x))^2 \\
			& \leq C 3^{18 m}\log(\lambda^{-1}).
		\end{align*}
		The passage from the first line to the second line follows \eqref{eq.GreenDBound} in Proposition~\ref{prop.GreenEstimate}.  The last line follows Corollary~\ref{Cor.TildeGNorm} and \eqref{eq.Gg}.
	\end{proof}
	
	We finish this subsection with a uniform bound for $\pi_b\tilde{G}_\lambda$.
	\begin{corollary}\label{Cor.Jump}
		There exists a constant $C(c_{-},c_{+})$ such that, for any bond $b\in(\Zt)^*$ and $\lambda\in(0,1)$,  we have the following uniform bound for $\pi_b\tilde{G}_\lambda$:
		\begin{equation*}
			\norm{\pi_b\tilde{G}_\lambda}_{L^\infty}\leq C m 3^{4m},\qquad\forall b\in(\Zt)^*,\ \lambda\in(0,1),\quad\forall m\in\N.
		\end{equation*}
		\begin{proof}
			Similar to the proof for \eqref{eq.H1TwoScaleVar} in Proposition~\ref{prop.H1TwoScale}, we can compute $\pi_{b}\tilde{G}_\lambda(\eta)$ exactly using the expression for $\tilde{G}_\lambda$ and the locality of the corrector Lemma~\ref{lem.Homogenization} \eqref{lem.Correctorlocal}. Suppose that the bond $b= \{x,x+e_i\}$ lies in $\ov{(z+\cu_m)^*}$ for $z\in\Z_m$, then we have
			\begin{equation*}
				\pi_b\tilde{G}_\lambda(\eta) = \D_{e_i}\bar{g}_\lambda(x)(\bar{\eta}(x)-\bar{\eta}(x+e_i))+\sum_{j=1}^{2}(\D_{e_j}\bar{g}_\lambda)_{z+\cu_m}\pi_{x,x+e_i}\phi^z_{e_j,m}.
			\end{equation*}
			Combining \eqref{eq.GreenDBound} in Proposition~\ref{prop.GreenEstimate} and \eqref{lem.CorrectorLp} in Lemma~\ref{lem.Homogenization}, we can obtain the desired result.
		\end{proof}
	\end{corollary}

	\subsection{Martingale CLT of $\tilde{M}^{N}$}\label{subsec.MTildeCLT}
	The estimates in the last two subsections allow us to finish the martingale CLT of $\tilde{M}^{N}$.
	\begin{proof}[Proof of Proposition~\ref{prop.CLT_M}]
		Based on the central limit theorem for martingales (\cite[Lemma 4.1]{QJS02} and see also \cite{RolandoCentral} for the details), we need to verify the following two conditions: 
		\begin{itemize}
			\item [(i)] $\bracket{\tilde{M}^N}_t \xrightarrow{N \to \infty} \sigma^2 t$ in probability with $\sigma^2=\tfrac{\chi(\rho)}{2\pi\sqrt{\det[\DD(\rho)]}}$.
			\item [(ii)] For every $\varepsilon>0$, we have $\sigma_{\varepsilon}(\tilde{M}^N)(t) \xrightarrow{N \to \infty} 0$ in probability. Here the quantity $\sigma_{\varepsilon}(\tilde{M}^N)(t)$ counts the accumulated jumps larger than $\epsilon$
			\begin{equation*}
				\sigma_{\varepsilon}(\tilde{M}^N)(t) :=\sum_{0\leq \tau_i\leq t}\vert \Delta \tilde{M}^N(\tau_i)\vert^2 \1_{\{\vert \Delta \tilde{M}^N(\tau_i)\vert\geq \varepsilon\}},
			\end{equation*}
			where $(\tau_i)_{i \in \N_+}$ are the jump moments and $\Delta \tilde{M}^N(s) := \tilde{M}^N(s) - \tilde{M}^N(s-)$.
		\end{itemize}
		We also recall the setting 
		\begin{align}\label{eq.scaleM.RC}
			\lambda = \frac{1}{N}, \qquad 1 \ll 3^{18m} \ll \log N.
		\end{align}
		
		Let us check at first the condition (ii). By the graphical construction of the process, we know almost surely there is at most one jump at every jump moment. Using Corollary~\ref{Cor.Jump} and the chosen scale \eqref{eq.scaleM.RC}, we have
		\begin{equation*}
			\sup_{0\leq \tau_i\leq t}\left|\Delta\tilde{M}^N(\tau_i)\right|\leq \frac{1}{\sqrt{N\log N}} \sup_{b\in(\Zt)^*}\norm{\pi_b\tilde{G}_{1/N}}_{L^\infty}\leq \frac{C m 3^{4m}}{\sqrt{N\log N}}\xrightarrow{N \to \infty} 0.
		\end{equation*}
		Therefore, for every fixed $\varepsilon>0$, the jump will be smaller than $\epsilon$ and $\sigma_{\varepsilon}(\tilde{M}^N)(t) = 0$ for $N$ large enough. This justifies (ii).
		
		We then turn to the condition (i), and its proof relies on its first and second moments. The quadratic variation of $\tilde{M}^{N}$ is given by
		\begin{align*}
			\bracket{\tilde{M}^N}_t = \frac{1}{N\log N} \int_0^{Nt} \sum_{b \in (\Zt)^*} \Ll(c_b (\pi_b \tilde{G}_{1/N})^2\Rr)(\eta_s) \, \d s.
		\end{align*} 
		Combining \eqref{eq.H1TwoScaleExpect} in Proposition~\ref{prop.H1TwoScale}, Corollary~\ref{Cor.TildeGNorm}, and the chosen scale \eqref{eq.scaleM.RC}, we can obtain
		\begin{equation}\label{eq.TildeMEx}
			\begin{aligned}
				\E_{\rho}\left[\bracket{\tilde{M}^N}_t\right] &=\frac{1}{N\log N}\int_{0}^{Nt}\E_\rho\left[\sum_{b \in (\Zt)^*} \Ll(c_b (\pi_b \tilde{G}_{1/N})^2\Rr)(\eta)\right]\, \d s\\
				&= \frac{N t}{N\log N} \Ll(2\expec{\bar{G}_{1/N},-\Lb\bar{G}_{1/N}} + O(3^{-\alpha m} \log N +3^m \sqrt{\log N})\Rr)\\
				&=\frac{\chi(\rho)t}{2\pi\sqrt{\det[\DD(\rho)]}}+o(1).
			\end{aligned}
		\end{equation}
		Similarly, combining \eqref{eq.H1TwoScaleVar} in Proposition~\ref{prop.H1TwoScale}, Corollary~\ref{Cor.TildeGNorm}, and the chosen scale \eqref{eq.scaleM.RC}, by Cauchy--Schwarz inequality, we can obtain
		\begin{equation}\label{eq.TildeMVar}
			\begin{split}
				\var_\rho\left[\bracket{\tilde{M}^N}_t\right] &\leq \frac{N t}{(N\log N)^2}\int_{0}^{Nt}\var_\rho\left[\sum_{b \in (\Zt)^*} \Ll(c_b (\pi_b \tilde{G}_{1/N})^2\Rr)(\eta)\right]\, \d s\\
				&\leq \frac{t^2}{(\log N)^2} C3^{18m}\log N \\
				&\leq \frac{C3^{18m}t^2}{\log N}=o(1).
			\end{split}
		\end{equation}
		The equations \eqref{eq.TildeMEx} and \eqref{eq.TildeMVar} yield the condition (i) with $\sigma^2=\tfrac{\chi(\rho)}{2\pi\sqrt{\det[\DD(\rho)]}}$.
	\end{proof}
	
	\subsection{Tightness of $\Gamma^N$}\label{subsec.tightness}
	In this section, we prove the tightness of the occupation time. We use the shorthand notation $\Gamma^N \equiv (\Gamma^N_t)_{t \geq 0}$ for the process. 
	\begin{proposition}\label{prop.tightGamma}
		For $d = 2$, the occupation time $\{\Gamma^N\}_{N \in \N}$ defined in \eqref{eq.Occupation} is tight in $C(\R_+, \R)$.
	\end{proposition}
	
	The proof will make use of a lemma about the flow, so let us recall its definition. Given $p,p^\prime$ two probability measures on $\mathbb{Z}^d$, we call $\phi: \mathbb{Z}^d \times \mathbb{Z}^d \rightarrow \R$ a flow connecting $p$ to $p^\prime$ if it satisfies the following conditions:
	\begin{itemize}
		\item[(i)] $\phi$ is anti-symmetric, $\phi(x,y) = - \phi (y,x)$ for any $x,y \in \mathbb{Z}^d$;
		\item[(ii)] $\phi (x,y) = 0$ unless $y \sim x$;
		\item[(iii)] for any $x \in \mathbb{Z}^d$, $\sum_{y} \phi (x,y) =  p^\prime (x) - p(x)$.
	\end{itemize}
	For any $\ell > 0$, let $p_\ell$ be the uniform measure on the cube $\Lambda_\ell = (- \frac{\ell}{2}, \frac{\ell}{2}) \cap \Zd$ with $\ell \in \N_+$. The following lemma was proved by Jara and Menezes \cite{jarmen18}.
	
	\begin{lemma}[Flow lemma]
		There exists a flow $\phi_\ell$ connecting $\delta_0$ to $p_\ell$ satisfying the following conditions.
		\begin{itemize}\label{lem flow}
			\item $\phi_\ell (x,y) = 0$ unless $x,y \in \Lambda_\ell$;
			\item There exists a finite positive constant $C_0 (d)$ such that
			\[\sum_{x \sim y} \phi_\ell (x,y)^2 \leq C_0 g_d (\ell),\]
			where
			\[g_d (\ell) = \begin{cases}
				\ell, \quad&\text{if } \;d=1,\\
				\log \ell, \quad&\text{if } \;d=2,\\
				1, \quad&\text{if } \;d\geq 3.
			\end{cases}\]
		\end{itemize} 
	\end{lemma}

	\begin{proof}[Proof of Proposition~\ref{prop.tightGamma}]
		For every integer $\ell \in \N_+$ and any configuration $\eta$, define
		\[\bar{\eta}^\ell (0) := \frac{1}{|\Lambda_\ell|}\sum_{x \in \Lambda_\ell} \bar{\eta} (x).\]
		We then decompose the occupation time into two parts 
		\begin{equation}\label{eq.occupation_decom}
			\begin{split}
				\Gamma^N_t &= \Gamma^{N,1}_t + \Gamma^{N,2}_t, \\
				\Gamma^{N,1}_t &:= \frac{1}{a(N,2)} \int_0^{Nt}  \bar{\eta}^\ell_s (0) \, \d s, \\
				\Gamma^{N,2}_t &:= \frac{1}{a(N,2)} \int_0^{Nt}  (\bar{\eta}_s(0) - \bar{\eta}^\ell_s (0)) \, \d s.
			\end{split}
		\end{equation}
		We aim to show the tightness for the two parts with the choice $\ell = \lfloor \sqrt{N / \log N}\rfloor$.
		
		Concerning $\Gamma^{N,1}$, the Cauchy--Schwarz inequality applies to every $0<s<t<\infty$,
		\[\E_\rho \Ll[ \Ll(\Gamma^{N,1}_t - \Gamma^{N,1}_s \Rr)^2\Rr] \leq \frac{(t-s)^2N}{\log N} \E_\rho\left[\left(\bar{\eta}^\ell (0)\right)^2\right] \leq \frac{C(t-s)^2N}{\ell^2\log N} \leq C(t-s)^2.\]
		Thus, by Kolmogorov--Chentsov criterion, the part $\{\Gamma^{N,1}\}_{N \in \N}$ is tight in $C(\R_+, \R)$.

		It remains to prove the tightness of $\{\Gamma^{N,2}\}_{N \in \N}$.  By the Feynman--Kac formula, for any $\alpha > 0$ and any $s < t$,
		\begin{equation}\label{eq.FK}
			\begin{split}
				&\log \E_\rho \Ll[ \exp \Ll(\alpha (\Gamma^{N,2}_t - \Gamma^{N,2}_s)\Rr) \Rr] \\
				&= \log \E_\rho \Ll[ \exp \Ll\{ \frac{\alpha}{a(N,2)}\int_{sN}^{t N} \left(\bar{\eta}_r (0) - \bar{\eta}_r^\ell (0)\right)\, \d r \Rr\} \Rr] \\
				&\leq (t-s)N \sup_{f \geq 0, \E_\rho[f] = 1} \Ll\{ \frac{\alpha}{a(N,2)}\expec{\bar{\eta} (0) - \bar{\eta}^\ell (0),f} - \expec{\sqrt{f},-\L\sqrt{f}}\Rr\}.
			\end{split}
		\end{equation} 
		We aim to estimate the last line of \eqref{eq.FK}. 
		
		Let $\phi_\ell (x,y)$ be the flow connecting $\delta_0$ to $p_\ell$ in Lemma \ref{lem flow}. We write
		\[    \bar{\eta}^\ell (0) - \bar{\eta} (0) = \sum_{x\sim y} \phi_\ell (x,y) \eta(x) = \frac{1}{2} \sum_{x \sim y} \phi_\ell (x,y) \left(\eta(x) - \eta (y)\right),\]
		where we used the antisymmetry of $\phi_\ell$. By marking the change of variables $\eta \mapsto \eta^{x,y}$ 
		\begin{align*}
			\frac{\alpha}{a(N,2)}\expec{\bar{\eta} (0) - \bar{\eta}^\ell (0),f} & = \Er\Ll[ \frac{\alpha}{4a(N,2)}\sum_{x \sim y} \phi_\ell (x,y) \left(\eta(x) - \eta (y)\right) \pi_{x,y} f (\eta) \Rr]\\
			& = \frac{\alpha}{4a(N,2)}\sum_{x \sim y} \phi_\ell (x,y)\Er\Ll[  \left(\eta(x) - \eta (y)\right) \pi_{x,y} f (\eta) \Rr].
		\end{align*}
		Here we use the fact that $\phi_\ell$ is a deterministic function. We notice the trick
		\begin{align*}
			f(\eta^{x,y}) - f(\eta) = \Ll(\sqrt{f(\eta^{x,y})} - \sqrt{f(\eta)}\Rr)\Ll(\sqrt{f(\eta^{x,y})} + \sqrt{f(\eta)}\Rr),
		\end{align*}
		and then apply the Cauchy--Schwarz inequality to the term $\pi_{x,y} f (\eta)$ 
		\begin{align*}
			&\Ll|\Er\Ll[  \left(\eta(x) - \eta (y)\right) \pi_{x,y} f (\eta) \Rr]\Rr| \\
			&\leq 2\Er\Ll[ c_{x,y}(\eta) \Ll(\sqrt{f(\eta^{x,y})} - \sqrt{f(\eta)}\Rr)^2 \Rr]^{\frac{1}{2}}\Er\Ll[ c^{-1}_{x,y}(\eta) \Ll(\sqrt{f(\eta^{x,y})} + \sqrt{f(\eta)}\Rr)^2\Rr]^{\frac{1}{2}}\\
			&\leq 8(c_{-})^{-\frac{1}{2}} \Er\Ll[ c_{x,y}(\eta) \Ll(\pi_{x,y}\sqrt{f}\Rr)^2(\eta) \Rr]^{\frac{1}{2}} \Er[f]^{\frac{1}{2}}.
		\end{align*}
		Recall the condition $\Er[f]=1$. Then we obtain the estimate for $\frac{\alpha}{a(N,2)}\expec{\bar{\eta} (0) - \bar{\eta}^\ell (0),f}$
		\begin{align*}
			\frac{\alpha}{a(N,2)}\expec{\bar{\eta} (0) - \bar{\eta}^\ell (0),f} &\leq \frac{C\alpha}{4a(N,2)}\sum_{x \sim y} \vert \phi_\ell (x,y) \vert \Er\Ll[ c_{x,y}(\eta) \Ll(\pi_{x,y}\sqrt{f}\Rr)^2(\eta) \Rr]^{\frac{1}{2}}\\
			&\leq \sum_{x \sim y} \Ll(\frac{C \alpha^2}{a (N,2)^2} \phi_\ell (x,y)^2 + \Er\Ll[ c_{x,y}(\eta) \Ll(\pi_{x,y}\sqrt{f}\Rr)^2(\eta) \Rr]\Rr)\\
			&\leq \frac{C \alpha^2}{a (N,2)^2} g_2 (\ell) + \expec{\sqrt{f},-\L\sqrt{f}}.
		\end{align*}
		We insert this estimate, the result in Lemma~\ref{lem flow}, and the choice $\ell = \lfloor \sqrt{N / \log N}\rfloor$ into the last line of \eqref{eq.FK}, then we obtain that
		\[\log \E_\rho \Ll[ \exp \Ll(\alpha (\Gamma^{N,2}_t - \Gamma^{N,2}_s)\Rr) \Rr]\leq C \alpha^2 (t-s).\]
		By using Garsia--Rodemich--Rumsey inequality and following  \cite[Proof of Lemma 4.2]{QJS02} line by line, we can show that for any $\delta > 0$,
		\[\E_\rho \Ll[ \sup_{|s -t|\leq T \atop |s-t| \leq \delta} \Ll| \Gamma^{N,2}_t - \Gamma^{N,2}_s \Rr| \Rr] \leq C (T) \sqrt{\delta} (1 + |\log \delta|).\]
		This concludes the proof of tightness.
	\end{proof}

	\subsection{Proof of Theorem~\ref{thm.main_basic}}\label{subsec.Convergence}
	
	Recall that we decompose $\Gamma_t^N$ in \eqref{eq.Gamma_decomTilde} as:
	\begin{equation*}
		\Gamma^N_t=\tilde{M}^{N}_t+R_t^N,
	\end{equation*}
	with the remainder part $R^N_t$ given by
	\begin{equation*}
		R^N_t = \frac{1}{\sqrt{N\log N}} \Ll( \tilde{G}_\lambda(\eta_0) - \tilde{G}_\lambda(\eta_{Nt}) + \int_0^{Nt}  \lambda \bar{G}_\lambda(\eta_s) \, \d s + \int_0^{Nt} (\L\tilde{G}_\lambda-\Lb\bar{G}_\lambda)(\eta_s) \, \d s \Rr).
	\end{equation*}
	We now give the description of the remainder.

	\begin{proposition}\label{prop.RemainderConverge}
		Under the same setting as Proposition~\ref{prop.CLT_M}, the remainder $\{R^N\}_{N \in \N}$ converges weakly to the zero process.
	\end{proposition}
	\begin{proof}
		Combining Proposition~\ref{prop.CLT_M} and the tightness of $\{\Gamma^N\}_{N \in \N}$ proved in Section~\ref{subsec.tightness}, we obtain that the remainder $\{R^N\}_{N \in \N}$ is tight in $D([0,T],\R)$ and the trajectory of the limit process almost surely lies in $C([0,T],\R)$. It remains to characterize the limit.
		Using Corollary~\ref{Cor.TildeGNorm}, \eqref{prop.L2Err} in Proposition~\ref{prop.L2FluxErr} and the chosen scale \eqref{eq.scaleM}, we can obtain:
		\begin{equation*}
			\norm{\bar{G}_{1/N}}_{L^2}^2=O(N),\quad \norm{\tilde{G}_{1/N}}_{L^2}^2=O(N).
		\end{equation*}
		
		The first two terms in the remainder $R^N_t$ can be estimated:
		\begin{equation*}
			\E_{\rho}\left[\left(\frac{1}{\sqrt{N\log N}}\Ll(\tilde{G}_{1/N}(\eta_{Nt})-\tilde{G}_{1/N}(\eta_0)\Rr)\right)^2\right]\leq \frac{C}{\log N},
		\end{equation*}
		and
		\begin{equation*}
			\E_\rho\Ll[\Ll(\frac{1}{\sqrt{N\log N}}\int_{0}^{Nt}\frac{1}{N}\bar{G}_{1/N}(\eta_s)\, \d s\Rr)^2\Rr]\leq \frac{C N^2 t^2}{N^3\log N}\E_\rho\Ll[\Ll(\bar{G}_{1/N}\Rr)^2\Rr]\leq \frac{C t^2}{\log N}.
		\end{equation*}
	
		For the last term in the remainder $R_t^N$, we test it with $V\in L^2$:
		\begin{multline*}
			\expec{\frac{1}{\sqrt{N\log N}}\int_{0}^{Nt}\Ll(\L\tilde{G}_{1/N}-\Lb\bar{G}_{1/N}\Rr)(\eta_s)\, \d s,V}\\
			=\frac{1}{\sqrt{N\log N}}\int_{0}^{Nt}\expec{P_s\Ll(\L\tilde{G}_{1/N}-\Lb\bar{G}_{1/N}\Rr),V}\, \d s.
		\end{multline*}
		Using the reversibility of $\L$ and $\Lb$ under $\P_\rho$, combining \eqref{prop.FluxErr} in Proposition~\ref{prop.L2FluxErr}, Corollary~\ref{Cor.TildeGNorm} with the chosen scale \eqref{eq.scaleM}, we obtain that
		\begin{align*}
			\Ll|\int_0^{Nt}\expec{P_s\Ll(\L\tilde{G}_{1/N}-\Lb\bar{G}_{1/N}\Rr),V}\, \d s\Rr|
			&=\Ll|\int_0^{Nt}\expec{\L\tilde{G}_{1/N}-\Lb\bar{G}_{1/N},P_s V}\, \d s\Rr|\\
			&\leq C\int_{0}^{Nt}\expec{P_s V,-\Lb P_s V}^{\frac{1}{2}}\Ll(3^{-\alpha m}(\log N)^{\frac{1}{2}}+3^m\Rr)\, \d s\\
			&\leq o((\log N)^{\frac{1}{2}})\sqrt{Nt}\Ll(\int_{0}^{N t}\expec{P_s V,-\L P_s V}\, \d s\Rr)^{\frac{1}{2}}\\
			&\leq o(\sqrt{N\log N})\sqrt{t}\norm{V}_{L^2}.
		\end{align*}
		Here in the third line, we use \cite[Corollary 2.3]{gu2025relaxation} and Cauchy--Schwarz inequality.
		Therefore we have 
		\begin{equation*}
			\E_\rho\Ll[\Ll(\frac{1}{\sqrt{N\log N}}\int_{0}^{Nt}\Ll(\L\tilde{G}_{1/N}-\Lb\bar{G}_{1/N}\Rr)(\eta_s)\, \d s\Rr)^2\Rr]= o(1)t.
		\end{equation*}
		The above computations justify that any finite distribution of the limit process coincide with the zero process.  Combining the fact that the trajectory of the limit process almost surely lies in $C([0,T],\R)$, we prove the limit process is essentially the zero process.
	\end{proof}

	We collect enough tools for the main theorem.
	\begin{proof}[Proof of Theorem~\ref{thm.main_basic}]
		Firstly, concerning the occupation time for $d=2$, we have 
		\begin{align*}
			\Gamma^N_t=\tilde{M}^{N}_t+R_t^N.
		\end{align*}
		Using the chosen scale in \eqref{eq.scaleM}, the martingale CLT of $\{\tilde{M}^{N}\}_{N \in \N}$ is established in Proposition~\ref{prop.CLT_M} and the limit of the remainder $\{R^N\}_{N \in \N}$ is justified to be zero in Proposition~\ref{prop.RemainderConverge}. This concludes the scaling limit of $\{\Gamma^N\}_{N \in \N}$.
		
		Secondly, for a general local function $f$, the additive functional has a decomposition
		\begin{align*}
			\Gamma^N_t(f) = \bar{f}^\prime (\rho) \Gamma^N_t + \Ll(\Gamma^N_t (f) - \bar{f}^\prime (\rho) \Gamma^N_t\Rr). 
		\end{align*}
		The second term is exactly
		\begin{align*}
			\Gamma^N_t (f) - \bar{f}^\prime (\rho) \Gamma^N_t = \Gamma^N_t\Ll(f(\eta) - \bar{f}^\prime (\rho) \bar{\eta}(0)\Rr).
		\end{align*}
		Because the function $f(\eta) - \bar{f}^\prime (\rho) \bar{\eta}(0)$ has degree at least $2$, it enters the regime of Kipnis--Varadhan theorem with a normalization factor $\sqrt{N}$; see Proposition~\ref{prop.KVnonGradient}. Therefore, due to the larger normalization factor $\sqrt{N \log N}$, the part ${\big(\Gamma^N_t (f) - \bar{f}^\prime (\rho) \Gamma^N_t)\big)_{t \geq 0}}$ converges weakly to the zero process. In conclusion, the limit of $\{\Gamma^N(f)\}_{N \in \N}$ coincides with that of $\{\bar{f}^\prime (\rho)\Gamma^N\}_{N \in \N}$, which finishes the proof.
	\end{proof}
	
	\section{Critical case \texorpdfstring{$d=1$}{} and \texorpdfstring{$\deg(f)=2$}{}}\label{sec.critical_d_1}
	
	In this section, we discuss the critical case $d=1$ and $\deg(f)=2$. For this case, we will focus on the occupation time on two different sites \eqref{eq.OccupationDeg2} in Section~\ref{subsec.HigherOrder}:
	\begin{equation*}
		\Gamma^N_t := \frac{1}{\sqrt{N\log N}} \int_0^{Nt}  \bar{\eta}_s(0) \bar{\eta}_s(1)\, \d s.
	\end{equation*}
	
	The structure of the proof for deriving the limit of $\{\Gamma_t^N\}_{N \in \N}$ is exactly the same as the dimension $d=2$ case in Section~\ref{sec.critical_d_2}, but in each part, more effort is needed due to its second order chaos nature. 
	
	For our non-gradient exclusion process in one dimension, the diffusion matrix $\DD(\rho)$ is a scalar, and therefore we can construct a SSEP with the same diffusion matrix $\DD(\rho)$. We remark that one-dimensional SSEP, which is the simplest homogenized process in the exclusion family, significantly simplifies the computations and estimates in both the Green function analysis and the homogenization analysis.
	
	\subsection{Resolvent of SSEP}
	Suppose the generator of the constructed SSEP is $\Lb$:
	\begin{equation}\label{eq.defSSEPG}
		\Lb = \DD(\rho)\sum_{x\in\Zo}\pi_{x,x+1}.
	\end{equation}
	Since $\Lb$ is closed in the Fock space, the solution of \eqref{eq.ResolventSSEP} is in the Fock space $\HH_2$ and can be expressed as
	\begin{equation}\label{eq.defBarGH2}
		\bar{G}_\lambda(\eta) = \sum_{Y\in\K_2(\Zo)}\bar{g}_\lambda(Y)\bar{\eta}_Y = \sum_{x\neq y}\bar{g}_\lambda(x,y)\bar{\eta}(x)\bar{\eta}(y).
	\end{equation}
	For convenience, we assume that
	\begin{equation}\label{eq.SymmetricAssumption}
		\bar{g}_\lambda(x,y) = \bar{g}_\lambda(y,x) = \frac{1}{2}\bar{g}_\lambda(\{x,y\}),\quad\forall x\neq y.
	\end{equation}
	
	We denote $\D_1$ and $\D_2$ the finite difference operator defined similarly to \eqref{eq.defD} for a function $f:\Zt\rightarrow\R$:
	\begin{equation*}
		\D_1 f(x,y):= f(x+1,y)-f(x,y),\ \D_2 f(x,y):= f(x,y+1)-f(x,y),\qquad\forall x,y\in\Zo.
	\end{equation*}
	
	For simplicity, we make the following convention throughout Section~\ref{sec.critical_d_1}:
	\begin{align*}
		\D_1\bar{g}_\lambda(x,y)&=0,\quad\forall y\in\{x,x+1\},\\
		\D_1\D_1\bar{g}_\lambda(x,y)&=0,\quad\forall y\in\{x,x+1,x+2\},\\
		\quad \D_1\D_2\bar{g}_\lambda(x,y)&=0,\quad\forall y\in\{x-1,x,x+1\}.
	\end{align*}
	By symmetry we have
	\begin{equation}\label{eq.PibGbar}
		\pi_{x,x+1}\bar{G}_\lambda(\eta) = 2\sum_{y\in\Zo}\D_1\bar{g}_\lambda(x,y)(\bar{\eta}_x-\bar{\eta}_{x+1})\bar{\eta}_y,
	\end{equation}

	To study the Green function $\bar{g}_\lambda$, we consider a continuous-time simple random walk $(S_t)_{t \geq 0}$ on $\K_2(\Zo)$ with transition rate $\DD(\rho)$ to one of its nearest neighbors. 
	\begin{multline*}
		Y=\{x_1,y_1\} \sim Y'=\{x_2,y_2\} \\
		\Longleftrightarrow |x_2-x_1|+|y_2-y_1|=1,\quad \forall Y,Y'\in\K_2(\Zo),\, x_1<y_1,\, x_2<y_2.
	\end{multline*}
	
	For any $Y\in\K_2(\Zo)$, $\bar{g}_\lambda(Y)$ is the Laplace transform of the transition kernel $\bar{p}_t(\{0,1\},Y)$:
	\begin{equation*}
		\bar{g}_\lambda(Y) = \int_{0}^{\infty} e^{-\lambda t} \bar{p}_t(\{0,1\},Y)\, \d t.
	\end{equation*}
	
	\begin{lemma}\label{lem.TransitionGreen}
		The following estimates hold for the transition kernel $\bar{p}_t(\{0,1\},Y)$ and the Green function $\bar{g}_\lambda(Y)$:
		\begin{enumerate}
			\item \cite[Theorem 1.1]{LandimGaussian}\label{lem.TwoParticleTransition} There exist finite positive constants $C_1,C_2,C_3$ such that for every $t>C_1$ and every $x<y$, the following estimate holds:
			\begin{equation*}
				\bar{p}_t(\{0,1\},\{x,y\})\leq \frac{2C_1}{1+t}\Ll(\exp\Ll(-\frac{\norm{(x,y)}_{2}^2}{C_2t}\Rr)+\exp\Ll(-\frac{\norm{(x,y)}_2}{C_2\log t}\Rr)\1_{\norm{(x,y)}_2\geq \frac{C_3et}{\log t}}\Rr),
			\end{equation*}
			where the norm $\norm{(x,y)}_2$ is defined as $\norm{(x,y)}_2^2:=x^2+(y-1)^2.$
			\item \cite[Corollary 2.7, $\theta=-1,\alpha=2\DD(\rho)$]{CarinciExact}\label{lem.TwoParticleGreen}
			The Laplace transform of the transition kernel for every $x<y$ is given by:
			\begin{equation*}
				\bar{g}_\lambda(\{x,y\})\\
				=\frac{1}{8\pi^2\DD(\rho)}\int_{\pi}^{\pi}\int_{-\pi}^{\pi}\frac{e^{i(k_1 x+k_2(y-1))}+\A(\frac{k_1+k_2}{2},\lambda)e^{i(k_1 y+k_2(x-1))}}{2+\lambda-(\cos k_1+\cos k_2)}\, \d k_1\, \d k_2,
			\end{equation*}
			where the function $\A(k,\lambda)$ is defined as:
			\begin{equation*}
				\A(k,\lambda):=2\frac{\cos k\cdot(\xi_{\lambda,k}^{-2}-1)}{\cos k\cdot(\xi_{\lambda,k}^{-2}-1)-2(x_{\lambda,k}-\cos k)}-1,
			\end{equation*}
			with
			\begin{equation*}
				\xi_{\lambda,k}=x_{\lambda,k}-\sqrt{x_{\lambda,k}^2-1},\qquad x_{\lambda,k}=\frac{1}{\cos k}\left(1+\frac{\lambda}{2}\right).
			\end{equation*} 
		\end{enumerate}
	\end{lemma}
	
	\begin{remark}
		\cite[Section 4.3]{erhard2024nonequilibrium} establishes a similar estimate to $(1)$ in Lemma~\ref{lem.TransitionGreen} with an extra factor $t^{-\frac{1}{4}}$ for the difference of the transition kernel
		\begin{equation*} |\bar{p}_t(\{0,1\},\{x+1,y\})-\bar{p}_t(\{0,1\},\{x,y\})|.
		\end{equation*} 
	\end{remark}
	
	Notice that if we fix $\delta>0$ and for $k\in(-\pi+\delta,-\delta)\cup (\delta,\pi-\delta)$, we have
	\begin{equation*}
		\left|\A(k,\lambda)\right|\leq 2\left|\frac{\left(x_{\lambda,k}+\sqrt{x_{\lambda,k}^2-1}\right)^2-1}{\left(x_{\lambda,k}+\sqrt{x_{\lambda,k}^2-1}\right)^2+1-\frac{2x_{\lambda,k}}{\cos k}}\right|+1=2\left|\frac{\sqrt{x_{\lambda,k}^2-1}\left(x_{\lambda,k}+\sqrt{x_{\lambda,k}^2-1}\right)}{\frac{\lambda}{2}\frac{1}{\cos k}x_{\lambda,k}+x_{\lambda,k}\sqrt{x_{\lambda,k}^2-1}}\right|+1.
	\end{equation*}
	Then for $k$ satisfying $\cos k\geq 0 $, we have $x_{\lambda,k}\geq 1$ and therefore
	\begin{equation*}
		\left|\A(k,\lambda)\right|\leq 2\frac{2x_{\lambda,k}\sqrt{x_{\lambda,k}^2-1}}{x_{\lambda,k}\sqrt{x_{\lambda,k}^2-1}}+1=5.
	\end{equation*}
	For $k$ satisfying $\cos k\leq 0$, we have $x_{\lambda,k}\leq -1$ and 
	\begin{equation*}
		\left|\A(k,\lambda)\right|\leq 2\left|\frac{x_{\lambda,k}+\sqrt{x_{\lambda,k}^2-1}}{\frac{\lambda}{2}\frac{1}{\cos k}+\sqrt{x_{\lambda,k}^2-1}}\right|+1\leq \frac{\left(2+\lambda\right)\left(1+\lambda\right)}{(\sin k)^2+\lambda}+1\leq C(\delta),\quad \forall \lambda\in(0,1).
	\end{equation*}
	Combining the previous two estimates, we obtain
	\begin{equation*}
		\left|\A(k,\lambda)\right|\leq C(\delta),\quad \forall k\in(-\pi+\delta,-\delta)\cup (\delta,\pi-\delta),\forall\lambda\in(0,1).
	\end{equation*}
	
	For $k$ near $0$ or $\pm\pi$, denote $\tilde{k}=\min\{|k|,|\pi-k|,|\pi+k|\}$, we have 
	\begin{equation*}
		|x_{\lambda,k}|=1+\frac{\lambda+\tilde{k}^2}{2}+O(\tilde{k}^4+\lambda\tilde{k}^2),\quad x_{\lambda,k}-\cos k=O(\lambda+\tilde{k}^2),\quad \xi_{\lambda,k}^{-2}-1=2\sqrt{\lambda+\tilde{k}^2}+O(\lambda+\tilde{k}^2),
	\end{equation*}
	and therefore
	\begin{equation}\label{eq.ANear0}
		\A(k,\lambda)=1+O\left(\sqrt{\lambda+\tilde{k}^2}\right),\qquad\text{as}\  \tilde{k},\lambda\rightarrow 0.
	\end{equation}
	Combining the above estimates, we know $\A(k,\lambda)$ is uniformly bounded for $k$ and $\lambda$:
	\begin{equation}\label{eq.ABound}
		|\A(k,\lambda)|\leq C,\quad\forall k\in[-\pi,\pi],\, \lambda\in(0,1).
	\end{equation}
	
	Similar to Corollary~\ref{Cor.TildeGNorm}, we can obtain the $L^2$, $H^1$, and $H^2$ estimates for $\bar{G}_\lambda$:
	\begin{proposition}\label{prop.GreenEstimateDeg2}
		We have the following estimates for the Green function $\bar{G}_\lambda$.
		\begin{enumerate}
			\item \label{eq.GreenL2Deg2}
			The $L^2$ norm of $\bar{G}_\lambda$ is of order $\lambda^{-1}$:
			\begin{equation*}
				\norm{\bar{G}_\lambda}_{L^2}^2=O(\lambda^{-1}).
			\end{equation*}
			\item \label{eq.GreenH1Deg2}
			The Dirichlet form of $\bar{G}_\lambda$ is of order $\log(\lambda^{-1})$:
			\begin{equation*}
				\expec{\bar{G}_\lambda,-\Lb\bar{G}_\lambda}=\tfrac{\chi(\rho)^2}{2\pi\DD(\rho)}\log(\lambda^{-1})+O(1).
			\end{equation*}
			\item \label{eq.GreenH2Deg2}
			The $H^{2}$ norm of $\bar{G}_\lambda$ is uniformly bounded with respect to $\lambda\in(0,1)$:
			\begin{equation*}
				\expec{-\Lb\bar{G}_\lambda,-\Lb\bar{G}_\lambda}=O(1).
			\end{equation*}
		\end{enumerate} 
	\end{proposition}
	
	\begin{proof}
		For \eqref{eq.GreenL2Deg2}, we use the reversibility of SSEP and Proposition~\ref{prop.fasterDecay},
		\begin{align*}
			\E_\rho \Ll[\bar{G}_\lambda^2\Rr] &= \int_0^\infty \, \d t \int_0^\infty \, \d s \, e^{-\lambda (t+s)} \E_\rho \Ll[\bar{P}_t (\bar{\eta}(0)\bar{\eta}(1)) \bar{P}_s(\bar{\eta}(0)\bar{\eta}(1))\Rr] \\
			&= \int_0^\infty te^{-\lambda t} \E_\rho \Ll[\Ll(\bar{P}_{t/2} (\bar{\eta}(0)\bar{\eta}(1))\Rr)^2 \Rr]\, \d t  \leq C \lambda^{-1}.
		\end{align*}		
		
		Using \eqref{eq.ResolventSSEP} we have
		\begin{equation}\label{eq.GreenH1Compute}
			\expec{\bar{G}_\lambda,-\Lb\bar{G}_\lambda}=\expec{\bar{G}_\lambda,\bar{\eta}(0)\bar{\eta}(1)}-\lambda\norm{\bar{G}_\lambda}_{L^2}^2=\chi(\rho)^2\bar{g}_\lambda(\{0,1\})+O(1).
		\end{equation}
		Since $\A(k,\lambda)$ is uniformly bounded, for fixed $\delta>0$, we have
		\begin{equation*}
			\Ll|\int_{k_1^2+k_2^2>\delta^2}\frac{e^{i(k_1 x+k_2(y-1))}+\A(\frac{k_1+k_2}{2},\lambda)e^{i(k_1 y+k_2(x-1))}}{2+\lambda-(\cos k_1+\cos k_2)}\, \d k_1\, \d k_2\Rr|\leq C(\delta),
		\end{equation*}
		and therefore the main contribution of $\bar{g}_\lambda(\{x,y\})$ expressed as \eqref{lem.TwoParticleGreen} in Lemma~\ref{lem.TransitionGreen} is from the integration over the domain $k_1^2+k_2^2\leq \delta^2$. Using \eqref{eq.ANear0} we have
		\begin{equation}\label{eq.AsymGreen}
			\bar{g}_\lambda(\{0,1\})=\frac{1}{8\pi^2\DD(\rho)}\int_{-\pi}^{\pi}\int_{-\pi}^{\pi}\frac{2+O\left(\sqrt{\lambda+k_1^2+k_2^2}\right)}{2+\lambda-(\cos k_1+\cos k_2)}\, \d k_1\, \d k_2+O(1).
		\end{equation}
		Using the elementary integral
		\begin{equation*}
			\int_{-\pi}^{\pi}\frac{1}{a-\cos x}\, \d x=\frac{2\pi}{\sqrt{a^2-1}},
		\end{equation*}
		We can compute
		\begin{equation*}
			\int_{-\pi}^{\pi}\int_{-\pi}^{\pi}\frac{1}{2+\lambda-(\cos k_1+\cos k_2)}\, \d k_1 \, \d k_2 = 2\pi \int_{-\pi}^{\pi}\frac{1}{\sqrt{(2+\lambda-\cos k)^2-1}}\, \d k.
		\end{equation*}
		The main contribution from the above integral still comes from the integration over the domain $[-\delta,\delta]$ on which we have the following estimate
		\begin{equation*}
			\frac{1}{\sqrt{(2+\lambda-\cos k)^2-1}}=\frac{1}{\sqrt{2\lambda+k^2}}+O(1).
		\end{equation*}
		Notice that $(\sqrt{2\lambda+k^2})^{-1}$ is integrable:
		\begin{equation*}
			\int_{-\delta}^{\delta}\frac{1}{\sqrt{2\lambda+k^2}}\, \d k = 2 \operatorname{arcsinh}\left(\frac{\delta}{\sqrt{2\lambda}}\right)=\log(\lambda^{-1})+O(1),
		\end{equation*}
		and therefore we have
		\begin{equation}\label{eq.AsymMain}
			\int_{-\pi}^{\pi}\int_{-\pi}^{\pi}\frac{1}{2+\lambda-(\cos k_1+\cos k_2)}\, \d k_1 \, \d k_2 = 2\pi\log(\lambda^{-1})+O(1).
		\end{equation}
		For the remainder term, since $1-\cos x\geq \tfrac{2}{\pi^2}x^2$ for $x\in[-\pi,\pi]$, and using polar coordinate we can compute
		\begin{equation}\label{eq.AsymRemainder}
			\begin{aligned}
				\int_{-\pi}^{\pi}\int_{-\pi}^{\pi}\frac{\sqrt{\lambda+k_1^2+k_2^2}}{2+\lambda-(\cos k_1 + \cos k_2)}\, \d k_1 \, \d k_2
				&\leq C\int_{-\pi}^{\pi}\int_{-\pi}^{\pi}\frac{1}{\sqrt{\lambda+k_1^2+k_2^2}}\, \d k_1 \, \d k_2\\
				&\leq C\int_{0}^{\sqrt{2}\pi}\frac{r}{\sqrt{\lambda+r^2}}\, \d r = O(1).
			\end{aligned}
		\end{equation}
		Combining \eqref{eq.AsymGreen}, \eqref{eq.AsymMain}, and \eqref{eq.AsymRemainder} we have:
		\begin{equation*}
			\bar{g}_\lambda(\{0,1\})=\frac{1}{2\pi\DD(\rho)}\log(\lambda^{-1})+O(1).
		\end{equation*} 
		\eqref{eq.GreenH1Deg2} then follows from \eqref{eq.GreenH1Compute}.
		
		The $H^2$ estimate \eqref{eq.GreenH2Deg2} of $\bar{G}_\lambda$ follows directly from (1) and \eqref{eq.ResolventSSEP}:
		\begin{equation*}
			\norm{-\Lb\bar{G}_\lambda}_{L^2}=\norm{\bar{\eta}(0)\bar{\eta}(1)-\lambda\bar{G}_\lambda}_{L^2}\leq \norm{\lambda\bar{G}_\lambda}_{L^2}+\norm{\bar{\eta}(0)\bar{\eta}(1)}_{L^2}=O(1).
		\end{equation*}
	\end{proof}
	
	In the next proposition, we collect some other useful estimates for $\bar{g}_\lambda$.
	\begin{proposition}\label{prop.SobolevNorm}
		We have the following estimates for $\bar{g}_\lambda$:
		\begin{enumerate}
			\item\label{prop.H11D}The first order finite difference of $\bar{g}_\lambda$ satisfies:
			\begin{equation*}
				\sum_{x,y\in\Zo}\left(\D_1\bar{g}_\lambda(x,y)\right)^2 = \frac{1}{2} \DD(\rho)^{-1}\chi(\rho)^{-2}\expec{\bar{G}_\lambda,-\Lb\bar{G}_\lambda}.
			\end{equation*}
			\item\label{prop.pibpib'} The second order finite difference of $\bar{g}_\lambda$ satisfies:
			\begin{equation*}
				\sum_{x,y\in\Zo}\left(\D_1\D_2\bar{g}_\lambda(x,y)\right)^2\leq\frac{1}{4} \DD(\rho)^{-2}\chi(\rho)^{-2}\expec{-\Lb\bar{G}_\lambda,-\Lb\bar{G}_\lambda},
			\end{equation*}
			\begin{equation*}
				\sum_{x,y\in\Zo}\left(\D_1\D_1\bar{g}_\lambda(x,y)\right)^2\leq \frac{1}{4}\DD(\rho)^{-2}\chi(\rho)^{-2}\expec{-\Lb\bar{G}_\lambda,-\Lb\bar{G}_\lambda}.
			\end{equation*}
			\item\label{prop.H1Near} For every fixed constant $s\geq 1$, there exists a positive constant $C(\rho)$ such that for every $\ell\in\mathbb{Z}_{+}$, the following estimate holds:
			\begin{equation*}
				\sum_{x\in\Zo}\sum_{|y-x|\leq s}\left(\D_1\bar{g}_\lambda(x,y)\right)^2
				\leq s\cdot C(\rho)\left( \ell^{-1}\expec{\bar{G}_\lambda,-\Lb\bar{G}_\lambda}+\ell\expec{-\Lb\bar{G}_\lambda,-\Lb\bar{G}_\lambda}\right).
			\end{equation*}
		\end{enumerate}
	\end{proposition} 
	
	\begin{proof}
		Notice that we have the relation 
		\begin{equation*}
			\pi_b\pi_b=-2\pi_{b},\qquad\forall b\in\K_2,
		\end{equation*}
		and if the bond $b$ and $b'$ are disjoint, the Kawasaki operators $\pi_b$ and $\pi_{b'}$ commute:
		\begin{equation*}
			\pi_b\pi_{b'}=\pi_{b'}\pi_{b},\quad \forall b\cap b'=\emptyset.
		\end{equation*}
		Moreover, the Kawasaki operator is a self-adjoint operator in $L^2(\X,\F,\P_\rho)$:
		\begin{equation*}
			\expec{\pi_b F,G} = \expec{F,\pi_b G},\quad\forall F,G\in L^2(\X,\F,\P_\rho),\quad\forall b\in\K_2.
		\end{equation*}
		Therefore, when the bond $b$ and $b'$ are disjoint, we have:
		\begin{equation}\label{eq.disjointbond}
			\expec{\pi_b F,\pi_{b'}F} = \frac{1}{4}\expec{\pi_b\pi_b F,\pi_{b'}\pi_{b'}F} = \frac{1}{4}\expec{\pi_{b}\pi_{b'}F,\pi_{b}\pi_{b'}F} = \frac{1}{4}\E_\rho\Ll[(\pi_b \pi_{b'}F)^2\Rr].
		\end{equation}
		By \eqref{eq.defSSEPG} and the previous observation, we have
		\begin{equation*}
			\expec{\bar{G}_\lambda,-\Lb\bar{G}_\lambda}=\frac{1}{2}\DD(\rho)\sum_{x\in\Zo}\E_\rho\Ll[\left(\pi_{x,x+1}\bar{G}_\lambda\right)^2\Rr].
		\end{equation*}
		Combining \eqref{eq.PibGbar} we obtain estimate \eqref{prop.H11D}:
		\begin{equation*}
			\expec{\bar{G}_\lambda,-\Lb\bar{G}_\lambda}=2\DD(\rho)\chi(\rho)^2\sum_{x,y\in\Zo}\left(\D_1\bar{g}_\lambda(x,y)\right)^2.
		\end{equation*}	
		Using \eqref{eq.defSSEPG}, we can decompose $\expec{-\Lb\bar{G}_\lambda,-\Lb\bar{G}_\lambda}$ into:
		\begin{equation}\label{eq.BarGH2Decom}
			\begin{multlined}
				\expec{-\Lb\bar{G}_\lambda,-\Lb\bar{G}_\lambda} =\DD(\rho)^2\bigg(\sum_{\{x,x+1\}\cap\{y,y+1\}=\emptyset}\expec{\pi_{x,x+1}\bar{G}_\lambda,\pi_{y,y+1}\bar{G}_\lambda}\\ +\sum_{x\in\Zo}\expec{\pi_{x,x+1}\bar{G}_\lambda,\pi_{x,x+1}\bar{G}_\lambda}+2\sum_{x\in\Zo}\expec{\pi_{x,x+1}\bar{G}_\lambda,\pi_{x-1,x}\bar{G}_\lambda}\bigg).
			\end{multlined}
		\end{equation}
		Combining \eqref{eq.disjointbond} and \eqref{eq.PibGbar}, for disjoint bonds $\{x,x+1\}$ and $\{y,y+1\}$, we can compute:
		\begin{equation}\label{eq.D1D2gBar}
			\expec{\pi_{x,x+1}\bar{G}_\lambda,\pi_{y,y+1}\bar{G}_\lambda}=\frac{1}{4}\E_\rho\Ll[(\pi_{x,x+1}\pi_{y,y+1}\bar{G}_\lambda)^2\Rr]=4\chi(\rho)^2\Ll(\D_1\D_2\bar{g}_\lambda(x,y)\Rr)^2.
		\end{equation}
		Using the expression \eqref{eq.defBarGH2}, we can calculate other terms as
		\begin{equation*}
			\expec{\pi_{x,x+1}\bar{G}_\lambda,\pi_{x,x+1}\bar{G}_\lambda}=8\chi(\rho)^2\sum_{y\in\Zo}\Ll(\D_1\bar{g}_\lambda(x,y)\Rr)^2,
		\end{equation*}
		and
		\begin{multline*}
			\expec{\pi_{x,x+1}\bar{G}_\lambda,\pi_{x-1,x}\bar{G}_\lambda}\\	=-4\chi(\rho)^2\Ll(\sum_{y\in\Zo}\D_1\bar{g}_\lambda(x,y)\D_1\bar{g}_\lambda(x-1,y)+\D_1\bar{g}_\lambda(x-1,x+1)\D_1\bar{g}_\lambda(x,x-1)\Rr).
		\end{multline*}
		Therefore, we consider one pair of bonds $\{x-1,x\}$ and $\{x,x+1\}$ and obtain:
		\begin{equation}\label{eq.D1D1gBar}
			\begin{multlined}
				\expec{\pi_{x-1,x}\bar{G}_\lambda,\pi_{x-1,x}\bar{G}_\lambda}+\expec{\pi_{x,x+1}\bar{G}_\lambda,\pi_{x,x+1}\bar{G}_\lambda}+4\expec{\pi_{x,x+1}\bar{G}_\lambda,\pi_{x-1,x}\bar{G}_\lambda}\\
				=8\chi(\rho)^2\bigg(\sum_{y\in\Zo}\Ll(\D_1\D_1\bar{g}_\lambda(x-1,y)\Rr)^2+\Ll(\D_1\bar{g}_\lambda(x,x-1)-\D_1\bar{g}_\lambda(x-1,x+1)\Rr)^2\bigg).
			\end{multlined}
		\end{equation}
		We can conclude \eqref{prop.pibpib'} through \eqref{eq.BarGH2Decom}, \eqref{eq.D1D2gBar}, and \eqref{eq.D1D1gBar}.
		
		For the last estimate \eqref{prop.H1Near}, we first fix $y$ such that $|y-x|\leq s$. Without loss of generality we can assume $y<x$, and by Cauchy--Schwarz inequality we have
		\begin{align*}
			\left(\D_1\bar{g}_\lambda(x,y)\right)^2
			&\leq 2\left(\frac{1}{\ell}\sum_{z\in [y-\ell,y-1]}\D_1\bar{g}_\lambda(x,z)\right)^2 + 2\left(\D_1\bar{g}_\lambda(x,y)-\frac{1}{\ell}\sum_{z\in [y-\ell,y-1]}\D_1\bar{g}_\lambda(x,z)\right)^2\\
			&\leq \frac{2}{\ell}\sum_{z\in\Zo}\left(\D_1\bar{g}_\lambda(x,z)\right)^2 + 2\left(\sum_{i=1}^{\ell}\frac{i}{\ell}\, \D_1\D_2\bar{g}_\lambda(x,y-\ell+i-1)\right)^2\\
			&\leq \frac{2}{\ell}\sum_{z\in\Zo}\left(\D_1\bar{g}_\lambda(x,z)\right)^2+\ell\sum_{z\in\Zo}\left(\D_1\D_2\bar{g}_\lambda(x,z)\right)^2,
		\end{align*}
		and \eqref{prop.H1Near} then follows from summing over all $x\in\Zo$ and $y\in\Zo$ such that $|y-x|\leq s$, \eqref{prop.H11D}, and \eqref{prop.pibpib'}.
	\end{proof}
	
	For every $r\in\Zo_+$, we can define the following notations:
	\begin{equation}\label{eq.defDr}
		\D_{(r)}\bar{g}_\lambda(x,y):=\sum_{|x'-x|\leq r}\sum_{|y'-y|\leq r}|\D_1 \bar{g}_\lambda(x',y')|,
	\end{equation}
	\begin{equation}\label{eq.defD2r}
		\D^2_{(r)}\bar{g}_\lambda(x,y):=\sum_{|x'-x|\leq r}\sum_{|y'-y|\leq r}\left(|\D_1\D_1\bar{g}_\lambda(x',y')|+|\D_1\D_2\bar{g}_\lambda(x',y')|\right),
	\end{equation}
	\begin{equation}\label{eq.defDTilde}
		\tilde{\D}_{(r)}\bar{g}_\lambda:=\D_{(r)}\bar{g}_\lambda+\D^2_{(r)}\bar{g}_\lambda.
	\end{equation}
	The next proposition will be used in estimating the variance of the quadratic variation of the martingale associated with $\tilde{G}_\lambda$:
	\begin{proposition}\label{prop.GreenQuartic}
		We have the following estimates for the quartic sum consisting of the finite difference of $\bar{g}_\lambda$:
		\begin{enumerate}
			\item\label{prop.QuarticFourBlock} We have the following estimates for the four-block quartic sum.
			\begin{enumerate}
				\item There exists a constant $C$ independent of $\lambda$ such that for every non-negative bounded weight function $f\in\ell^{\infty}(\Zo)$, the following estimate holds:
				\begin{multline*}
					\sum_{x,x',y,y'\in\Zo}f(x)f(x') \D_1\bar{g}_\lambda(x,y)\D_1\bar{g}_\lambda(x,y')\D_1\bar{g}_\lambda(x',y)\D_1\bar{g}_\lambda(x',y')\\
					\leq C\norm{f}_{\ell^\infty(\Zo)}^{2}\log(\lambda^{-1}).
				\end{multline*}
				\item There exists a constant $C$ independent of $\lambda$, such that for every $r\in\Zo_{+}$, the following estimate holds:
				\begin{multline*}
					\sum_{z_i,z_i'\in\Z_m}\D^2_{(r)}\bar{g}_\lambda(z_1,z_1')\tilde{\D}_{(r)}\bar{g}_\lambda(z_2,z_1')\tilde{\D}_{(r)}\bar{g}_\lambda(z_1,z_2')\tilde{\D}_{(r)}\bar{g}_\lambda(z_2,z_2')\\
					\leq C r^8\left((\log(\lambda^{-1}))^{\frac{3}{2}}+1\right).
				\end{multline*}
			\end{enumerate}
			
			\item\label{prop.QuarticThreeBlock}	There exists a constant $C$ independent of $\lambda$, such that for every $r\in\Zo_{+}$, the following estimates about the three-block quartic sum hold.
			
			\begin{enumerate}
				\item For some terms which we can write as a summation of quadratic terms
				\begin{multline*}
					\sum_{z\in\Z_m}\left(\sum_{z'\in\Z_m}\left(\tilde{\D}_{(r)}\bar{g}_\lambda(z,z')\right)^2\right)^2+	\sum_{z'\in\Z_m}\left(\sum_{z\in\Z_m}\left(\tilde{\D}_{(r)}\bar{g}_\lambda(z,z')\right)^2\right)^2\\
					\leq Cr^8\left(\log(\lambda^{-1})+1\right),
				\end{multline*} 
				\item For some other three blocks  quartic sum $\operatorname{TB}$
				\begin{equation*}
					\operatorname{TB}\leq Cr^8\left(\left(\log(\lambda^{-1})\right)^{\frac{3}{2}}+1\right),
				\end{equation*}
				where $\operatorname{TB}$ can be one of the following:
				\begin{multline*}
					\sum_{z_i\in\Z_m}\left(\tilde{\D}_{(r)}\bar{g}_\lambda(z_1,z_1)+\tilde{\D}_{(r)}\bar{g}_\lambda(z_1,z_2)\right)\Ll(\tilde{\D}_{(r)}\bar{g}_\lambda(z_2,z_2)+\tilde{\D}_{(r)}\bar{g}_\lambda(z_2,z_1)\Rr)\\
					\tilde{\D}_{(r)}\bar{g}_\lambda(z_1,z_3)\tilde{\D}_{(r)}\bar{g}_\lambda(z_2,z_3),
				\end{multline*}
				or
				\begin{equation*}
					\sum_{z_i\in\Z_m}\left(\tilde{\D}_{(r)}\bar{g}_\lambda(z_1,z_3)\right)^2\left(\tilde{\D}_{(r)}\bar{g}_\lambda(z_2,z_2)+\tilde{\D}_{(r)}\bar{g}_\lambda(z_2,z_1)\right)^2.
				\end{equation*}
			\end{enumerate}
			
			\item\label{prop.QuarticTwoBlock} There exists a constant $C$ independent of $\lambda$, such that for every $r\in\Zo_{+}$, the following estimate about the two-block quartic sum holds:
			\begin{multline*}
				\sum_{z_i\in\Z_m}\left(\tilde{\D}_{(r)}\bar{g}_\lambda(z_1,z_1)+\tilde{\D}_{(r)}\bar{g}_\lambda(z_1,z_2)\right)^2\Ll(\tilde{\D}_{(r)}\bar{g}_\lambda(z_2,z_2)+\tilde{\D}_{(r)}\bar{g}_\lambda(z_2,z_1)\Rr)^2\\
				\leq Cr^8 \left(\log(\lambda^{-1})+1\right).
			\end{multline*}
		\end{enumerate}
		
	\end{proposition}
	\begin{proof}
		We prove estimate $(\text{a})$ in \eqref{prop.QuarticFourBlock} first.
		Notice that we can rewrite it as a summation of quadratic terms:
		\begin{multline*}
			\sum_{x,x',y,y'\in\Zo}f(x)f(x')\D_1\bar{g}_\lambda(x,y)\D_1\bar{g}_\lambda(x,y')\D_1\bar{g}_\lambda(x',y)\D_1\bar{g}_\lambda(x',y')\\
			=\sum_{x,x'\in\Zo}f(x)f(x')\left(\sum_{y}\D_1\bar{g}_\lambda(x,y)\D_1 \bar{g}_\lambda(x',y)\right)^2\leq \norm{f}^2_{\ell^{\infty}(\Zo)}Q(\D_1\bar{g}_\lambda),
		\end{multline*}
		where in the last inequality we use the non-negativeness and boundedness of the weight function $f$ and we denote the quartic sum without weight by $Q(\D_1\bar{g}_\lambda)$ for simplicity
		\begin{equation*}
			Q(\D_1\bar{g}_\lambda):=\sum_{x,x',y,y'\in\Zo}\D_1\bar{g}_\lambda(x,y)\D_1\bar{g}_\lambda(x,y')\D_1\bar{g}_\lambda(x',y)\D_1\bar{g}_\lambda(x',y').
		\end{equation*}
		
		In the following part, we will use the operator theory to estimate $Q(\D_1\bar{g}_\lambda)$.
		For a kernel function $g:\mathbb{Z}^2\rightarrow\R$ satisfying $\norm{g}_{\ell^2(\Zt)}<\infty$, we can 
		define a linear bounded operator $T[g]$ on $\ell^2(\Zo)$ of convolutional type:
		\begin{equation*}
			(T[g]\circ f)(x):=\sum_{y\in\Zo}g(x,y)f(y),\quad\forall f\in\ell^2(\Zo).
		\end{equation*}
		We consider the \emph{Schatten 4-class} on the Hilbert space $\ell^2(\Zo)$ consisting of all linear compact operators $T$ on $\ell^2(\Zo)$ with a finite \emph{4-th Schatten norm}:
		\begin{equation*}
			\norm{T}_{S_4}:=\left(\operatorname{Tr}\left((T^*\circ T)^{2}\right)\right)^{\frac{1}{4}}.
		\end{equation*}
		Notice that for the convolutional type operator $T[g]$ we have
		\begin{equation*}
			\norm{T[g]}_{S_4}^4=\operatorname{Tr}\left(((T[g])^*\circ T[g])^2\right)=\sum_{x,y,z,w\in\Zo}g(x,z)g(y,z)g(x,w)g(y,w),
		\end{equation*}
		and therefore we have
		\begin{equation*}
			Q(\D_1\bar{g}_\lambda)=\norm{T[\D_1\bar{g}_\lambda]}_{S_4}^4.
		\end{equation*}
		
		We utilize \eqref{lem.TwoParticleGreen} in Lemma~\ref{lem.TransitionGreen} to define a full kernel $g_\lambda$ on $\Zt$ which can roughly be viewed as the Green function of two independent particle starting from the source $(0,1)$:
		\begin{equation*}
			g_\lambda(x,y):=\frac{1}{16\pi^2\DD(\rho)}\int_{\pi}^{\pi}\int_{-\pi}^{\pi}\frac{e^{i(k_1 x+k_2(y-1))}+\A(\frac{k_1+k_2}{2},\lambda)e^{i(k_1 y+k_2(x-1))}}{2+\lambda-(\cos k_1+\cos k_2)}\, \d k_1\, \d k_2.
		\end{equation*}
		
		For a kernel function $g$ on $\Zt$, we define its triangular truncation $U\circ g$ and $U'\circ g$ as:
		\begin{equation*}
			U\circ g(x,y):=g(x,y)\1_{\{x\geq y\}},\quad
			U'\circ g(x,y):=g(x,y)\1_{\{x\geq y-1\}}.
		\end{equation*}
		Then from our symmetric assumption \eqref{eq.SymmetricAssumption} and  \eqref{lem.TwoParticleGreen} in Lemma~\ref{lem.TransitionGreen}, we have
		\begin{equation*}
			\D_1\bar{g}_\lambda = \D_1 g_\lambda - U'\circ \D_1g_\lambda - (\D_2 g_\lambda - U\circ \D_2g_\lambda)^*.
		\end{equation*}
		By Macaev’s theorem \cite[Proposition 4.2]{ArazySome}(see also \cite[Corollary 5.4.5]{HytonenAnalysis} for more details), the triangular truncation is a bounded operator on the \emph{Schatten p-class} for any $p\in(1,\infty)$, and combining $\norm{T}_{S_p}=\norm{T^*}_{S_p}$ we have
		\begin{equation*}
			\norm{T[U'\circ \D_1g_\lambda]}_{S_4}\leq C\norm{T[\D_1g_\lambda]}_{S_4},\quad \norm{T[U\circ \D_2g_\lambda]}_{S_4}\leq C\norm{T[\D_2g_\lambda]}_{S_4}.
		\end{equation*}
		Therefore, by Cauchy--Schwarz inequality, we have
		\begin{equation}\label{eq.SymmetricSeparation}
			Q(\D_1\bar{g}_\lambda)\leq C(\norm{T[\D_1g_\lambda]}_{S_4}^4+\norm{T[\D_2g_\lambda]}_{S_4}^4).
		\end{equation}
		
		We consider the Fourier transform $\mathscr{F}:\ell^2(\Zo)\rightarrow L^2(\mathbb{T})$ and the inverse Fourier transform $\mathscr{F}^{-1}:L^2(\mathbb{T}) \rightarrow \ell^2(\Zo)$ defined by:
		\begin{equation*}
			\mathscr{F}\circ f(k) = \sum_{x\in\Zo} f(x) e^{-i k x},\quad \mathscr{F}^{-1}\circ F(x)=\frac{1}{2\pi}\int_{-\pi}^{\pi}F(k)e^{ikx}\, \d k,\quad\forall f\in\ell^2(\Zo), F\in L^2(\mathbb{T}).
		\end{equation*}
		
		For any linear bounded operator $T:
		\ell^2(\Zo)\rightarrow\ell^2(\Zo)$, we can consider the Fourier conjugate of $T$ defined by:
		\begin{equation*}
			\tilde{T}:L^2(\mathbb{T})\rightarrow L^2(\mathbb{T}),\quad \tilde{T}:= \mathscr{F}\circ T \circ \mathscr{F}^{-1}.
		\end{equation*}		
		By Parseval's theorem, the Fourier transform is a unitary transform from $\ell^2(\Zo)$ to $L^2(\mathbb{T})$, which means it preserves the singular value of the operator. The \emph{Schatten norm} is invariant under unitary transformation, and therefore we have
		\begin{equation}\label{eq.UnitaryInvariance}
			\norm{T}_{S_4(\ell^2(\Zo))}=\norm{\tilde{T}}_{S_4(L^2(\mathbb{T}))}.
		\end{equation}
		
		Denote $A(k_1,k_2,\lambda)$ and $A_{-}(k_1,k_2,\lambda)$ the Fourier mode in $g_\lambda$:
		\begin{align*}
			A(k_1,k_2,\lambda)&:=\frac{1}{4\DD(\rho)}\frac{1}{2+\lambda-(\cos k_1+\cos k_2)},\\
			A_{-}(k_1,k_2,\lambda)&:=\frac{1}{4\DD(\rho)}\frac{\A(\frac{k_1+k_2}{2},\lambda)}{2+\lambda-(\cos k_1+\cos k_2)}.
		\end{align*}
		Then by symmetry, we have
		\begin{equation}\label{eq.glambda}
			g_\lambda = \mathscr{F}_1^{-1}\otimes\mathscr{F}_2^{-1}\Ll( e^{-i k_2}A(k_1,k_2,\lambda)\Rr) + \mathscr{F}_1^{-1}\otimes\mathscr{F}_2^{-1}\Ll( e^{-i k_1}A_{-}(k_1,k_2,\lambda)\Rr),
		\end{equation}
		where $\mathscr{F}_1^{-1}$ and $\mathscr{F}_2^{-1}$ are the inverse Fourier transform on the first and second coordinate respectively. Therefore we have
		\begin{equation}\label{eq.D1glambda}
			\D_1 g_\lambda = \mathscr{F}_1^{-1}\otimes\mathscr{F}_2^{-1}\left(M_\lambda(k_1,k_2)\right),
		\end{equation}
		where the multiplier $M_\lambda(k_1,k_2)$ is defined as:
		\begin{equation}\label{eq.defMLambda}
			M_\lambda(k_1,k_2):=(e^{ik_1}-1) e^{-i k_2}A(k_1,k_2,\lambda)+(1- e^{-i k_1}) A_{-}(k_1,k_2,\lambda).
		\end{equation}
		We can compute the Fourier transform of $T[\D_1 g_\lambda]$ for $f\in \ell^2(\Zo)$:
		\begin{align*}
			\mathscr{F}\circ T[\D_1 g_\lambda]\circ f(k)
			&= \sum_{x\in\Zo}T[\D_1 g_\lambda]\circ f(x) e^{-i k x}\\
			&=\sum_{x\in\Zo}\sum_{y\in\Zo}\D_1 g_\lambda(x,y)f(y)e^{-ik x}\\
			&=\sum_{y\in\Zo}f(y)\mathscr{F}_1\circ(\D_1 g_\lambda(k,y))\\
			&=\mathscr{F}_2\circ(f\mathscr{F}_1\circ(\D_1 g_\lambda(k,\cdot)))(0)\\
			&=\frac{1}{2\pi}\int_{-\pi}^{\pi}\mathscr{F}_1\otimes\mathscr{F}_2\left(\D_1 g_\lambda(k, -r)\right)\mathscr{F}\circ f(r)\, \d r.
		\end{align*}
		Therefore, we have
		\begin{equation*}
			\tilde{T[\D_1 g_\lambda]}\circ F(k) = \frac{1}{2\pi}\int_{-\pi}^{\pi}M_\lambda(k,-r)F(r)\, \d r,\quad\forall F\in L^2(\mathbb{T}),
		\end{equation*}
		We can compute the adjoint operator of $\tilde{T[\D_1 g_\lambda]}$:
		\begin{equation*}
			\left(\tilde{T[\D_1 g_\lambda]}\right)^*\circ F(k)= \frac{1}{2\pi}\int_{-\pi}^{\pi}\bar{M_\lambda(r,-k)}F(r)\, \d r,\quad\forall F\in L^2(\mathbb{T}).
		\end{equation*}
		The composition of these two operators is also of convolutional type:
		\begin{equation*}
			\left(\tilde{T[\D_1 g_\lambda]}\right)^*\circ \tilde{T[\D_1 g_\lambda]}\circ F(k) = \frac{1}{2\pi}\int_{-\pi}^{\pi}B_\lambda(k,r)F(r)\, \d r,
		\end{equation*}
		with the following kernel $B_\lambda(k,r)$:
		\begin{equation}\label{eq.defBLambda}
			B_\lambda(k,r)=\frac{1}{2\pi}\int_{-\pi}^{\pi}\bar{M_\lambda(s,-k)}M_\lambda(s,-r)\, \d s.
		\end{equation}
		
		We can finally estimate the \emph{Schatten norm}:
		\begin{align*}
			\norm{\tilde{T[\D_1 g_\lambda]}}_{S_4(L^2(\mathbb{T}))}^4
			&=\norm{\left(\tilde{T[\D_1 g_\lambda]}\right)^*\circ \tilde{T[\D_1 g_\lambda]}}_{\text{HS}(L^2(\mathbb{T}))}^2\\
			&=\frac{1}{4\pi^2}\int_{-\pi}^{\pi}\int_{-\pi}^{\pi}\left|B_\lambda(k,r)\right|^2\, \d k \, \d r,
		\end{align*}
		where $\norm{\cdot}_{\text{HS}(L^2(\mathbb{T}))}$ is the Hilbert–Schmidt norm for the operator on $L^2(\mathbb{T})$.
		Since $\A$ is uniformly bounded for $0<\lambda<1$. From \eqref{eq.defMLambda} we have
		\begin{equation*}
			|M_\lambda(k_1,k_2)|\leq C\frac{|k_1|}{2+\lambda-(\cos k_1+\cos k_2)}.
		\end{equation*}
		Using the elementary integral
		\begin{equation*}
			\int_{\R}\frac{x^2}{(x^2+a)(x^2+b)} \, \d x = \frac{\pi}{\sqrt{a}+\sqrt{b}},\qquad\forall a,b>0,
		\end{equation*}
		and the fact that $1-\cos x\geq \tfrac{2}{\pi^2}x^2$ for $x\in[-\pi,\pi]$, we have 
		\begin{equation}\label{eq.EsBLambda}
			\begin{aligned}
				|B_\lambda(k,r)|
				&\leq C\int_{-\pi}^{\pi}\frac{s^2}{(2+\lambda-(\cos k+\cos s))(2+\lambda-(\cos r+\cos s))}\, \d s\\
				&\leq C\frac{1}{\sqrt{1+\lambda-\cos k}+\sqrt{1+\lambda-\cos r}}.
			\end{aligned}
		\end{equation}
		Therefore, we obtain the desired estimate about the \emph{Schatten norm}:
		\begin{align*}
			\norm{\tilde{T[\D_1 g_\lambda]}}_{S_4(L^2(\mathbb{T}))}^4
			&\leq C\int_{-\pi}^{\pi}\int_{-\pi}^{\pi}\frac{1}{(\sqrt{1+\lambda-\cos k}+\sqrt{1+\lambda-\cos r})^2}\, \d k \, \d r\\
			&\leq C\int_{-\pi}^{\pi}\int_{-\pi}^{\pi}\frac{1}{2+2\lambda-(\cos k+\cos r)}\, \d k \, \d r\\
			&\leq C\log(\lambda^{-1}),
		\end{align*}
		where in the second line, we have already estimated this integral in \eqref{eq.AsymMain}.
		A similar estimate yields the \emph{Schatten norm} for the finite difference on the second coordinate satisfying
		\begin{equation*}
			\norm{\tilde{T[\D_2 g_\lambda]}}_{S_4(L^2(\mathbb{T}))}^4\leq C\log (\lambda^{-1}).
		\end{equation*}
		Combining \eqref{eq.SymmetricSeparation}, \eqref{eq.UnitaryInvariance}, and the previous two estimates we obtained $(\text{a})$ in  \eqref{prop.QuarticFourBlock}.
		
		We next prove $(\text{b})$ in \eqref{prop.QuarticFourBlock}. We will consider the following operators:
		\begin{equation*}
			T[\tilde{\D}_{(r)}\bar{g}_\lambda],\quad T[\D^2_{(r)}\bar{g}_\lambda].
		\end{equation*}
		Since $\tilde{\D}_{(r)}\bar{g}_\lambda$ and $\D^2_{(r)}\bar{g}_\lambda$ are non-negative functions on $\Zt$, we have
		\begin{equation*}
			\text{LHS} \leq \operatorname{Tr}(T[\D^2_{(r)}\bar{g}_\lambda]\circ T[\tilde{\D}_{(r)}\bar{g}_\lambda]^*\circ T[\tilde{\D}_{(r)}\bar{g}_\lambda]\circ T[\tilde{\D}_{(r)}\bar{g}_\lambda]^*).
		\end{equation*}
		Using the H\"older inequality for \emph{Schatten classes} \cite[Corollary 2.9]{ArazySome}, we can obtain
		\begin{equation*}
			\text{LHS} \leq \norm{T[\D^2_{(r)}\bar{g}_\lambda]}_{S_4}\norm{T[\tilde{\D}_{(r)}\bar{g}_\lambda]}^3_{S_4}.
		\end{equation*}
		Utilizing the singular value formulation of the \emph{Schatten norm}, we have $S_2\subset S_4$ as a consequence of $\ell^2(\N)\subset \ell^4(\N)$, and therefore the following estimates hold:
		\begin{equation*}
			\norm{T[\D^2_{(r)}\bar{g}_\lambda]}_{S_4}\leq \norm{T[\D^2_{(r)}\bar{g}_\lambda]}_{\operatorname{HS}},\quad \norm{T[\tilde{\D}_{(r)}\bar{g}_\lambda]}_{S_4}\leq \norm{T[\tilde{\D}_{(r)}\bar{g}_\lambda]}_{\operatorname{HS}}.
		\end{equation*}
		Using Proposition~\ref{prop.SobolevNorm}, Proposition~\ref{prop.GreenEstimateDeg2} and Cauchy--Schwarz inequality, we have
		\begin{align*}
			\norm{T[\D^2_{(r)}\bar{g}_\lambda]}_{\operatorname{HS}}^2 &= \sum_{x,y\in\Zo}\Ll(\D^2_{(r)}\bar{g}_\lambda(x,y)\Rr)^2\\
			&\leq Cr^4\sum_{x,y\in\Zo}\Ll(\Ll(\D_1\D_1\bar{g}_\lambda(x,y)\Rr)^2+\Ll(\D_1\D_2\bar{g}_\lambda(x,y)\Rr)^2\Rr)\leq Cr^4,
		\end{align*}
		and 
		\begin{align*}
			\norm{T[\tilde{\D}_{(r)}\bar{g}_\lambda]}_{\operatorname{HS}}^2 &= \sum_{x,y\in\Zo}\Ll(\tilde{\D}_{(r)}\bar{g}_\lambda(x,y)\Rr)^2\\
			& \leq Cr^4 \sum_{x,y\in\Zo}\Ll(\Ll(\D_1\bar{g}_\lambda(x,y)\Rr)^2+\Ll(\D_1\D_1\bar{g}_\lambda(x,y)\Rr)^2+\Ll(\D_1\D_2\bar{g}_\lambda(x,y)\Rr)^2\Rr)\\
			& \leq Cr^4(\log(\lambda^{-1})+1).
		\end{align*}
		Combining the above estimates, we can obtain the desired estimate $(\text{b})$ in \eqref{prop.QuarticFourBlock}.
		
		We next prove $(\text{a})$ in \eqref{prop.QuarticThreeBlock}. Using Cauchy--Schwarz inequality, we can obtain
		\begin{multline*}
			\sum_{z'\in\Z_m}\left(\tilde{\D}_{(r)}\bar{g}_\lambda(z,z')\right)^2\\
			\leq Cr^3\sum_{|x-z|\leq r}\sum_{y\in\Zo}\left(\left(\D_1\bar{g}_\lambda(x,y)\right)^2+\left(\D_1\D_1\bar{g}_\lambda(x,y)\right)^2+\left(\D_1\D_2\bar{g}_\lambda(x,y)\right)^2\right).
		\end{multline*}
		Using Cauchy--Schwarz inequality again, we have
		\begin{multline*}
			\sum_{z\in\Z_m}\left(\sum_{z'\in\Z_m}\left(\tilde{\D}_{(r)}\bar{g}_\lambda(z,z')\right)^2\right)^2\\
			\leq Cr^8\sum_{x\in\Zo}\left(\sum_{y\in\Zo}\left(\left(\D_1\bar{g}_\lambda(x,y)\right)^2+\left(\D_1\D_1\bar{g}_\lambda(x,y)\right)^2+\left(\D_1\D_2\bar{g}_\lambda(x,y)\right)^2\right)\right)^2.
		\end{multline*}
		
		We first show the following estimate holds:
		\begin{equation}\label{eq.H1FixedOnePoint}
			\sup_{\lambda\in(0,1)}\sup_{x\in\Zo}\sum_{y\in\Zo}\left(\D_1\bar{g}_\lambda(x,y)\right)^2\leq C.
		\end{equation}
		To prove \eqref{eq.H1FixedOnePoint}, we fixed $x\in\Zo$ and use the same trick in the proof of \eqref{prop.QuarticFourBlock}:
		\begin{equation*}
			\sum_{y\in\Zo}\left(\D_1\bar{g}_\lambda(x,y)\right)^2\leq \sum_{y\in\Zo}\left(\left(\D_1 g_\lambda(x,y)\right)^2+\left(\D_2 g_\lambda(x,y)\right)^2\right).
		\end{equation*}
		By Parseval's theorem and \eqref{eq.D1glambda}, we have
		\begin{align*}
			\sum_{y\in\Zo}\left(\D_1 g_\lambda(x,y)\right)^2
			&= \frac{1}{2\pi}\int_{-\pi}^{\pi}\left|\mathscr{F}_2\cdot \D_1 g_\lambda (x,k)\right|^2\, \d k\\
			&=\frac{1}{2\pi}\int_{-\pi}^{\pi}\left|\frac{1}{2\pi}\int_{-\pi}^{\pi}M_\lambda(k_1,k_2)e^{ik_1 x}\, \d k_1 \right|^2\, \d k_2\\
			&\leq \frac{1}{8\pi^3 }\int_{-\pi}^{\pi}\left(\int_{-\pi}^{\pi}|M_\lambda(k_1,k_2)|\, \d k_1\right)^2\, \d k_2.
		\end{align*}
		Using \eqref{eq.defMLambda}, the uniform boundedness of $\A(k,\lambda)$ \eqref{eq.ABound}, and the fact that ${1-\cos x\geq \tfrac{2}{\pi^2}x^2}$ for $x\in[-\pi,\pi]$ we have
		\begin{align*}
			\int_{-\pi}^{\pi}|M_\lambda(k_1,k_2)|\, \d k_1\
			&\leq \int_{-\pi}^{\pi}\frac{|k_1|}{2+\lambda-(\cos k_1+\cos k_2)}\, \d k_1\\
			&\leq C\int_{-\pi}^{\pi}\frac{|k_1|}{\lambda+k_1^2+k_2^2}\, \d k_1\leq C\log\left(1+\frac{\pi^2}{\lambda+k_2^2}\right).
		\end{align*}
		For $k_2\in[-\pi,\pi]$, we have
		\begin{equation*}
			\log\left(1+\frac{\pi^2}{\lambda+k_2^2}\right)\leq \log\left(\frac{k_2^2+\pi^2}{k_2^2}\right)\leq C(1+|\log |k_2||),
		\end{equation*}
		and therefore, we can obtain
		\begin{equation*}
			\sum_{y\in\Zo}\left(\D_1 g_\lambda(x,y)\right)^2\leq C\int_{-\pi}^{\pi}\left(1+|\log|k_2||\right)^2\, \d k_2\leq C.
		\end{equation*}
		Similarly, the following estimate holds:
		\begin{align*}
			\sum_{y\in\Zo}\left(\D_2 g_\lambda(x,y)\right)^2
			&\leq C\int_{-\pi}^{\pi}\left(\int_{-\pi}^{\pi}\frac{|k_2|}{2+\lambda-(\cos k_1+\cos k_2)}\, \d k_1\right)^2\, \d k_2\\
			&\leq C\int_{-\pi}^{\pi}\frac{k_2^2}{(2+\lambda-\cos k_2)^2-1}\, \d k_2\\
			&\leq C\int_{-\pi}^{\pi}\frac{k_2^2}{(2+\lambda)(\lambda+k_2^2)}\, \d k_2\leq C,
		\end{align*}
		where in the first line, we use \eqref{eq.glambda}, in the second line, we use the standard integration formula $\int_{-\pi}^{\pi}\frac{\d k}{a-\cos k}=\frac{2\pi}{\sqrt{a^2-1}}$ , and in the third line, we use the fact that $1-\cos x\geq \tfrac{2}{\pi^2}x^2$ for $x\in[-\pi,\pi]$ again.
		
		Combining the above two estimates, we cam obtain the desired result \eqref{eq.H1FixedOnePoint}.
		
		Combining \eqref{eq.H1FixedOnePoint}, Proposition~\ref{prop.SobolevNorm}, and \eqref{eq.GreenH1Deg2} in Proposition~\ref{prop.GreenEstimateDeg2}, we can conclude:
		\begin{multline*}
			\sum_{z\in\Z_m}\left(\sum_{z'\in\Z_m}\left(\tilde{\D}_{(r)}\bar{g}_\lambda(z,z')\right)^2\right)^2\\
			\leq C r^8\sum_{x\in\Zo}\sum_{y\in\Zo}\left(\left(\D_1\bar{g}_\lambda(x,y)\right)^2+\left(\D_1\D_1\bar{g}_\lambda(x,y)\right)^2+\left(\D_1\D_2\bar{g}_\lambda(x,y)\right)^2\right)\\
			\leq C r^8 \left(\log(\lambda^{-1})+1\right).
		\end{multline*}
		The second term of $(\text{a})$ in \eqref{prop.QuarticThreeBlock} can be treated similarly. Combining them, we can conclude $(\text{a})$ in \eqref{prop.QuarticThreeBlock}. 
		
		For simplicity, we will use the following notations for the rest of the proof:
		\begin{equation*}
			\alpha_{(r)}(z):=\sum_{z'\in\Z_m}\left(\tilde{\D}_{(r)}\bar{g}_\lambda(z,z')\right)^2,\quad \beta_{(r)}(z):=\sum_{z'\in\Z_m}\left(\tilde{\D}_{(r)}\bar{g}_\lambda(z',z)\right)^2.
		\end{equation*}
		\begin{equation*}
			\kappa_{(r)}:=\sum_{z\in\Z_m}\left(\tilde{\D}_{(r)}\bar{g}_\lambda(z,z)\right)^2.
		\end{equation*}
		Using $(\text{a})$ in \eqref{prop.QuarticThreeBlock} we can immediately obtain
		\begin{equation}\label{eq.AlphaBetaBound}
			\sum_{z\in\Z_m}\left(\left(\alpha_{(r)}(z)\right)^2+\left(\beta_{(r)}(z)\right)^2\right)\leq C r^8 \left(\log(\lambda^{-1})+1\right).
		\end{equation}
		
		Using Cauchy--Schwarz inequality and combining Proposition~\ref{prop.SobolevNorm} and Proposition~\ref{prop.GreenEstimateDeg2}, we also have:
		\begin{equation}\label{eq.AlphaBetaBound2}
			\sum_{z\in\Z_m}\left(\alpha_{(r)}(z)+\beta_{(r)}(z)\right)\leq Cr^4\left(\log(\lambda^{-1})+1\right).
		\end{equation}
		
		By Cauchy--Schwarz inequality,  \eqref{eq.H1FixedOnePoint} and symmetry, we have
		\begin{equation}\label{eq.AlphaBetaLInfty}
			\sup_{z\in\Zo}\max\{\alpha_{(r)}(z),\beta_{(r)}(z)\}\leq Cr^4.
		\end{equation}
		Using Cauchy--Schwarz inequality and combining \eqref{prop.H1Near} ($s=2r,\ell=\lfloor\sqrt{\log(\lambda^{-1})}\rfloor\vee 1$) in Proposition~\ref{prop.SobolevNorm} and  Proposition~\ref{prop.GreenEstimateDeg2}, we have
		\begin{align*}
			&\sum_{z\in\Z_m}\left(\tilde{\D}_{(r)}\bar{g}_\lambda(z,z)\right)^2\\
			&\quad\leq Cr^3\sum_{x\in\Zo}\sum_{|y-x|\leq 2r}\left(\left(\D_1\bar{g}_\lambda(x,y)\right)^2+\left(\D_1\D_1\bar{g}_\lambda(x,y)\right)^2+\left(\D_1\D_2\bar{g}_\lambda(x,y)\right)^2\right)\\
			&\quad\leq Cr^4\left(\left(\sqrt{\log(\lambda^{-1})}\right)^{-1}\log(\lambda^{-1})+\sqrt{\log(\lambda^{-1})}+1\right)\leq Cr^4\left(\log(\lambda^{-1})^{\frac{1}{2}}+1\right),
		\end{align*}
		and therefore we have the following bound for $\kappa_{(r)}$\:
		\begin{equation}\label{eq.KappaBound}
			\kappa_{(r)}\leq Cr^4\left(\log(\lambda^{-1})^{\frac{1}{2}}+1\right).
		\end{equation}
		
		Now we start to prove $(\text{b})$ in \eqref{prop.QuarticThreeBlock}. It suffices to obtain the same bound for the following six terms $Q_i$:
		\begin{align*}
			Q_1&=\sum_{z_i\in\Z_m}\tilde{\D}_{(r)}\bar{g}_\lambda(z_1,z_1)\tilde{\D}_{(r)}\bar{g}_\lambda(z_2,z_2)\tilde{\D}_{(r)}\bar{g}_\lambda(z_1,z_3)\tilde{\D}_{(r)}\bar{g}_\lambda(z_2,z_3),\\
			Q_2&=\sum_{z_i\in\Z_m}\tilde{\D}_{(r)}\bar{g}_\lambda(z_1,z_1)\tilde{\D}_{(r)}\bar{g}_\lambda(z_2,z_1)\tilde{\D}_{(r)}\bar{g}_\lambda(z_1,z_3)\tilde{\D}_{(r)}\bar{g}_\lambda(z_2,z_3),\\
			Q_3&=\sum_{z_i\in\Z_m}\tilde{\D}_{(r)}\bar{g}_\lambda(z_1,z_2)\tilde{\D}_{(r)}\bar{g}_\lambda(z_2,z_2)\tilde{\D}_{(r)}\bar{g}_\lambda(z_1,z_3)\tilde{\D}_{(r)}\bar{g}_\lambda(z_2,z_3),\\
			Q_4&=\sum_{z_i\in\Z_m}\tilde{\D}_{(r)}\bar{g}_\lambda(z_1,z_2)\tilde{\D}_{(r)}\bar{g}_\lambda(z_2,z_1)\tilde{\D}_{(r)}\bar{g}_\lambda(z_1,z_3)\tilde{\D}_{(r)}\bar{g}_\lambda(z_2,z_3),\\
			Q_5&=\sum_{z_i\in\Z_m}\left(\tilde{\D}_{(r)}\bar{g}_\lambda(z_1,z_3)\right)^2\left(\tilde{\D}_{(r)}\bar{g}_\lambda(z_2,z_1)\right)^2,\\
			Q_6&=\sum_{z_i\in\Z_m}\left(\tilde{\D}_{(r)}\bar{g}_\lambda(z_1,z_3)\right)^2\left(\tilde{\D}_{(r)}\bar{g}_\lambda(z_2,z_2)\right)^2.
		\end{align*}
		By Cauchy--Schwarz inequality, we have
		\begin{equation*}
			Q_1 = \sum_{z_3\in\Z_m}\left(\sum_{z\in\Z_m}\tilde{\D}_{(r)}\bar{g}_\lambda(z,z)\tilde{\D}_{(r)}\bar{g}_\lambda(z,z_3)\right)^2\leq \sum_{z_3\in\Z_m}\kappa_{(r)}\beta_{(r)}(z_3).
		\end{equation*}
		By \eqref{eq.KappaBound} and \eqref{eq.AlphaBetaBound2}, we can obtain
		\begin{equation*}
			Q_1\leq Cr^8\left(\log(\lambda^{-1})^{\frac{3}{2}}+1\right).
		\end{equation*}
		
		For $Q_2$, we first apply Cauchy--Schwarz inequality for $z_3\in\Z_m$
		\begin{align*}
			Q_2&\leq \sum_{z_1,z_2\in\Z_m}\tilde{\D}_{(r)}\bar{g}_\lambda(z_1,z_1)\tilde{\D}_{(r)}\bar{g}_\lambda(z_2,z_1)\left(\alpha_{(r)}(z_1)\right)^{\frac{1}{2}}\left(\alpha_{(r)}(z_2)\right)^{\frac{1}{2}}\\
			&\leq \sum_{z_1\in\Z_m}\tilde{\D}_{(r)}\bar{g}_\lambda(z_1,z_1)\left(\alpha_{(r)}(z_1)\right)^{\frac{1}{2}}\left(\beta_{(r)}(z_1)\right)^{\frac{1}{2}}\left(\sum_{z\in\Z_m}\alpha_{(r)}(z)\right)^{\frac{1}{2}}\\
			&\leq \kappa_{(r)}^{\frac{1}{2}}\left(\sum_{z\in\Z_m}\alpha_{(r)}(z)\beta_{(r)}(z)\right)^{\frac{1}{2}}\left(\sum_{z\in\Z_m}\alpha_{(r)}(z)\right)^{\frac{1}{2}}\\
			&\leq \kappa_{(r)}^{\frac{1}{2}}\left(\sum_{z\in\Z_m}\left(\alpha_{(r)}(z)\right)^2\right)^{\frac{1}{4}}\left(\sum_{z\in\Z_m}\left(\beta_{(r)}(z)\right)^2\right)^{\frac{1}{4}}\left(\sum_{z\in\Z_m}\alpha_{(r)}(z)\right)^{\frac{1}{2}},
		\end{align*}
		where in the second inequality, we apply Cauchy--Schwarz for $z_2\in\Z_m$, in the third inequality, we apply Cauchy--Schwarz for $z_1\in\Z_m$, and in the last inequality, we apply Cauchy--Schwarz for $z\in\Z_m$.
		
		Combining \eqref{eq.KappaBound}, \eqref{eq.AlphaBetaBound}, and \eqref{eq.AlphaBetaBound2}, we can obtain
		\begin{equation*}
			Q_2\leq Cr^8 \left(\log(\lambda^{-1})^{\frac{5}{4}}+1\right).
		\end{equation*}
		By symmetry, we also have the same bound for $Q_3$:
		\begin{equation*}
			Q_3\leq Cr^8 \left(\log(\lambda^{-1})^{\frac{5}{4}}+1\right).
		\end{equation*}
		
		For $Q_4$, we first apply Cauchy--Schwarz inequality for $z_3\in\Z_m$
		\begin{align*}
			Q_4&\leq \sum_{z_1,z_2\in\Z_m}\tilde{\D}_{(r)}\bar{g}_\lambda(z_1,z_2)\tilde{\D}_{(r)}\bar{g}_\lambda(z_2,z_1)\left(\alpha_{(r)}(z_1)\right)^{\frac{1}{2}}\left(\alpha_{(r)}(z_2)\right)^{\frac{1}{2}}\\
			&\leq \left(\sum_{z_1,z_2\in\Z_m}\left(\tilde{\D}_{(r)}\bar{g}_\lambda(z_1,z_2)\right)^2\alpha_{(r)}(z_1)\right)^{\frac{1}{2}}\left(\sum_{z_1,z_2\in\Z_m}\left(\tilde{\D}_{(r)}\bar{g}_\lambda(z_2,z_1)\right)^2\alpha_{(r)}(z_2)\right)^{\frac{1}{2}}\\
			&=\sum_{z\in\Z_m}\left(\alpha_{(r)}(z)\right)^2,
		\end{align*}
		where in the second inequality, we apply Cauchy--Schwarz for $z_1,z_2\in\Z_m$. Therefore, using \eqref{eq.AlphaBetaBound}, we have
		\begin{equation*}
			Q_4\leq Cr^8\left(\log(\lambda^{-1})+1\right).
		\end{equation*}
		For $Q_5$, we can apply Cauchy--Schwarz inequality and \eqref{eq.AlphaBetaBound}  to obtain
		\begin{multline*}
			Q_5=\sum_{z_1\in\Z_m}\alpha_{(r)}(z_1)\beta_{(r)}(z_1)\\
			\leq \left(\sum_{z\in\Z_m}\left(\alpha_{(r)}(z)\right)^2\right)^{\frac{1}{2}}\left(\sum_{z\in\Z_m}\left(\beta_{(r)}(z)\right)^2\right)^{\frac{1}{2}}\leq Cr^8\left(\log(\lambda^{-1})+1\right).
		\end{multline*}
		For the last term $Q_6$, we use \eqref{eq.AlphaBetaBound2} and \eqref{eq.KappaBound},
		\begin{equation*}
			Q_6=\kappa_{(r)}\left(\sum_{z\in\Z_m}\alpha_{(r)}(z)\right)\leq Cr^8\left(\log(\lambda^{-1})^{\frac{3}{2}}+1\right).
		\end{equation*}
		
		In the last, we prove \eqref{prop.QuarticTwoBlock}, and it suffices to prove the same bound for the following four terms $Q_i$:
		\begin{align*}
			Q_7&=\sum_{z_i\in\Z_m}\left(\tilde{\D}_{(r)}\bar{g}_\lambda(z_1,z_1)\right)^2\left(\tilde{\D}_{(r)}\bar{g}_\lambda(z_2,z_2)\right)^2,\\
			Q_8&=\sum_{z_i\in\Z_m}\left(\tilde{\D}_{(r)}\bar{g}_\lambda(z_1,z_1)\right)^2\left(\tilde{\D}_{(r)}\bar{g}_\lambda(z_2,z_1)\right)^2,\\
			Q_9&=\sum_{z_i\in\Z_m}\left(\tilde{\D}_{(r)}\bar{g}_\lambda(z_1,z_2)\right)^2\left(\tilde{\D}_{(r)}\bar{g}_\lambda(z_2,z_2)\right)^2,\\
			Q_{10}&=\sum_{z_i\in\Z_m}\left(\tilde{\D}_{(r)}\bar{g}_\lambda(z_1,z_2)\right)^2\left(\tilde{\D}_{(r)}\bar{g}_\lambda(z_2,z_1)\right)^2.
		\end{align*}
		By computation and \eqref{eq.AlphaBetaBound}, \eqref{eq.AlphaBetaBound2}, \eqref{eq.AlphaBetaLInfty}, and \eqref{eq.KappaBound}, we can prove that
		\begin{align*}
			Q_7&=\kappa_{(r)}^2\leq Cr^8\left(\log(\lambda^{-1})+1\right),\\
			Q_8&=\sum_{z_1\in\Z_m}\left(\tilde{\D}_{(r)}\bar{g}_\lambda(z_1,z_1)\right)^2\beta_{(r)}(z_1)\leq Cr^4\kappa_{(r)}\leq Cr^8\left(\left(\log(\lambda^{-1})\right)^{\frac{1}{2}}+1\right),\\
			Q_9&=\sum_{z_2\in\Z_m}\left(\tilde{\D}_{(r)}\bar{g}_\lambda(z_2,z_2)\right)^2\beta_{(r)}(z_2)\leq Cr^8\left(\left(\log(\lambda^{-1})\right)^{\frac{1}{2}}+1\right).
		\end{align*}
		For $Q_{10}$, using \eqref{eq.H1FixedOnePoint}, we have
		\begin{equation*}
			\sup_{\lambda\in(0,1)}\sup_{x,y\in\Zo}|\D_1\bar{g}_\lambda(x,y)|\leq C,
		\end{equation*}
		and therefore by Cauchy--Schwarz inequality, we have
		\begin{equation*}
			\sup_{z_i\in\Z_m}\left(\tilde{\D}_{(r)}\bar{g}_\lambda(z_1,z_2)\right)^2\leq Cr^4, 
		\end{equation*}
		which implies
		\begin{equation*}
			Q_{10}\leq Cr^4\sum_{z_i\in\Z_m}\left(\tilde{\D}_{(r)}\bar{g}_\lambda(z_2,z_1)\right)^2=Cr^4\left(\sum_{z\in\Z_m}\alpha_{(r)}(z)\right)\leq Cr^8\left(\log(\lambda^{-1})+1\right).
		\end{equation*}
	\end{proof}
	
	\subsection{Two-scale expansion}
	In this section, we consider the following two-scale expansion on the $\HH_2$ space in $d=1$:
	\begin{equation}\label{eq.TwoscaleExDeg2}
		\tilde{G}_\lambda:=\bar{G}_\lambda+\sum_{z \in \Z_m}  [\D \bar{G}_\lambda]^z_{m} \phi^z_{m}.
	\end{equation}
	Since there is only one direction in $d=1$, we can omit the subscript $e_i$ and use $\phi_m^z$ to denote the only corrector. 
	Unlike the two-scale expansion on the $\HH_1$ space \eqref{eq.TwoScaleExpan}, the averaged slope $[\D \bar{G}_\lambda]^z_{m}$ is no longer a constant but a random variable in the $\HH_1$ space and the precise definition of the averaged slope $[\D \bar{G}_\lambda]^z_{m}$ is given by:
	\begin{equation}\label{eq.AveSlopeDeg2}
		[\D \bar{G}_\lambda]^z_{m} := \sum_{y\in (N_\r(z+\cu_m^+))^c}\left(\frac{2}{|\cu_m|}\sum_{x\in z+\cu_m}\D_1\bar{g}_\lambda(x,y)\right)\bar{\eta}_y.
	\end{equation}
	Recall that the set of vertices $N_\r(z+\cu_m^+)$ is defined in \eqref{eq.defNrLa}. There are many possible candidates for $[\D \bar{G}_\lambda]^z_{m}$, and we choose to sum over all vertices at least $\r+1$ away from $z+\cu_m^+$ in order to guarantee the independence between $[\D \bar{G}_\lambda]^z_{m}$ and $\{c_{x,x+1}\pi_{x,x+1}\phi_{m}^z,x\in z+\cu_m\}$.
	
	Notice that the center flux define in \eqref{eq.defFlux} is essentially 
	\begin{equation*}
		\g_{m,b}^z := c_b  \pi_b (\ell_{e} + \phi^z_{m}) - \DD(\rho)\pi_b \ell_{e},
	\end{equation*}
	where $\ell_e$ is the affine function \eqref{eq.defaffinefuc} in $d=1$:
	\begin{equation*}
		\ell_e:=\sum_{x\in\Zo}x\cdot\bar{\eta}_x.
	\end{equation*}
	
	We first recover the $L^2$ and flux estimate about the two scale expansion in parallel with \eqref{prop.L2Err} and \eqref{prop.FluxErr} in Proposition~\ref{prop.L2FluxErr}.
	
	\begin{proposition}\label{prop.L2FluxDeg2} The following $L^2$ and flux estimates hold for the two-scale expansion $\tilde{G}_\lambda$:
		\begin{enumerate}
			\item\label{prop.L2ErrDeg2} There exists a finite positive constant $C(c_{-},c_{+},\rho)$ such that the  two-scale expansion satisfies
			\begin{equation*}
				\norm{\tilde{G}_\lambda-\bar{G}_\lambda}^2_{L^2}\leq C3^{3m}\expec{\bar{G}_\lambda,-\Lb \bar{G}_\lambda}.
			\end{equation*}
			\item\label{prop.FluxErrDeg2} There exists an exponent $\alpha(c_{-},c_{+}, \r) > 0$ and a finite positive constant $C(c_{-},c_{+}, \r,\rho)$ such that the following estimate holds for any $V\in L^2$ satisfying $\expec{V,-\Lb V}<\infty$:
			\begin{equation*}
				\left|\expec{V,-\L\tilde{G}_\lambda+\Lb \bar{G}_\lambda}\right|\leq C\expec{V,-\Lb V}^{\frac{1}{2}}\left(3^{-\alpha m}\expec{\bar{G}_\lambda,-\Lb \bar{G}_\lambda}^{\frac{1}{2}}+3^{7m} \expec{-\Lb \bar{G}_\lambda,-\Lb \bar{G}_\lambda}^{\frac{1}{2}}\right).
			\end{equation*}
		\end{enumerate}
	\end{proposition}
	
	\begin{proof}
		Using the expression for $\tilde{G}_\lambda$ in \eqref{eq.TwoscaleExDeg2}, we have
		\begin{align*}
			\norm{\tilde{G}_\lambda-\bar{G}_\lambda}_{L^2}^2 &=\E_\rho\left[\left(\sum_{z\in\Z_m}[\D\bar{G}_\lambda]_m^z\phi_m^z\right)^2\right]\\
			&=\sum_{z\in\Z_m}\E_\rho\left[\left([\D\bar{G}_\lambda]_m^z\phi_m^z\right)^2\right]+\sum_{z,z'\in\Z_m, z\neq z'}\E_\rho\left[[\D\bar{G}_\lambda]_m^z\phi_m^z[\D\bar{G}_\lambda]_m^{z'}\phi_m^{z'}\right].
		\end{align*}
		Utilizing the independence between $[\D\bar{G}_\lambda]_m^z$ and $\phi_m^z$, we can first treat the diagonal terms:
		\begin{equation*}
			\E_\rho\left[\left([\D\bar{G}_\lambda]_m^z\phi_m^z\right)^2\right]=\E_\rho\left[\left([\D\bar{G}_\lambda]_m^z\right)^2\right]\E_\rho\left[\left(\phi_m^z\right)^2\right].
		\end{equation*}
		By Cauchy--Schwarz inequality and \eqref{eq.AveSlopeDeg2} we have
		\begin{equation*}
			\E_\rho\left[\left([\D\bar{G}_\lambda]_m^z\right)^2\right] \leq 4\chi(\rho)3^{-m}\sum_{y\in (N_\r(z+\cu_m^+))^c}\sum_{x\in z+\cu_m}\left(\D_1\bar{g}_\lambda(x,y)\right)^2,
		\end{equation*}
		and therefore combining \eqref{prop.H11D} in Proposition~\ref{prop.SobolevNorm}, the summation on $z\in\Z_m$ yields:
		\begin{align*}
			\sum_{z\in\Z_m}\E_\rho\left[\left([\D\bar{G}_\lambda]_m^z\right)^2\right]
			&\leq 4 \chi(\rho)3^{-m}\sum_{x,y\in\Zo}\left(\D_1\bar{g}_\lambda(x,y)\right)^2\\
			&= 2 \DD(\rho)^{-1}\chi(\rho)3^{-m}\expec{\bar{G}_\lambda,-\Lb\bar{G}_\lambda}.
		\end{align*}
		Combining \eqref{lem.CorrectorH1L2} in Lemma~\ref{lem.Homogenization} we can obtain:
		\begin{equation}\label{eq.L2Diagonal}
			\sum_{z\in\Z_m}\E_\rho\left[\left([\D\bar{G}_\lambda]_m^z\phi_m^z\right)^2\right]\leq C(c_{-},c_{+},\rho)3^{2m}\expec{\bar{G}_\lambda,-\Lb\bar{G}_\lambda}.
		\end{equation}
		In the following part about estimates for the non-diagonal terms, we regard the coefficient of the missing terms as zero for convenience:
		\begin{equation*}
			\D_1\bar{g}_\lambda(x,y)=0,\quad \forall x\in z+\cu_m,\, y\in N_\r(z+\cu_m^+).
		\end{equation*}
		For the non-diagonal terms, using the independence and by \eqref{lem.Correctorlocal} and \eqref{lem.CorrectorMean} in Lemma~\ref{lem.Homogenization}, we observe that all the possible non-vanishing terms are of the form:
		\begin{equation*}
			\E_\rho\left[\bar{\eta}_y\bar{\eta}_{y'}\phi_m^z\phi_m^{z'}\right],\quad y\in z'+\cu_m^-,\ y'\in z+\cu_m^-,
		\end{equation*} 
		with the coefficient
		\begin{equation*}
			\left(\frac{2}{|\cu_m|}\sum_{x\in z+\cu_m}\D_1\bar{g}_\lambda(x,y)\right)\left(\frac{2}{|\cu_m|}\sum_{x\in z'+\cu_m}\D_1\bar{g}_\lambda(x,y')\right).
		\end{equation*}
		By Cauchy--Schwarz inequality, \eqref{lem.Correctorlocal} and \eqref{lem.CorrectorH1L2} in Lemma~\ref{lem.Homogenization}, we have
		\begin{equation*}
			\left|\E_\rho\left[\bar{\eta}_y\bar{\eta}_{y'}\phi_m^z\phi_m^{z'}\right]\right|=\left|\E_\rho\left[\bar{\eta}_y\phi_m^{z'}\right]\E_\rho\left[\bar{\eta}_{y'}\phi_m^{z}\right]\right|\leq C(c_{-},c_{+},\rho)3^{3m},
		\end{equation*}
		where the first equality is because we have
		\begin{equation*}
			\bar{\eta}_y\phi_m^{z'}\in \F_0(z'+\cu_m^-),\qquad \bar{\eta}_{y'}\phi_m^z\in \F_0(z+\cu_m^-).
		\end{equation*}
		
		We can use AM--GM inequality to obtain:
		\begin{multline*}
			\left|\left(\frac{2}{|\cu_m|}\sum_{x\in z+\cu_m}\D_1\bar{g}_\lambda(x,y)\right)\left(\frac{2}{|\cu_m|}\sum_{x\in z'+\cu_m}\D_1\bar{g}_\lambda(x,y')\right)\right|\\
			\leq 2\cdot3^{-m}\left(\sum_{x\in z+\cu_m}\left(\D_1\bar{g}_\lambda(x,y)\right)^2 + \sum_{x\in z'+\cu_m}\left(\D_1\bar{g}_\lambda(x,y')\right)^2\right).
		\end{multline*}
		Therefore we have
		\begin{multline*}
			\left|\E_\rho\left[[\D\bar{G}_\lambda]_m^z\phi_m^z[\D\bar{G}_\lambda]_m^{z'}\phi_m^{z'}\right]\right|\\
			\leq C(c_{-},c_{+},\rho)3^{3m} \left(\sum_{y\in z'+\cu_m^-}\sum_{x\in z+\cu_m}\left(\D_1\bar{g}_\lambda(x,y)\right)^2 + \sum_{y'\in z+\cu_m^-} \sum_{x\in z'+\cu_m}\left(\D_1\bar{g}_\lambda(x,y')\right)^2\right).
		\end{multline*}
		Combining \eqref{prop.H11D} in Proposition~\ref{prop.SobolevNorm}, the summation on $z,z'\in\Z_m,z\neq z'$ yields:
		\begin{equation}\label{eq.L2NonDiagonal}
			\begin{split}
				&\sum_{z,z'\in\Z_m, z\neq z'}\left|\E_\rho\left[[\D\bar{G}_\lambda]_m^z\phi_m^z[\D\bar{G}_\lambda]_m^{z'}\phi_m^{z'}\right]\right|\\
				&\leq C(c_{-},c_{+},\rho)3^{3m}\sum_{x,y\in\Zo}\left(\D_1\bar{g}_\lambda(x,y)\right)^2\\
				&\leq C(c_{-},c_{+},\rho)3^{3m}\expec{\bar{G}_\lambda,-\Lb\bar{G}_\lambda}.
			\end{split}
		\end{equation}
		Combining \eqref{eq.L2Diagonal} and \eqref{eq.L2NonDiagonal} we can conclude
		\eqref{prop.L2ErrDeg2}.
		
		The proof of \eqref{prop.FluxErrDeg2} is similar to \cite[Lemma 4.7]{gu2025relaxation} and can be divided into 4 steps.
		
		\textit{Step~0: decomposition.} We start from the terms involving $\L$
		\begin{equation*}
			\expec{V,-\L\tilde{G}_\lambda}=\frac{1}{2}\sum_{z\in\Z_m}\sum_{x\in z+\cu_m}\expec{c_{x,x+1}(\pi_{x,x+1}V),\pi_{x,x+1}\tilde{G}_\lambda}.
		\end{equation*}
		Using \eqref{eq.TwoscaleExDeg2}, we can express $\pi_{x,x+1}\tilde{G}_\lambda$ as:
		\begin{equation}\label{eq.PibTwoScale}
			\pi_{x,x+1}\tilde{G}_\lambda = \pi_{x,x+1}\bar{G}_\lambda + \sum_{z'\in\Z_m}\pi_{x,x+1}\left([\D\bar{G}_\lambda]_m^{z'}\phi_m^{z'}\right).
		\end{equation}
		By chain rule and \eqref{lem.Correctorlocal} in Lemma~\ref{lem.Homogenization}, for every $x\in z+\cu_m$, we have
		\begin{equation*}
			\sum_{z'\in\Z_m}\pi_{x,x+1}\left([\D\bar{G}_\lambda]_m^{z'}\phi_m^{z'}\right)=\sum_{z'\in\Z_m}\pi_{x,x+1}\left([\D\bar{G}_\lambda]_m^{z'}\right)\phi_m^{z'}+[\D\bar{G}_\lambda]_m^z\pi_{x,x+1}\phi_m^z.
		\end{equation*}
		We denote the oscillation of the averaged slope for a bond $b$ paired with the corrector $\phi$ by $w_{\D,\phi}(b)$:
		\begin{equation}\label{eq.defOscillation}
			w_{\D,\phi}(b):=\sum_{z\in\Z_m}\pi_{b}\left([\D\bar{G}_\lambda]_m^{z}\right)\phi_m^{z}.
		\end{equation}
		
		We then aim to make  the centered flux $\g_{m, b}^z$ appear in \eqref{eq.PibTwoScale}. 
		For every $x\in \Zo$, we have 
		\begin{align*}
			\pi_{x,x+1}\ell_e=\bar{\eta}_x-\bar{\eta}_{x+1}.
		\end{align*}
		We denote the error of the slope to be $[\D_{err}\bar{G}_\lambda(b)]_{m}^{z}$ for a cube $z+\cu_m$ and a bond $\{x,x+1\}=b\in\ov{(z+\cu_m)^*}$:
		\begin{equation}\label{eq.defDErrGBar}
			[\D_{err}\bar{G}_\lambda(b)]_{m}^{z} := 2\sum_{y\in \Zo}\D_1 \bar{g}_\lambda(x,y)\bar{\eta}_y-[\D\bar{G}_\lambda]_{m}^{z},\quad\forall \{x,x+1\}=b\in\ov{(z+\cu_m)^*}.
		\end{equation} 
		and therefore we can obtain:
		\begin{equation}\label{eq.PibGBarGTilde}
			\begin{aligned}
				\pi_b\bar{G}_\lambda &= [\D\bar{G}_\lambda]_m^{z}\pi_{b}\ell_{e}+[\D_{err}\bar{G}_\lambda(b)]_{m}^{z}\pi_{b}\ell_{e},\\
				\pi_{b}\tilde{G}_\lambda &= [\D\bar{G}_\lambda]_{m}^{z}\pi_{b}(\ell_{e}+\phi_m^z)+[\D_{err}\bar{G}_\lambda(b)]_{m}^{z}\pi_{b}\ell_{e}+w_{\D,\phi}(b).
			\end{aligned}
		\end{equation}
		
		We further apply $c_{b}$, which yields
		\begin{align*}
			c_{b}\pi_{b} \tilde G_\lambda	&= [\D\bar{G}_\lambda]_{m}^{z} (c_{b}\pi_{b}(\ell_{e} + \phi^{z}_{m})) + [\D_{err}\bar{G}_\lambda(b)]_{m}^{z}c_{b}\pi_{b}\ell_{e}+c_{b}w_{\D,\phi}(b) \\
			&= [\D\bar{G}_\lambda]_{m}^{z}\Ll(c_b  \pi_b (\ell_{e} + \phi^z_{m} ) - \DD(\rho)\pi_b \ell_{ e}\Rr)\\
			&\qquad + [\D_{err}\bar{G}_\lambda(b)]_{m}^{z}c_b\pi_{b}\ell_{e}+c_bw_{\D,\phi}(b) + [\D\bar{G}_\lambda]_{m}^{z}\DD(\rho)\pi_{b}\ell_{e}.
		\end{align*}
		Here in the second line, we make $\g^z_{m,b}=c_b  \pi_b (\ell_{e} + \phi^z_{m} ) - \DD(\rho)\pi_b \ell_{e}$ appear as desired. 
		We can also compute the terms involving $\Lb$ using \eqref{eq.PibGBarGTilde}:
		\begin{equation*}
			\expec{V,-\Lb\bar{G}_\lambda}=\frac{1}{2}\DD(\rho)\sum_{z\in\Z_m}\sum_{b\in \ov{(z + \cu_m)^*}}\expec{\pi_b V,[\D\bar{G}_\lambda]_m^{z}\pi_{b}\ell_{e}+[\D_{err}\bar{G}_\lambda(b)]_{m}^{z}\pi_{b}\ell_{e}}.
		\end{equation*}
		Therefore,  we conclude that
		\begin{align}\label{eq.fluxdecomshort}
			\bracket{V(-\L\tilde G_\lambda + \Lb  \bar G_\lambda)}  &= \mathbf{F.1} + \mathbf{F.2} + \mathbf{F.3},
		\end{align}
		where the three terms are
		\begin{equation}\label{eq.fluxdecomlong}
			\begin{split}
				\mathbf{F.1} &:= \frac{1}{2}\sum_{z \in \Z_m} \sum_{b \in \ov{(z + \cu_m)^*}}\bracket{ \pi_{b} V,  [\D\bar{G}_\lambda]_{m}^{z}\g^z_{m,b}},\\
				\mathbf{F.2} &:= \frac{1}{2}\sum_{z \in \Z_m} \sum_{b \in \ov{(z + \cu_m)^*}}\bracket{ \pi_{b} V, [\D_{err}\bar{G}_\lambda(b)]_{m}^{z}(c_b\pi_{b}\ell_{e}-\DD(\rho)\pi_{b}\ell_{e})}, \\
				\mathbf{F.3} &:= \frac{1}{2}\sum_{z \in \Z_m} \sum_{b \in \ov{(z + \cu_m)^*}} \bracket{\pi_{b} V, c_b w_{\D,\phi}(b)}.
			\end{split}
		\end{equation}
		These three terms have their own interpretations. The term $\mathbf{F.1}$ is the main part of the flux replacement. The term $\mathbf{F.2}$ is the error to fix the local slope. The term $\mathbf{F.3}$ is the error from the randomness of the averaged slope. A similar decomposition of two-scale expansion can be found in the previous work \cite[eq.(4.10)]{gu2024quantitative}, \cite[eq.(4.43)]{gu2025relaxation}.
		
		In the following paragraphs, we treat the three terms separately.
		\medskip
		
		\textit{Step~1: term~$\mathbf{F.1}$ as the error in flux replacement.}
		For this term, we make the centered flux $\g_{m,b}^z$ appear. Moreover, as the averaged slope $[\D\bar{G}_\lambda]_{m}^{z}$ does not depend on the configuration in $(z+\cu_m)^+$, we have 
		\begin{equation*}
			\pi_b V\cdot [\D\bar{G}_\lambda]_{m}^{z}=\pi_b\Ll([\D\bar{G}]_m^z\cdot V\Rr),\quad \forall b \in \ov{(z + \cu_m)^*}.
		\end{equation*}
		Notice that the centered flux $\g_{m,b}^z$ is a local function:
		\begin{equation*}
			\g^z_{m,b}=c_b  \pi_b (\ell_{e} + \phi^z_{m} ) - \DD(\rho)\pi_b \ell_{e}\in \F_0(N_\r(z+\cu_m^+)),\quad\forall b\in \ov{(z+\cu_m)^*},
		\end{equation*}
		and therefore by taking conditional expectation with respect to $\fil_{N_\r(z+\cu_m^+)}$ we can obtain:
		\begin{align*}
			\expec{\pi_b V, [\D\bar{G}_\lambda]_m^z\g_{m,b}^z}
			&= \E_\rho\Ll[\g_{m,b}^z\E_\rho\Ll[\pi_b\left.\left([\D\bar{G}_\lambda]_m^z\cdot V\right)\right|\fil_{N_\r(z+\cu_m^+)}\Rr]\Rr]\\
			&=\E_\rho\Ll[\g_{m,b}^z\pi_b\left(\E_\rho\Ll[\left.\left([\D\bar{G}_\lambda]_m^z\cdot V\right)\right|\fil_{N_\r(z+\cu_m^+)}\Rr]\right)\Rr],
		\end{align*}
		here in the second line, the operator $\pi_b$ commutes with the conditional expectation because $b\in \ov{(z+\cu_m)^*}$ and $z+\cu_m^+$ belongs to $N_\r(z+\cu_m^+)$.

		For $v=\E_\rho\Ll[\left.\left([\D\bar{G}_\lambda]_m^z\cdot V\right)\right|\fil_{N_\r(z+\cu_m^+)}\Rr]$, we apply the flux cancellation \eqref{lem.FluxReplacement} in Lemma~\ref{lem.Homogenization} to obtain:
		\begin{align*}
			\Ll\vert \sum_{b \in \ov{(z + \cu_m)^*}}\bracket{\pi_{b} V, [\D\bar{G}_\lambda]_{m}^{z}\g_{m,b}^z} \Rr\vert 
			&= \Ll\vert \sum_{b \in \ov{(z + \cu_m)^*}} \bracket{\pi_{b} v,\g_{m,b}^z}\Rr\vert\\
			&\leq C 3^{-\alpha m} \vert \cu_m \vert^{\frac{1}{2}}\Ll(\sum_{b \in \ov{(z+\cu_m)^*} }\E_\rho\Ll[\Ll(\pi_b v\Rr)^2\Rr]\Rr)^{\frac{1}{2}}.
		\end{align*}
		We next use Cauchy--Schwarz inequality for conditional expectation to analyze $\E_\rho\left[(\pi_b v)^2\right]$ for our specifically chosen $v$:
		\begin{align*}
			\E_\rho[(\pi_b v)^2]
			& =\E_\rho\left[\left(\E_\rho\left.\left[\left(\pi_b V\right)\cdot [\D\bar{G}_\lambda]_m^z\right|\fil_{N_\r(z+\cu_m^+)}\right]\right)^2\right]\\
			&\leq \E_\rho\left[\E_\rho\left.\left[(\pi_b V)^2\right|\fil_{N_\r(z+\cu_m^+)}\right]\E_\rho\left.\left[\left([\D\bar{G}_\lambda]_m^z\right)^2\right|\fil_{N_\r(z+\cu_m^+)}\right]\right]\\
			&=\E_\rho[\left(\pi_b V\right)^2]\E_\rho\left[\left([\D\bar{G}_\lambda]_m^z\right)^2\right],
		\end{align*}
		where in the last line, we use the fact that $[\D\bar{G}_\lambda]_m^z\in \F_0\left(\left(N_\r(z+\cu_m^+)\right)^c\right)$ and therefore
		\begin{equation*}
			\E_\rho\left.\left[\left([\D\bar{G}_\lambda]_m^z\right)^2\right|\fil_{N_\r(z+\cu_m^+)}\right]=\E_\rho\left[\left([\D\bar{G}_\lambda]_m^z\right)^2\right].
		\end{equation*}
		We can use Cauchy--Schwarz inequality again to obtain an upper bound for $\E_\rho\left[\left([\D\bar{G}_\lambda]_m^z\right)^2\right]$:
		\begin{equation*}
			\E_\rho\left[\left([\D\bar{G}_\lambda]_m^z\right)^2\right]\leq 4\chi(\rho)\vert \cu_m \vert^{ -1}\sum_{y\in\left(N_\r(z+\cu_m^+)\right)^c}\sum_{x\in z+\cu_m}\left(\D_1 \bar{g}_\lambda(x,y)\right)^2.
		\end{equation*}
		The volume factor $\vert \cu_m \vert^{-\frac{1}{2}}$ compensates in the product of two estimates above. We then apply Jensen's inequality and \eqref{prop.H11D} in Proposition~\ref{prop.SobolevNorm} to obtain
		\begin{align}\label{eq.F.1}
			\vert \mathbf{F.1} \vert \leq C3^{-\alpha m}\expec{V,-\Lb V}^{\frac{1}{2}}\expec{\bar{G}_\lambda,-\Lb \bar{G}_\lambda}^{\frac{1}{2}}.
		\end{align}
		
		\medskip
		
		\textit{Step~2: term~$\mathbf{F.2}$ as the error to fix the slope.}
		We notice that for one bond $b=\{x,x+1\}$,
		\begin{equation*}
			|c_b\pi_{b}\ell_e-\DD(\rho)\pi_b\ell_e|=|(c_b-\DD(\rho))(\bar{\eta}(x)-\bar{\eta}(x+1))|\leq c_++\DD(\rho),\quad\forall \eta\in\X.
		\end{equation*}
		We apply at first Cauchy-Schwarz inequality and obtain
		\begin{equation*}
			\vert \mathbf{F.2} \vert^2\leq C\left(\sum_{b\in(\Zo)^*}\E_\rho\left[\left(\pi_b V\right)^2\right]\right)\left(\sum_{z\in\Z_m}\sum_{b\in\ov{(z+\cu_m)^*}}\E_\rho\left[\left([\D_{err}\bar{G}_\lambda(b)]_m^z\right)^2\right]\right).
		\end{equation*}
		Recall that $[\D_{err}\bar{G}_\lambda(b)]_m^z$ is defined in \eqref{eq.defDErrGBar}:
		\begin{equation*}
			[\D_{err}\bar{G}_\lambda(b)]_{m}^{z} := 2\sum_{y\in \Zo}\D_1 \bar{g}_\lambda(x,y)\bar{\eta}_y-[\D\bar{G}_\lambda]_{m}^{z},\quad\forall \{x,x+1\}=b\in\ov{(z+\cu_m)^*},
		\end{equation*}
		and therefore for one bond $b=\{x,x+1\}$, we can decompose it into two terms:
		\begin{multline*}
			[\D_{err}\bar{G}_\lambda(b)]_{m}^{z}=2\sum_{y\in N_\r(z+\cu_m^+)}\D_1 \bar{g}_\lambda(x,y)\bar{\eta}_y\\
			+2\sum_{y\in \left(N_\r(z+\cu_m^+)\right)^c}\left(\D_1\bar{g}_\lambda(x,y)-\frac{1}{|\cu_m|}\sum_{x'\in z+\cu_m}\D_1\bar{g}_\lambda(x',y)\right)\bar{\eta}_y.
		\end{multline*}
		By orthogonality of $\{\bar{\eta}_{x}\}_{x\in\Zo}$ we have
		\begin{multline*}
			\E_\rho\left[\left([\D_{err}\bar{G}_\lambda(b)]_{m}^{z}\right)^2\right]\leq 4\chi(\rho)\bigg(\sum_{y\in N_\r(z+\cu_m^+)}\left(\D_1 \bar{g}_\lambda(x,y)\right)^2\\
			+\sum_{y\in \left(N_\r(z+\cu_m^+)\right)^c}\bigg(\D_1\bar{g}_\lambda(x,y)-\frac{1}{|\cu_m|}\sum_{x'\in z+\cu_m}\D_1\bar{g}_\lambda(x',y)\bigg)^2\bigg).
		\end{multline*}
		For the first term, we use \eqref{prop.H1Near} in Proposition~\ref{prop.SobolevNorm} with $\ell=\lfloor3^{2\alpha m}(\r+3^m+1)\rfloor+1$ to obtain:
		\begin{equation*}
			\sum_{z\in\Z_m}\sum_{x\in z+\cu_m}\sum_{y\in N_\r(z+\cu_m^+)}\left(\D_1 \bar{g}_\lambda(x,y)\right)^2\leq C\left(3^{-2\alpha m}\expec{\bar{G}_\lambda,-\Lb\bar{G}_\lambda}+3^{4m}\expec{-\Lb\bar{G}_\lambda,-\Lb\bar{G}_\lambda}\right).
		\end{equation*}
		For the second term, we first apply Cauchy--Schwarz inequality to obtain
		\begin{equation*}
			\bigg(\D_1\bar{g}_\lambda(x,y)-\frac{1}{|\cu_m|}\sum_{x'\in z+\cu_m}\D_1\bar{g}_\lambda(x',y)\bigg)^2\leq C3^m\sum_{x'\in z+\cu_m}\left(\D_1\D_1\bar{g}_\lambda(x',y)\right)^2.
		\end{equation*}
		We then use \eqref{prop.pibpib'} in Proposition~\ref{prop.SobolevNorm} to obtain
		\begin{multline*}
			\sum_{z\in\Z_m}\sum_{x\in z+\cu_m}\sum_{y\in \left(N_\r(z+\cu_m^+)\right)^c}\bigg(\D_1\bar{g}_\lambda(x,y)-\frac{1}{|\cu_m|}\sum_{x'\in z+\cu_m}\D_1\bar{g}_\lambda(x',y)\bigg)^2\\
			\leq C 3^{2m}\expec{-\Lb\bar{G}_\lambda,-\Lb\bar{G}_\lambda}.
		\end{multline*}
		Therefore combining the above estimates, we have
		\begin{equation}\label{eq.F.2}
			\vert \mathbf{F.2} \vert\leq C\expec{V,-\Lb V}^{\frac{1}{2}}\left(3^{-\alpha m}\expec{\bar{G}_\lambda,-\Lb \bar{G}_\lambda}^{\frac{1}{2}}+3^{2m}\expec{-\Lb\bar{G}_\lambda,-\Lb\bar{G}_\lambda}^{\frac{1}{2}}\right)
		\end{equation}
		
		\medskip
		
		\textit{Step~3: term~$\mathbf{F.3}$ as the error for the randomness of the slope.} We first apply Cauchy-Schwarz inequality again and obtain
		\begin{equation*}
			\vert \mathbf{F.3} \vert^2 \leq C\left(\sum_{b\in(\Zo)^*}\E_\rho\left[\left(\pi_b V\right)^2\right]\right)\left(\sum_{b\in(\Zo)^*}\E_\rho\left[\left(c_b w_{\D,\phi}(b)\right)^2\right]\right).
		\end{equation*}
		Recall that $w_{\D,\phi}(b)$ is defined in \eqref{eq.defOscillation}:
		\begin{equation*}
			w_{\D,\phi}(b):=\sum_{z\in\Z_m}\pi_{b}\left([\D\bar{G}_\lambda]_m^{z}\right)\phi_m^{z}.
		\end{equation*}
		By the ellipticity of the jump rate $c_b$, we have
		\begin{equation*}
			\E_\rho\left[\left(c_b w_{\D,\phi}(b)\right)^2\right]\leq c_{+}^2 \E_\rho\left[\left(\sum_{z\in\Z_m}\pi_b\left([\D\bar{G}_\lambda]_m^z\right)\phi_m^z\right)^2\right].
		\end{equation*}
		We assume $3^m>\r$ and without loss of generality, we consider $\{x,x+1\}=b\in(\cu_m)^*$. By Cauchy--Schwarz inequality, we can obtain
		\begin{multline*}
			\E_\rho\left[\left(\sum_{z\in\Z_m}\pi_b\left([\D\bar{G}_\lambda]_m^z\right)\phi_m^z\right)^2\right]\\
			\leq 4\Ll(\sum_{z\in\Z_m,|z|\leq 3^m}\E_\rho\left[\left(\pi_b\left([\D\bar{G}_\lambda]_m^z\right)\phi_m^z\right)^2\right]+\E_\rho\left[\left(\sum_{z\in\Z_m,|z|>3^m}\pi_b\left([\D\bar{G}_\lambda]_m^z\right)\phi_m^z\right)^2\right]\Rr).
		\end{multline*}
		For the first term, we first notice that
		\begin{equation*}
			\pi_b[\D\bar{G}_\lambda]_m^0=0,
		\end{equation*}
		and therefore combining the $L^\infty$ estimate about the corrector $\phi_m^z$, \eqref{lem.CorrectorLp} in Proposition~\ref{lem.Homogenization}, we can obtain:
		\begin{equation*}
			\sum_{z\in\Z_m,|z|\leq 3^m}\E_\rho\left[\left(\pi_b\left([\D\bar{G}_\lambda]_m^z\right)\phi_m^z\right)^2\right]\leq C m^2\cdot 3^{6m}\sum_{z=\pm3^{m}}\E_\rho\left[\left(\pi_b[\D\bar{G}_\lambda]_m^z\right)^2\right].
		\end{equation*}
		Recall the averaged slope $[\D\bar{G}_\lambda]_m^z$ is defined in \eqref{eq.AveSlopeDeg2}:
		\begin{equation*}
			[\D \bar{G}_\lambda]^z_{m} := \sum_{y\in (N_\r(z+\cu_m^+))^c}\left(\frac{2}{|\cu_m|}\sum_{x'\in z+\cu_m}\D_1\bar{g}_\lambda(x',y)\right)\bar{\eta}_y.
		\end{equation*}
		Therefore, we can compute $\pi_b [\D\bar{G}_\lambda]_m^z$ explicitly as:
		\begin{equation}\label{eq.PibDG}
			\pi_b [\D\bar{G}_\lambda]_m^z=
			\begin{cases}
				\quad 0, &b\subset N_\r(z+\cu_{m}^+),\\
				\hfil \frac{2}{|\cu_m|}\sum_{x'\in z+\cu_m}\D_1\D_2\bar{g}_\lambda(x',x)(\bar{\eta}(x)-\bar{\eta}(x+1)), &b\subset (N_\r(z+\cu_m^+))^c,\\
				\hfil \frac{2}{|\cu_m|}\sum_{x'\in z+\cu_m}\D_1\bar{g}_\lambda(x',x)(\bar{\eta}(x+1)-\bar{\eta}(x)), &b\cap N_\r(z+\cu_m^+)=\{x\},\\
				\frac{2}{|\cu_m|}\sum_{x'\in z+\cu_m}\D_1\bar{g}_\lambda(x',x+1)(\bar{\eta}(x)-\bar{\eta}(x+1)), &b\cap N_\r(z+\cu_m^+)=\{x+1\}.
			\end{cases}
		\end{equation}
		Taking expectation with respect to $\P_\rho$ and by Cauchy--Schwarz inequality we can obtain
		\begin{multline}\label{eq.PibDNear}
			\sum_{b\in\ov{(\cu_m)^*}}\sum_{z=\pm 3^m}\E_\rho\left[\left(\pi_b[\D\bar{G}_\lambda]_m^z\right)^2\right]\\
			\leq C3^{-m}\sum_{z=\pm 3^{m}}\sum_{x'\in z+\cu_m}\left(\sum_{x\in  N_\r(z+\cu_m^+)}\left(\D_1\bar{g}_\lambda(x',x)\right)^2+\sum_{x\in \cu_m}\left(\D_1\D_2\bar{g}_\lambda(x',x)\right)^2\right).
		\end{multline}
		For any $|z|> 3^m$, we have
		\begin{equation*}
			\pi_b[\D\bar{G}_\lambda]_m^z=\frac{2}{|\cu_m|}\sum_{x'\in z+\cu_m}\D_1\D_2\bar{g}_\lambda(x',x)(\bar{\eta}(x)-\bar{\eta}(x+1))\in\F_0(\cu_m^+).
		\end{equation*}
		Therefore, for the summation over $|z|> 3^m$, we can use the $L^\infty$ estimate about the corrector $\phi_m^z$, \eqref{lem.CorrectorLp} in Proposition~\ref{lem.Homogenization} and Cauchy--Schwarz inequality to obtain
		\begin{equation}\label{eq.PibDFar}
			\begin{aligned}
				\E_\rho\left[\left(\sum_{z\in\Z_m,|z|>3^m}\pi_b\left([\D\bar{G}_\lambda]_m^z\right)\phi_m^z\right)^2\right]
				& = \sum_{z\in \Z_m,|z|>3^m}\E_\rho\left[\left(\pi_b[\D\bar{G}_\lambda]_m^z\right)^2(\phi_m^z)^2\right]\\
				& \leq Cm^2\cdot 3^{6m}\sum_{z\in\Z_m,|z|>3^m}\E_\rho\left[\left(\pi_b[\D\bar{G}_\lambda]_m^z\right)^2\right]\\
				& \leq C m^2 3^{5m}\sum_{|x'|>3^m}\left(\D_1\D_2\bar{g}_\lambda(x',x)\right)^2.
			\end{aligned}
		\end{equation}
		Here, the cross terms vanish in the first line  because $\pi_b[\D\bar{G}_\lambda]_m^z\in \F_0(\cu_m^+)$, together with the locality and mean zero property of the corrector $\phi_m^z$ (see \eqref{lem.Correctorlocal} and \eqref{lem.CorrectorMean} in Lemma~\ref{lem.Homogenization}).
		
		Combining \eqref{eq.PibDNear} and \eqref{eq.PibDFar} we have
		\begin{multline*}
			\sum_{b\in(\Zo)^*}\E_\rho\left[\left(c_b w_{\D,b}(b)\right)^2\right]\\
			\leq Cm^2\cdot 3^{5m}\left(\sum_{x\in\Zo}\sum_{|y-x|\leq \r+3^m+1}\left(\D_1\bar{g}_\lambda(x,y)\right)^2+\sum_{x,y\in\Zo}\left(\D_1\D_2\bar{g}_\lambda(x,y)\right)^2\right).
		\end{multline*}
		For the first term, we use \eqref{prop.H1Near} in Proposition~\ref{prop.SobolevNorm} with $\ell=\lfloor3^{6m+2\alpha m}(\r+3^m+1)\rfloor+1$, for the second term, we use \eqref{prop.pibpib'} in Proposition~\ref{prop.SobolevNorm}, and we can conclude
		\begin{equation}\label{eq.PibDL2}
			\sum_{b\in(\Zo)^*}\E_\rho\left[\left(c_b w_{\D,b}(b)\right)^2\right]\leq C\left(3^{-2\alpha m}\expec{\bar{G}_\lambda,-\Lb\bar{G}_\lambda}+3^{14m}\expec{-\Lb\bar{G}_\lambda,-\Lb\bar{G}_\lambda}\right).
		\end{equation}
		Therefore, we have
		\begin{equation}\label{eq.F.3}
			\vert \mathbf{F.3} \vert\leq C\expec{V,-\Lb V}^{\frac{1}{2}}\left(3^{-\alpha m}\expec{\bar{G}_\lambda,-\Lb \bar{G}_\lambda}^{\frac{1}{2}}+3^{7m}\expec{-\Lb\bar{G}_\lambda,-\Lb\bar{G}_\lambda}^{\frac{1}{2}}\right).
		\end{equation} 
		Finally, we combine \eqref{eq.F.1}, \eqref{eq.F.2} and \eqref{eq.F.3} to obtain the desired result \eqref{prop.FluxErrDeg2}.
	\end{proof}
	
	In the next step, we will derive a proposition about the carr\'e du champ of $\tilde{G}_\lambda$ in $1$ dimension degree $2$ similar to Proposition~\ref{prop.H1TwoScale}.
	
	\begin{proposition}\label{prop.H1TwoScaleDeg2}
		There exists a finite positive constant $C(c_{-},c_{+},\r,\rho)$ such that for every $\lambda\in(0,\tfrac{1}{2})$, the following estimates about the expectation and variance of the carr\'e du champ of $\tilde{G}_\lambda$ hold:
		\begin{multline}\label{eq.H1TwoScaleExpectDeg2}
			\Ll|\E_\rho\Ll[\sum_{b\in(\Zo)^*}\left(c_b(\pi_b\tilde{G}_\lambda)^2\right)\left(\eta\right)\Rr]-2\expec{\bar{G}_\lambda,-\Lb\bar{G}_\lambda}\Rr|\\
			\leq C\left(\sqrt{\log (\lambda^{-1})}+3^{14m}\right)(3^{-\alpha m}\sqrt{\log(\lambda^{-1})}+3^{7m}),
		\end{multline}
		\begin{equation}\label{eq.H1TwoScaleVarDeg2}
			\var_\rho\left[\sum_{b\in(\Zo)^*}\left(c_b(\pi_b\tilde{G}_\lambda)^2\right)(\eta)\right]\leq C 3^{38m}\left((\log(\lambda^{-1}))^{\frac{3}{2}}+1\right).
		\end{equation}
	\end{proposition}
	\begin{proof}
		The proof of \eqref{eq.H1TwoScaleExpectDeg2} is the same as \eqref{eq.H1TwoScaleExpect} in Proposition~\ref{prop.H1TwoScale} using Proposition~\ref{prop.L2FluxDeg2} and Proposition~\ref{prop.GreenEstimateDeg2} except we need to regain an upper bound for $\expec{\tilde{G}_\lambda,-\Lb\tilde{G}_\lambda}$ in the current setting:
		\begin{align*}
			\expec{\tilde{G}_\lambda-\bar{G}_\lambda,-\Lb\left(\tilde{G}_\lambda-\bar{G}_\lambda\right)}
			\leq C\sum_{z\in\Z_m}\sum_{b\in\ov{(z+\cu_m)^*}}\E_\rho\left[\left(\sum_{z'\in\Z_m}\pi_b\left([\D\bar{G}_\lambda]_m^{z'}\phi_m^{z'}\right)\right)^2\right].
		\end{align*}
		By the locality of the corrector $\phi_m^z$, \eqref{lem.Correctorlocal} in Lemma~\ref{lem.Homogenization}, the only non-vanishing terms for $\pi_b\phi_{m}^{z'}$ is exactly $\pi_b\phi_m^z$. Therefore, using Cauchy--Schwarz inequality, we can obtain:
		\begin{multline*}
			\expec{\tilde{G}_\lambda-\bar{G}_\lambda,-\Lb\left(\tilde{G}_\lambda-\bar{G}_\lambda\right)}\\
			\leq C\sum_{b\in(\Zo)^*}\E_\rho\left[\left(\sum_{z'\in\Z_m}\pi_b[\D\bar{G}_\lambda]_m^{z'}\cdot \phi_m^{z'}\right)^2\right] + C\sum_{z\in\Z_m}\sum_{b\in\ov{(z+\cu_m)^*}}\E_\rho\left[\left([\D\bar{G}_\lambda]_m^z\pi_b\phi_m^z\right)^2\right].
		\end{multline*}
		Since $[\D\bar{G}_\lambda]_{m}^z\in\F_0(\left(N_\r(z+\cu_m^+)\right)^c)$ and $\pi_b\phi_m^z\in\F_0(z+\cu_m^+)$, they are independent. Combining \eqref{lem.CorrectorH1L2} in Lemma~\ref{lem.Homogenization} and \eqref{prop.H11D} in Proposition~\ref{prop.SobolevNorm}, we have
		\begin{equation*}
			\sum_{z\in\Z_m}\sum_{b\in\ov{(z+\cu_m)^*}}\E_\rho\left[\left([\D\bar{G}_\lambda]_m^z\pi_b\phi_m^z\right)^2\right] \leq C\expec{\bar{G}_\lambda,-\Lb\bar{G}_\lambda}.
		\end{equation*} 
		Notice that $\sum_{z'\in\Z_m}\pi_b[\D\bar{G}_\lambda]_m^{z'}\cdot \phi_m^{z'}=w_{\D,\phi}(b)$, and therefore using \eqref{eq.PibDL2} we can obtain
		\begin{equation*}
			\sum_{b\in(\Zo)^*}\E_\rho\left[\left(\sum_{z'\in\Z_m}\pi_b[\D\bar{G}_\lambda]_m^{z'}\cdot \phi_m^{z'}\right)^2\right] \leq C\left(3^{-2\alpha m}\expec{\bar{G}_\lambda,-\Lb\bar{G}_\lambda}+3^{14m}\expec{-\Lb\bar{G}_\lambda,-\Lb\bar{G}_\lambda}\right).
		\end{equation*}
		Combining the previous two estimates, we have
		\begin{equation*}
			\expec{\tilde{G}_\lambda-\bar{G}_\lambda,-\Lb\left(\tilde{G}_\lambda-\bar{G}_\lambda\right)}\leq C\left(\expec{\bar{G}_\lambda,-\Lb\bar{G}_\lambda}+3^{14m}\expec{-\Lb\bar{G}_\lambda,-\Lb\bar{G}_\lambda}\right).
		\end{equation*}
		By triangular inequality, we have
		\begin{equation*}
			\expec{\tilde{G}_\lambda,-\Lb\tilde{G}_\lambda}\leq C\left(\expec{\bar{G}_\lambda,-\Lb\bar{G}_\lambda}+3^{14m}\expec{-\Lb\bar{G}_\lambda,-\Lb\bar{G}_\lambda}\right).
		\end{equation*}
		Plug this estimate in the proof of \eqref{eq.H1TwoScaleExpect} in Proposition~\ref{prop.H1TwoScale} and we can obtain the desired result \eqref{eq.H1TwoScaleExpectDeg2}.
		
		We next prove \eqref{eq.H1TwoScaleVarDeg2}. For simplicity, we denote the centered version of a random variable $X$ by $\bar{X}$:
		\begin{equation*}
			\bar{X}:=X-\E_\rho[X].
		\end{equation*}
		
		We denote the random variable containing bonds in $\ov{(z+\cu_m^+)^*}$ by $X_m^z$:
		\begin{equation*}
			X_m^z(\eta):=\sum_{b\in\ov{(z+\cu_m^+)^*}}\left(c_b(\pi_b\tilde{G}_\lambda)^2\right)(\eta),\quad \forall z\in\Z_m.
		\end{equation*} 
		
		We can decompose \eqref{eq.H1TwoScaleVarDeg2} by the near diagonal terms and the non-diagonal terms as follows:
		\begin{equation*}
			\var_\rho\left[\sum_{z\in\Z_m}X_m^z\right]=\sum_{\substack{z_1,z_2\in\Z_m\\ |z_1-z_2|<  3^{m+1}}}\expec{\bar{X_m^{z_1}},\bar{X_m^{z_2}}}+\sum_{\substack{z_1,z_2\in\Z_m\\ |z_1-z_2|\geq  3^{m+1}}}\expec{\bar{X_m^{z_1}},\bar{X_m^{z_2}}}.
		\end{equation*}
		We first treat the near diagonal terms, by AM--GM inequality, we have
		\begin{equation*}
			\sum_{\substack{z_1,z_2\in\Z_m\\ |z_1-z_2|< 3^{m+1}}}\left|\expec{\bar{X_m^{z_1}},\bar{X_m^{z_2}}}\right|\leq 5\sum_{z\in\Z_m}\E_\rho\left[\left(\bar{X_m^{z}}\right)^2\right]\leq 5\sum_{z\in\Z_m}\E_\rho\left[\left(X_m^{z}\right)^2\right].
		\end{equation*}
		We claim that we have the following bounds for $\E_\rho\left[\left(X_m^z\right)^2\right]$:
		\begin{equation}\label{eq.VarDigonal}
			\E_\rho\left[\left(X_m^{z}\right)^2\right]
			\leq C\cdot 3^{30m}\left(\sum_{z'\in\Z_m}\left(\tilde{\D}_{(3^{m+1})}\bar{g}_\lambda(z,z')\right)^2\right)^2,
		\end{equation}
		where $\tilde{\D}_{(r)}\bar{g}_\lambda$ is defined in \eqref{eq.defDTilde}.
		
		To prove \eqref{eq.VarDigonal}, we notice that each single term in $\pi_b \tilde{G}_\lambda$ for $b\in\ov{(z+\cu_m^+)^*}$ consists of one of the following random parts together with certain deterministic coefficient:
        \begin{equation*}
			\pi_b\ell_e\cdot\bar{\eta}_y,\ y \in \mathbb{Z},\quad \pi_b\ell_e\cdot \phi_{m}^{z'},\ z\neq z'\in\Z_m,\quad \pi_b\phi_m^z\cdot \bar{\eta}_y,\  y\in(N_\r(z+\cu_m^+))^c.
		\end{equation*}
		Therefore, utilizing \eqref{lem.Correctorlocal} in Lemma~\ref{lem.Homogenization}, we decompose  $\pi_b\tilde{G}_\lambda$ by classifying the location of the random part:
		\begin{equation*}
			\pi_b\tilde{G}_\lambda=\sum_{\substack{z'\in\Z_m\cup\{\infty\}\\|z'-z|\geq 2\cdot 3^m}}\sum_{\alpha\in\mcl{I}_b}c^{z'}_{\alpha}A_{\alpha}^{z'}B_{\alpha}^{z'},\quad \forall \alpha\in\mcl{I}_b,\ A_{\alpha}^{z'}\in\F_0(z+\cu_{m+1}),\ B_{\alpha}^{z'}\in \F_0(z'+\cu_m),
		\end{equation*} 
		we include the case that some terms only depend on $\mathscr{F}(z+\cu_{m+1})$ by assuming $B_{\alpha}^\infty=1$. Since for any $z'\in\Z_m$, $B_\alpha^{z'}$ is either $\bar{\eta}_z$ with $z\in z'+\cu_m$ or $\phi_m^{z'}$, we can utilize \eqref{lem.CorrectorMean} in Lemma~\ref{lem.Homogenization} to obtain
		\begin{equation}\label{eq.BMeanZero}
			\E_\rho\left[B_{\alpha}^{z'}\right]=0,\quad \forall z'\in\Z_m.
		\end{equation}
		Moreover, the coefficient $c_{\alpha}^{z'}$ and the cardinality of the index set $\mcl{I}_b$ satisfy
		\begin{equation}\label{eq.CoeffBound}
			|c_{\alpha}^{z'}|\leq \tilde{\D}_{(3^{m+1})}\bar{g}_\lambda(z,z'),\quad |c_{\alpha}^{\infty}|\leq \tilde{\D}_{(3^{m+1})}\bar{g}_\lambda(z,z),
		\end{equation}
		\begin{equation}\label{eq.IndexCardBound}
			|\mcl{I}_b|\leq 3^{m+2}.
		\end{equation}
		
		The next step is to plug the decomposition about $\pi_b\tilde{G}_\lambda$ into $X_m^z$:
		\begin{equation*}
			X_m^z=\sum_{b\in\ov{(z+\cu_m^+)^*}}\sum_{\substack{z_1\in\Z_m\cup\{\infty\}\\|z_1-z|\geq 2\cdot 3^m}}\sum_{\alpha_1\in\mcl{I}_b}\sum_{\substack{z_2\in\Z_m\cup\{\infty\}\\|z_2-z|\geq 2\cdot 3^m}}\sum_{\alpha_2\in\mcl{I}_b}c_{\alpha_1}^{z_1}c_{\alpha_2}^{z_2}(c_bA_{\alpha_1}^{z_1}A_{\alpha_2}^{z_2})B_{\alpha_1}^{z_1}B_{\alpha_2}^{z_2}.
		\end{equation*}
		Therefore, we can compute $\E_\rho\left[\left(X_m^z\right)^2\right]$ as:
		\begin{equation}\label{eq.DiagonalExpan}
			\begin{multlined}
				\E_\rho\left[\left(X_m^z\right)^2\right]= \sum_{b_i\in\ov{(z+\cu_m^+)^*}}\sum_{\alpha_1,\alpha_2\in\mcl{I}_{b_1}}\sum_{\alpha_3,\alpha_4\in\mcl{I}_{b_2}}\sum_{\substack{z_i\in\Z_m\cup\{\infty\}\\|z_i-z|\geq 2\cdot 3^m}}\\
				c_{\alpha_1}^{z_1}c_{\alpha_2}^{z_2}c_{\alpha_3}^{z_3}c_{\alpha_4}^{z_4}\expec{(c_{b_1}A_{\alpha_1}^{z_1}A_{\alpha_2}^{z_2})B_{\alpha_1}^{z_1}B_{\alpha_2}^{z_2},(c_{b_2}A_{\alpha_3}^{z_3}A_{\alpha_4}^{z_4})B_{\alpha_3}^{z_3}B_{\alpha_4}^{z_4}}.
			\end{multlined}
		\end{equation}
		We notice that if there exists an isolate $z_i\in\Z_m$ (i.e. $z_i\notin\{z_j,j\neq i\}$), and without loss of generality we can assume it is $z_1$, then we have
		\begin{equation*}
			(c_{b_1}A_{\alpha_1}^{z_1}A_{\alpha_2}^{z_2})B_{\alpha_2}^{z_2}(c_{b_2}A_{\alpha_3}^{z_3}A_{\alpha_4}^{z_4})B_{\alpha_3}^{z_3}B_{\alpha_4}^{z_4}\in\F_0((z_1+\cu_m)^c),\quad B_{\alpha_1}^{z_1}\in\F_0(z_1+\cu_m),
		\end{equation*}
		and therefore by independence and  \eqref{eq.BMeanZero} we have
		\begin{equation*}
			\expec{(c_{b_1}A_{\alpha_1}^{z_1}A_{\alpha_2}^{z_2})B_{\alpha_1}^{z_1}B_{\alpha_2}^{z_2},(c_{b_2}A_{\alpha_3}^{z_3}A_{\alpha_4}^{z_4})B_{\alpha_3}^{z_3}B_{\alpha_4}^{z_4}}=0.
		\end{equation*}
		Using the above observation, we can find all indexes $z_i$ in the non-zero terms in the expansion of $\E_\rho\left[\left(X_m^z\right)^2\right]$ belong to the following cases:
		\begin{enumerate}
			\item Two of $z_i$ equal to $z'\in\Z_m,\ |z'-z|\geq 2\cdot 3^m$, and the other two equal to $z''\in\Z_m,\ |z''-z|\geq 2\cdot 3^m$, $z'$ and $z''$ cannot be the same,
			\item At least two of $z_i$ equals $z'\in\Z_m,|z'-z|\geq 2\cdot 3^m$, and the others equal to $\infty$,
			\item All of $z_i$ equal to $\infty$.
		\end{enumerate}
		\begin{figure}[h!]
			\centering
			\includegraphics[width=0.8\linewidth]{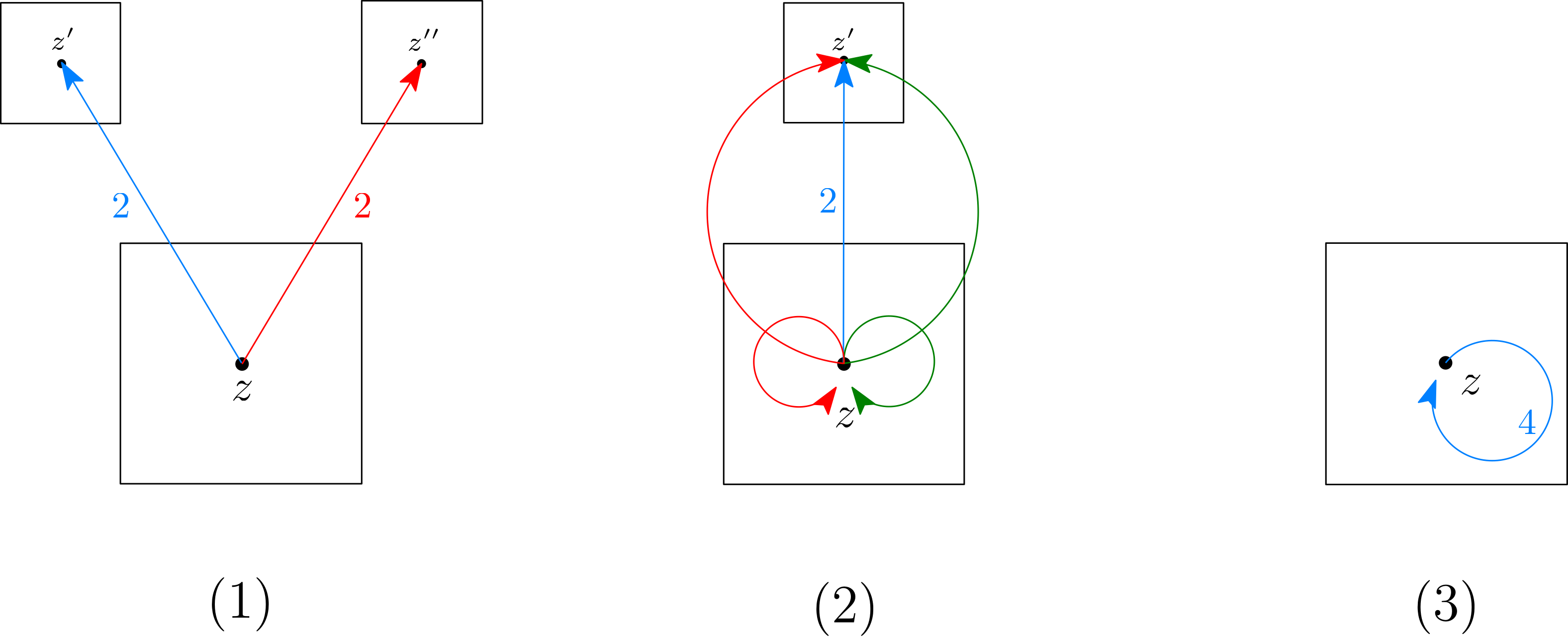}
			\caption{The diagram of non-vanishing terms. }\label{fig.Diagonal}
		\end{figure}
		
		Figure~\ref{fig.Diagonal} presents all non-vanishing terms (i.e. each point $z_i$ outside cube $z+\cu_{m+1}$ must be attached at least twice). Each arrow from $z$ to $\tilde{z}$ represents the appearance of  $c_{\alpha}^{\tilde{z}}A_{\alpha}^{\tilde{z}}B_{\alpha}^{\tilde{z}}$ with $\tilde{z}\in\{z,z',z''\}$, and the corresponding deterministic coefficient $c_{\alpha}^{\tilde{z}}$ can be controlled by $\tilde{\D}_{(3^{m+1})}(z,\tilde{z})$. The number $n$ attached to some arrows represents the corresponding arrow appears exactly $n$ times. We will use different colors to include multiple possible situations in one figure, in each subfigure, the arrows with the same color appear exactly once. Given $4$ arrows in one subfigure for one typical situation, multiplying the corresponding $c_{\alpha}^{\tilde{z}}A_{\alpha}^{\tilde{z}}B_{\alpha}^{\tilde{z}}$ together with the jump rate $c_{b_1}$, $c_{b_2}$ and taking expectation with respect to $\P_\rho$ yields one term in the summation of \eqref{eq.DiagonalExpan}.
		
		Using the $L^8$ bound for the corrector $\phi_m^z$, \eqref{lem.CorrectorLp} in Lemma~\ref{lem.Homogenization}, the ellipticity of the jump rate $c_b$, and H\"{o}lder inequality, we have
		\begin{multline*}
			\left|\expec{(c_{b_1}A_{\alpha_1}^{z_1}A_{\alpha_2}^{z_2})B_{\alpha_1}^{z_1}B_{\alpha_2}^{z_2},(c_{b_2}A_{\alpha_3}^{z_3}A_{\alpha_4}^{z_4})B_{\alpha_3}^{z_3}B_{\alpha_4}^{z_4}}\right|\\
			\leq  c_{+}^2\prod_{i=1}^{4}\left(\E_\rho\left[\left(A_{\alpha_{i}}^{z_i}\right)^8\right]\E_\rho\left[\left(B_{\alpha_{i}}^{z_i}\right)^8\right]\right)^{\frac{1}{8}}\leq C 3^{24m}.
		\end{multline*}
		By \eqref{eq.CoeffBound}, we can track the deterministic coefficients appear in each case, using AM--GM inequality and the above bounds for the expectation of the random part, we have for each single term belongs to $(1)$, $(2)$, and $(3)$:
		\begin{align*}
			|(1)|&\leq C3^{24m}\left(\tilde{\D}_{(3^{m+1})}\bar{g}_\lambda(z,z')\right)^2\left(\tilde{\D}_{(3^{m+1})}\bar{g}_\lambda(z,z'')\right)^2,\\
			|(2)|&\leq C3^{24m}\left(\tilde{\D}_{(3^{m+1})}\bar{g}_\lambda(z,z')\right)^2\Ll(\left(\tilde{\D}_{(3^{m+1})}\bar{g}_\lambda(z,z')\right)^2+\left(\tilde{\D}_{(3^{m+1})}\bar{g}_\lambda(z,z)\right)^2\Rr),\\
			|(3)|&\leq C3^{24m}\left(\tilde{\D}_{(3^{m+1})}\bar{g}_\lambda(z,z)\right)^4.
		\end{align*}  
		Summing the three parts and combining the  \eqref{eq.IndexCardBound}, we can obtain 
		\begin{align*}
			\E_\rho\left[\left(X_m^z\right)^2\right]
			&\leq C3^{30m}\sum_{z_1,z_2\in\Z_m}\left(\tilde{\D}_{(3^{m+1})}\bar{g}_\lambda(z,z_1)\right)^2\left(\tilde{\D}_{(3^{m+1})}\bar{g}_\lambda(z,z_2)\right)^2\\
			&\leq C3^{30m}\left(\sum_{z'\in\Z_m}\left(\tilde{\D}_{(3^{m+1})}g_\lambda(z,z')\right)^2\right)^2,
		\end{align*}
		which is exactly the desired estimate \eqref{eq.VarDigonal}.
		
		Using \eqref{prop.QuarticThreeBlock} $(\text{a})$ in Proposition~\ref{prop.GreenQuartic} by taking the parameter $r=3^{m+1}$ we have
		\begin{equation}\label{eq.NearDig}
			\sum_{\substack{z_1,z_2\in\Z_m\\ |z_1-z_2|\leq 3^{m+1}}}\left|\expec{\bar{X_m^{z_1}},\bar{X_m^{z_2}}}\right|\leq C3^{38m}\log(\lambda^{-1}). 
		\end{equation}
		
		We next consider the non-diagonal terms. We fix $z_1,z_2\in\Z_m$ such that $|z_1-z_2|\geq  3^{m+1}$, and we will focus on one particular term $\expec{\bar{X_m^{z_1}},\bar{X_m^{z_2}}}$. We will first derive a more precise decomposition of $\pi_b\tilde{G}_\lambda$ for any bond $b$. Suppose $b=\{x,x+1\}$ for $x\in z+\cu_m$, by definition of $\tilde{G}_\lambda$ in \eqref{eq.TwoscaleExDeg2}, we have
		\begin{equation*}
			\pi_b\tilde{G}_\lambda = \mathcal{R}_{*,b}+\mathcal{R}_{1,b}+\mcl{R}_{2,b},
		\end{equation*}
		where the main term $\mcl{R}_{*,b}$ and two remainder terms $\mcl{R}_{1,b}$ and $\mcl{R}_{2,b}$ are given by:
		\begin{equation*}
			\mathcal{R}_{*,b} := 2\sum_{y\in(z+\cu_{m+1})^c}\D_1\bar{g}_\lambda(x,y)\bar{\eta}_y\cdot \pi_b(\ell_e+\phi_m^z),
		\end{equation*}
		which collects all terms containing a random part $\bar{\eta}_y$ outside $z+\cu_{m+1}$ with a first order difference of $\bar{g}_\lambda$.
		\begin{multline*}
			\mcl{R}_{1,b} := 2 \sum_{y\in z+\cu_{m+1}}\D_1\bar{g}_\lambda(x,y)\bar{\eta}_y\cdot \pi_b\ell_e+\sum_{z'=z\pm 3^m}\pi_b\left([\D\bar{G}_\lambda]_m^{z'}\right)\phi_m^{z'}\\
			+\left(\frac{2}{|\cu_m|}\sum_{y\in z+\cu_{m+1}\backslash N_\r(z+\cu_m^+)}\sum_{x'\in z+\cu_m}\D_1\bar{g}_\lambda(x',y)\bar{\eta}_y-2\sum_{y\in z+\cu_{m+1}}\D_1\bar{g}_\lambda(x,y)\bar{\eta}_y\right)\cdot \pi_b\phi_m^z,
		\end{multline*}
		which collects all terms depends only on $z+\cu_{m+1}$.
		\begin{multline*}
			\mcl{R}_{2,b} := \sum_{z'\in \Z_m,|z'-z|>3^m}\pi_b\left([\D\bar{G}_\lambda]_m^{z'}\right)\phi_m^{z'}\\
			+2\sum_{y\in(z+\cu_{m+1})^c}\left(\frac{1}{|\cu_m|}\sum_{x'\in z+\cu_m}\D_1\bar{g}_\lambda(x',y)-\D_1\bar{g}_\lambda(x,y)\right)\bar{\eta}_y\cdot\pi_b\phi_m^z,
		\end{multline*}
		which is the remainder for $\pi_b\tilde{G}_\lambda$ after subtracting $\mcl{R}_{\ast,b}$ and $\mcl{R}_{1,b}$. 
		
		We further decompose $\mcl{R}_{1,b}$ and $\mcl{R}_{2,b}$ as a summation of single terms. We notice that $\mcl{R}_{1,b}\in\F_0(z+\cu_{m+1})$, and therefore we rewrite it as a summation of single terms:
		\begin{equation*}
			\mcl{R}_{1,b}=\sum_{\alpha\in \mcl{I}_{1,b}}c_{\alpha}A_{\alpha},\quad \forall \alpha\in \mcl{I}_{1,b},\ A_{\alpha}\in\F_0(z+\cu_{m+1}).
		\end{equation*}
		The coefficients $c_{\alpha}$ and the cardinality of the index set $\mcl{I}_{1,b}$ satisfy
		\begin{equation}\label{eq.CoeffIndexBound1}
			|c_{\alpha}|\leq \tilde{\D}_{(3^{m+1})}\bar{g}_\lambda(z,z),\quad |\mcl{I}_{1,b}|\leq 3^{m+2}.
		\end{equation}
		
		For $\mcl{R}_{2,b}$, we decompose it by classifying the location of the random part as before:
		\begin{equation*}
			\mcl{R}_{2,b} = \sum_{\substack{z'\in\Z_m\\|z'-z|\geq 2\cdot 3^m}}\sum_{\alpha\in\mcl{I}_{2,b}}c^{z'}_{\alpha}A_{\alpha}^{z'}B_{\alpha}^{z'},\quad \forall \alpha\in\mcl{I}_{2,b},\ A_{\alpha}^{z'}\in\F_0(z+\cu_{m+1}),\ B_{\alpha}^{z'}\in \F_0(z'+\cu_m).
		\end{equation*}
		Notice that the coefficients have a stronger bound than before, as a consequence of the first case of \eqref{eq.PibDG} and the fact that $\frac{1}{|\cu_m|}\sum_{x'\in z+\cu_m}\D_1\bar{g}_\lambda(x,y)-\D_1\bar{g}_\lambda(x,y),\ y\in (z+\cu_{m+1})^c$ can be decomposed into a summation of some second order finite difference of $\bar{g}_\lambda$:
		\begin{equation}\label{eq.CoeffIndexBound2}
			|c_{\alpha}^{z'}|\leq \D^{2}_{(3^{m+1})}\bar{g}_\lambda(z,z'),\quad |\mcl{I}_{2,b}|\leq 3^{m+2}.
		\end{equation}
		We claim that we have the following bounds for $\expec{\bar{X_m^{z_1}},\bar{X_m^{z_2}}}$:
		\begin{multline}\label{eq.VarNonDigonal}
			\left|\expec{\bar{X_m^{z_1}},\bar{X_m^{z_2}}}-64\chi(\rho)^2\mcl{M}(z_1,z_2)\right|\\
			\leq C3^{30m}\left(\mcl{N}_{1}(z_1,z_2)
			+ \mcl{N}_{2}(z_1,z_2)+\mcl{N}_{3}(z_1,z_2)+\mcl{N}_{4}(z_1,z_2)+\mcl{N}_5(z_1,z_2)\right),
		\end{multline}
		where $\mcl{M}(z_1,z_2)$, $\mcl{N}_{1}(z_1,z_2)$, $\mcl{N}_{2}(z_1,z_2)$, $\mcl{N}_3(z_1,z_2)$, $\mcl{N}_4(z_1,z_2)$, and $\mcl{N}_5(z_1,z_2)$ are given by
		\begin{multline*}
			\mcl{M}(z_1,z_2)\\
			:= \sum_{x_i\in z_i+\cu_m} \sum_{y_i\in (\cup_{i=1}^2 (z_i+\cu_{m+1}))^c}\prod_{i=1}^{2}\E_\rho\left[c_{x_i,x_i+1}\left(\pi_{x_i,x_i+1}(\ell_e+\phi_m^{z_i})\right)^2\right]\prod_{i,j=1}^{2}\D_1\bar{g}_\lambda(x_i,y_j),
		\end{multline*}
		which is the main parts containing the most possible single terms with only the first order difference of $\bar{g}_\lambda$. 
		\begin{align*}
			&\mcl{N}_{1}(z_1,z_2)\\
			&\quad:=\sum_{z_3,z_4\in\Z_m}\big(\D^2_{(3^{m+1})}\bar{g}_\lambda(z_1,z_3)\tilde{\D}_{(3^{m+1})}\bar{g}_\lambda(z_2,z_3)\tilde{\D}_{(3^{m+1})}\bar{g}_\lambda(z_1,z_4)\tilde{\D}_{(3^{m+1})}\bar{g}_\lambda(z_2,z_4)\\
			&\qquad+\D^2_{(3^{m+1})}\bar{g}_\lambda(z_2,z_3)\tilde{\D}_{(3^{m+1})}\bar{g}_\lambda(z_1,z_3)\tilde{\D}_{(3^{m+1})}\bar{g}_\lambda(z_1,z_4)\tilde{\D}_{(3^{m+1})}\bar{g}_\lambda(z_2,z_4)\big),
		\end{align*}
		which contains the most possible single terms as $\mcl{M}(z_1,z_2)$, but at least one second order difference of $\bar{g}_\lambda$ appears in $\D^2_{(3^{m+1})}\bar{g}_\lambda(z,z')$.
		
		$\mcl{N}_2(z_1,z_2)$, $\mcl{N}_3(z_1,z_2)$, $\mcl{N}_4(z_1,z_2)$, and $\mcl{N}_5(z_1,z_2)$ containing fewer terms than $\mcl{M}_{z_1,z_2}$ and $\mcl{N}_1(z_1,z_2)$ are defined as:
		\begin{equation*}
			\mcl{N}_{2}(z_1,z_2) := \sum_{z'\in\Z_m}\left(\tilde{\D}_{(3^{m+1})}\bar{g}_\lambda(z_1,z')\right)^2\left(\tilde{\D}_{(3^{m+1})}\bar{g}_\lambda(z_2,z')\right)^2,
		\end{equation*}
		\begin{multline*}
			\mcl{N}_{3}(z_1,z_2) := \sum_{z'\in\Z_m}\tilde{\D}_{(3^{m+1})}\bar{g}_\lambda(z_1,z')\tilde{\D}_{(3^{m+1})}\bar{g}_\lambda(z_2,z')\\
			\left(\tilde{\D}_{(3^{m+1})}\bar{g}_\lambda(z_1,z_1)+\tilde{\D}_{(3^{m+1})}\bar{g}_\lambda(z_1,z_2)\right)\left(\tilde{\D}_{(3^{m+1})}\bar{g}_\lambda(z_2,z_2)+\tilde{\D}_{(3^{m+1})}\bar{g}_\lambda(z_2,z_1)\right),
		\end{multline*}
		\begin{multline*}
			\mcl{N}_{4}(z_1,z_2) := \sum_{z'\in\Z_m}\left(\tilde{\D}_{(3^{m+1})}\bar{g}_\lambda(z_1,z')\right)^2\left(\tilde{\D}_{(3^{m+1})}\bar{g}_\lambda(z_2,z_2)+\tilde{\D}_{(3^{m+1})}\bar{g}_\lambda(z_2,z_1)\right)^2\\
			+\sum_{z'\in\Z_m}\left(\tilde{\D}_{(3^{m+1})}\bar{g}_\lambda(z_2,z')\right)^2\left(\tilde{\D}_{(3^{m+1})}\bar{g}_\lambda(z_1,z_1)+\tilde{\D}_{(3^{m+1})}\bar{g}_\lambda(z_1,z_2)\right)^2,
		\end{multline*}
		and
		\begin{multline*}
			\mcl{N}_5(z_1,z_2)\\
			:=\left(\tilde{\D}_{(3^{m+1})}\bar{g}_\lambda(z_1,z_1)+\tilde{\D}_{(3^{m+1})}\bar{g}_\lambda(z_1,z_2)\right)^2\left(\tilde{\D}_{3^{m+1}}\bar{g}_\lambda(z_2,z_2)+\tilde{\D}_{(3^{m+1})}\bar{g}_\lambda(z_2,z_1)\right)^2.
		\end{multline*}
		
		The philosophy in the proof of \eqref{eq.VarNonDigonal} is that the non-vanishing terms after taking expectation must not contain an isolated random part outside cubes $z_1+\cu_{m+1}$ and $z_2+\cu_{m+1}$. Specifically, we fix two bonds $b_1=\{x_1,x_1+1\}$ and $b_2=\{x_2,x_2+1\}$ with $x_1\in z_1+\cu_{m}$ and $x_2\in z_2+\cu_m$, and consider $\expec{\bar{c_{b_1}(\pi_{b_1}\tilde{G}_\lambda})^2,\bar{c_{b_2}(\pi_{b_2}\tilde{G}_\lambda})^2}$. We will discuss the contribution of $\mcl{R}_{*,b}$, $\mcl{R}_{1,b}$, and $\mcl{R}_{2,b}$.
		\medskip
		
		\textit{Case~1: only $\mcl{R}_{*,b_1}$ and $\mcl{R}_{*,b_2}$ appear.} We first notice that for $y_1\neq y_2\in (z_1+\cu_{m+1})^c$, we have
		\begin{equation*}
			\E_\rho\left[c_{b_1}\left(\pi_{b_{1}}(\ell_e+\phi_m^{z_1})\right)^2\bar{\eta}_{y_1}\bar{\eta}_{y_2}\right]=0.
		\end{equation*}
		For convenience, denote the random part in $z_1+\cu_{m+1}$ and $z_2+\cu_{m+1}$ by $X_1$ and $X_2$:
		\begin{equation*}
			X_1 := c_{b_1}\left(\pi_{b_{1}}(\ell_e+\phi_m^{z_1})\right)^2,\quad X_2 := c_{b_2}\left(\pi_{b_{2}}(\ell_e+\phi_m^{z_2})\right)^2,
		\end{equation*}
		and therefore decomposing $c_{b_i}\left(\mcl{R}_{*,b_i}\right)^2$ into diagonal terms and non-diagonal terms yields:
		\begin{align*}
			\bar{c_{b_i}\left(\mcl{R}_{*,b_i}\right)^2} &= 4\sum_{y_i\in(z_i+\cu_{m+1})^c}\D_1\bar{g}_\lambda(x_i,y_i)^2\bar{X_i\left(\bar{\eta}_{y_i}\right)^2}\\
			&\quad + 4X_i\sum_{y_i\neq y_i'\in(z_i+\cu_{m+1})^c}\D_1\bar{g}_\lambda(x_i,y_i)\D_1\bar{g}_\lambda(x_i,y_i')\bar{\eta}_{y_i}\bar{\eta}_{y_i'}.
		\end{align*}
		For each single term in $\expec{\bar{c_{b_1}(\mcl{R}_{*,b_1})^2},\bar{c_{b_2}(\mcl{R}_{*,b_2})^2}}$, if it contains an isolated $\bar{\eta}_{y_i}$, $\bar{\eta}_{y_i'}$, or $\bar{X_i(\bar{\eta_{y_i}})^2}$, it becomes zero after taking expectation. We observe that all non-vanishing terms belong to one of the following cases:
		\medskip
		\begin{enumerate}
			\item $y_i\neq y_i'\in (\cup_{i=1}^2 (z_i+\cu_{m+1}))^c$,\ $\{y_1,y_1'\}=\{y_2,y_2'\}$.
			\item $y_1=y_2\in z'+\cu_m$, $\min\{|z'-z_1|,|z'-z_2|\}\geq 2\cdot 3^m$, $y_1'=y_2'\in \cup_{i=1}^2(z_1+\cu_{m+1})$,  and other symmetric cases.
			\item $y_1=y_1'\in z'+\cu_m$, $\min\{|z'-z_1|,|z'-z_2|\}\geq 2\cdot 3^m$, $y_2\in (z_1+\cu_{m+1})\cup (z'+\cu_m)$, $y_2'\in \cup_{i=1}^2(z_i+\cu_{m+1})\cup(z'+\cu_m)$, and other symmetric cases.
			\item $y_1,y_1',y_2,y_2'\in \cup_{i=1}^2(z_1+\cu_{m+1})$, and at least one $y_1,y_1'$ belongs to $z_2+\cu_{m+1}$ or at least one $y_2,y_2'$ belongs to $z_1+\cu_{m+1}$.
		\end{enumerate}
		\begin{figure}[h!]
			\centering
			\includegraphics[width=0.8\linewidth]{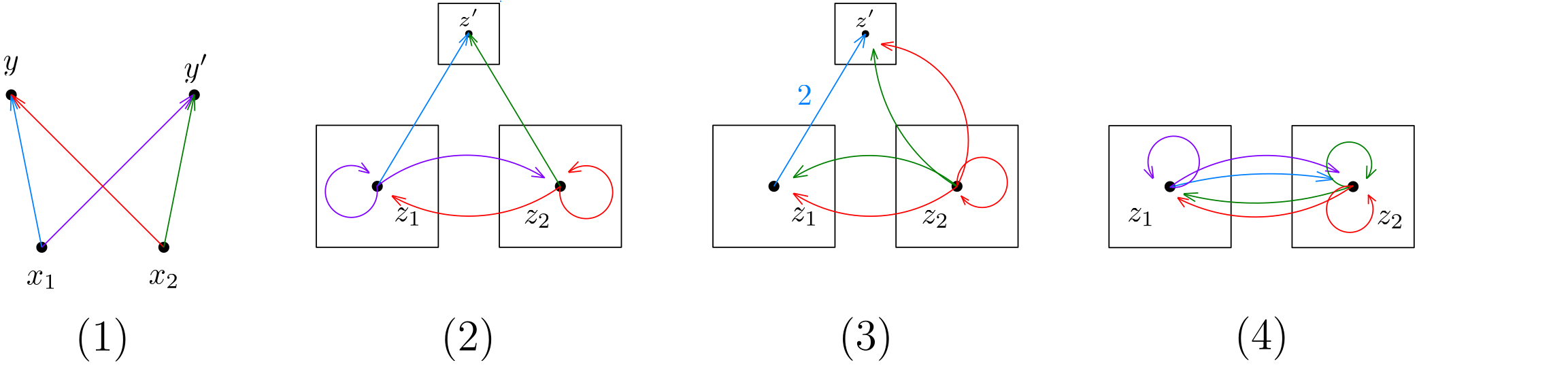}
			\caption{The diagram of non-vanishing terms. }\label{fig.NonDiagonalCase1}
		\end{figure}
		
		Figure~\ref{fig.NonDiagonalCase1} presents all non-vanishing terms (i.e. each point $z'$ outside cube $z+\cu_{m+1}$ must be attached at least twice, and the whole undirected graph is connected). In subfigure $(1)$, each arrow from $x_i$ to $y_i$ represents the appearance of $2\pi_{b_i}(\ell_e+\phi_m^{z_i})\D_1\bar{g}_\lambda(x_i,y_i)\bar{\eta}_{y_i}$. In other subfigures, each arrow from $z_i$ to $\tilde{z}$ with $\tilde{z}\in\{z_1,z_2,z'\}$ represents the appearance of  $2\pi_{b_i}(\ell_e+\phi_m^{z_i})\D_1\bar{g}_\lambda(x_i,y)\bar{\eta}_{y}$ with $y\in \tilde{z}+\cu_{\tilde{m}}$. The meaning of the number $n$ attached to some arrows and the color of the arrows is the same as the diagonal illustration Figure~\ref{fig.Diagonal}. 
		
		Therefore, similar to the diagonal case, by Cauchy--Schwarz inequality, AM--GM inequality and the $L^4$ estimate of the corrector $\phi_m^z$, \eqref{lem.CorrectorLp} in Lemma~\ref{lem.Homogenization}, for each individual term we have
		\begin{equation*}
			(1)= 64(\chi(\rho))^2\E_\rho[X_1]\E_\rho[X_2]\D_1\bar{g}_\lambda(x_1,y_1)\D_1\bar{g}_\lambda(x_1,y_1')\D_1\bar{g}_\lambda(x_2,y_1)\D_1\bar{g}_\lambda(x_2,y_1'),
		\end{equation*}
		\begin{align*}
			|(2)|&\leq C3^{12m}\mcl{N}_3(z_1,z_2),\\
			|(3)|&\leq C3^{12m}\left(\mcl{N}_2(z_1,z_2)+\mcl{N}_4(z_1,z_2)\right),\\
			|(4)|&\leq C3^{12m}\mcl{N}_5(z_1,z_2).
		\end{align*}
		
		Combining the above results, we can obtain:
		\begin{multline*}
			\left|\sum_{b_1\in\ov{(z_1+\cu_m)^*}}\sum_{b_2\in\ov{(z_2+\cu_m)^*}}\expec{\bar{c_{b_1}(\mcl{R}_{*,b_1})^2},\bar{c_{b_2}(\mcl{R}_{*,b_2})^2}}-64\chi(\rho)^2\mcl{M}(z_1,z_2)\right|\\
			\leq  C3^{18m}\left(\mcl{N}_2(z_1,z_2)+\mcl{N}_3(z_1,z_2)+\mcl{N}_4(z_1,z_2)+\mcl{N}_5(z_1,z_2)\right).
		\end{multline*}
		
		\medskip
		
		\textit{Case~2: at least one $\mcl{R}_{1,b_i}$ appears.} We may consider $\mcl{R}_{1,b_1}$ appearing as an example, and then we need to estimate
		\begin{equation*}
			\expec{\bar{c_{b_1}\mcl{R}_{1,b_1}\tilde{\mcl{R}_{b_1}}},\bar{c_{b_2}\tilde{\mcl{R}_{b_2}}\tilde{\mcl{R}_{b_2}'}}},
		\end{equation*}
		where 
		\begin{equation*}
			\tilde{\mcl{R}_{b_i}}, \tilde{\mcl{R}_{b_i}'}\in \{\mcl{R}_{*,b_i},\mcl{R}_{1,b_i},\mcl{R}_{2,b_i}\}.
		\end{equation*}
		By a similar independence analysis as before, we can observe all non-vanishing terms belong to one of the following cases:
		\begin{figure}[h!]
			\centering
			\includegraphics[width=0.8\linewidth]{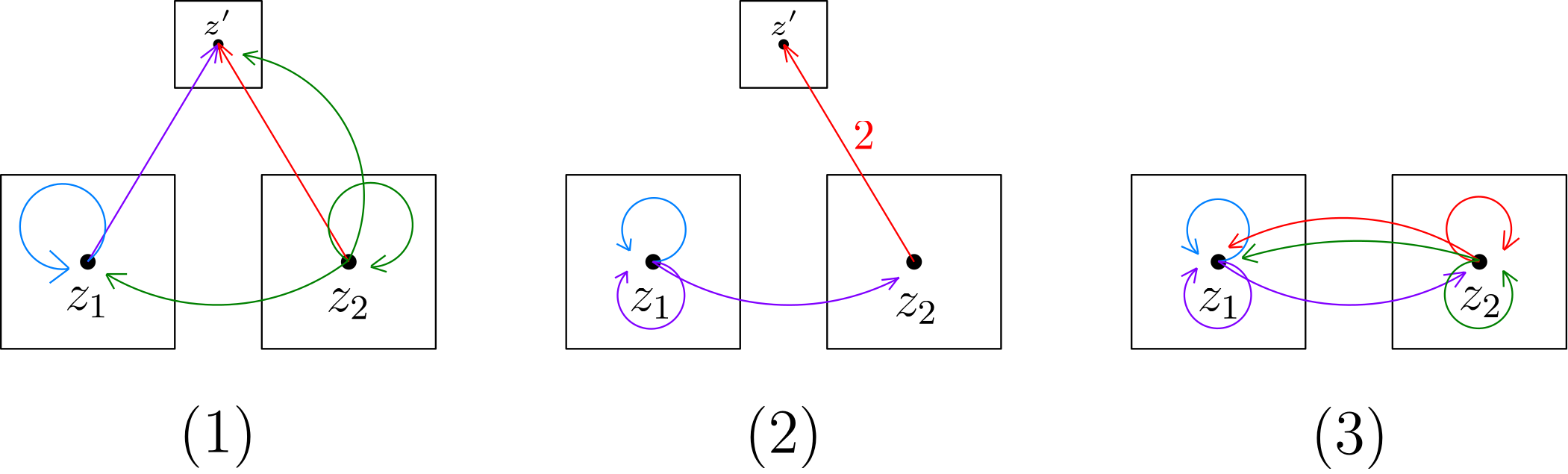}
			\caption{The diagram of non-vanishing terms. }\label{fig.NonDiagonalCase2}
		\end{figure}
		
		Combining \eqref{eq.CoeffIndexBound1}, \eqref{eq.CoeffIndexBound2}, by counting the number of the terms appearing and using the $L^8$ estimate about the corrector $\phi_m^z$ as we did in treating the near diagonal terms $\E_\rho\left[\left(X_m^z\right)^2\right]$, we have 
		\begin{equation*}
			\left|\expec{\bar{c_{b_1}\mcl{R}_{1,b_1}\tilde{\mcl{R}_{b_1}}},\bar{c_{b_2}\tilde{\mcl{R}_{b_2}}\tilde{\mcl{R}_{b_2}'}}}\right|\leq C3^{28m}\left(\mcl{N}_2(z_1,z_2)+\mcl{N}_3(z_1,z_2)+\mcl{N}_4(z_1,z_2)+\mcl{N}_5(z_1,z_2)\right).
		\end{equation*}
		
		\textit{Case~3: at least one $\mcl{R}_{2,b_i}$ appears.} We may consider $\mcl{R}_{2,b_1}$ appearing as an example, and then we need to estimate
		\begin{equation*}
			\expec{\bar{c_{b_1}\mcl{R}_{2,b_1}\tilde{\mcl{R}_{b_1}}},\bar{c_{b_2}\tilde{\mcl{R}_{b_2}}\tilde{\mcl{R}_{b_2}'}}}.
		\end{equation*}
		By a similar independence analysis as in $(1)$, we can observe all non-vanishing terms belong to one of the following cases:
		\begin{figure}[h!]
			\centering
			\includegraphics[width=0.8\linewidth]{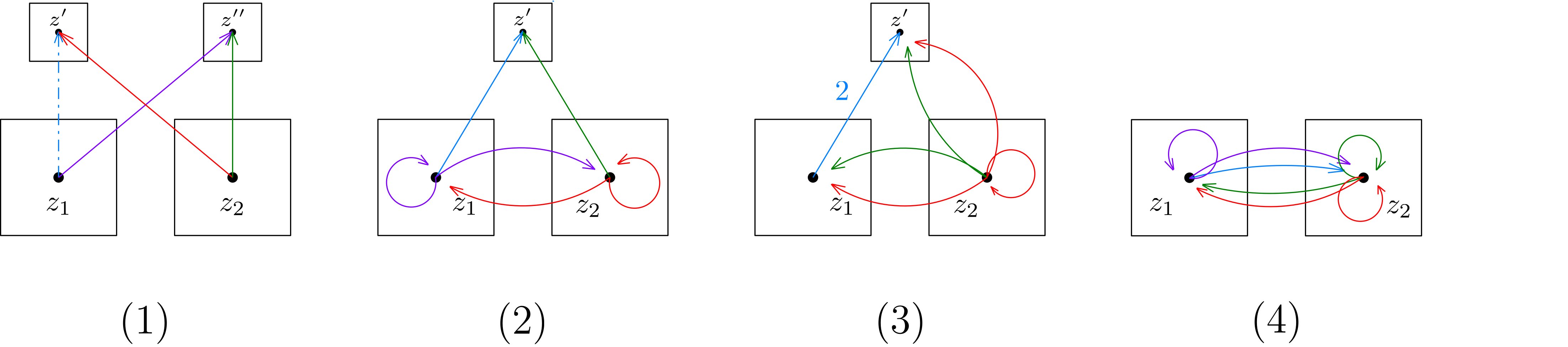}
			\caption{The diagram of non-vanishing terms. }\label{fig.NonDiagonalCase3}
		\end{figure}
		
		In subfigure $(1)$, the dash dotted arrow represents the terms with the coefficient of a second order finite difference of $\bar{g}_\lambda$.
		
		Combining \eqref{eq.CoeffIndexBound1} and \eqref{eq.CoeffIndexBound2}, by a similar independence estimate as in $(1)$ and counting the number of the terms appearing and using the $L^8$ estimate about the corrector $\phi_m^z$ as we did in treating the near diagonal terms $\E_\rho\left[\left(X_m^z\right)^2\right]$, we have 
		\begin{multline*}
			\left|\expec{\bar{c_{b_1}\mcl{R}_{2,b_1}\tilde{\mcl{R}_{b_1}}},\bar{c_{b_2}\tilde{\mcl{R}_{b_2}}\tilde{\mcl{R}_{b_2}'}}}\right|\\
			\leq C3^{28m}\left(\mcl{N}_1(z_1,z_2)+\mcl{N}_2(z_1,z_2)+\mcl{N}_3(z_1,z_2)+\mcl{N}_4(z_1,z_2)+\mcl{N}_5(z_1,z_2)\right).
		\end{multline*}
		
		Combining the computation in \textit{Case~1}, \textit{Case~2}, and \textit{Case~3}, we can obtain the desired estimate \eqref{eq.VarNonDigonal}. 
		
		Taking summation over $z_1,z_2\in\Z_m,|z_1-z_2|>3^{m+1}$, by \eqref{eq.VarNonDigonal}, we have
		\begin{multline*}
			\left|\sum_{\substack{z_1,z_2\in\Z_m\\ |z_1-z_2|> 3^{m+1}}}\expec{\bar{X_m^{z_1}},\bar{X_m^{z_2}}}-\sum_{\substack{z_1,z_2\in\Z_m\\ |z_1-z_2|> 3^{m+1}}}64\chi(\rho)^2\mcl{M}(z_1,z_2)\right|\\
			\leq C3^{30m}\sum_{z_1,z_2\in\Z_m}\left(\mcl{N}_1(z_1,z_2)+\mcl{N}_2(z_1,z_2)+\mcl{N}_3(z_1,z_2)+\mcl{N}_4(z_1,z_2)+\mcl{N}_5(z_1,z_2)\right).
		\end{multline*}
		Notice that using the $L^4$ estimate of the corrector $\phi_m^z$, \eqref{lem.CorrectorLp} in Lemma~\ref{lem.Homogenization}, we have
		\begin{multline*}
			\left|\sum_{\substack{z_1,z_2\in\Z_m\\ |z_1-z_2|> 3^{m+1}}}\mcl{M}(z_1,z_2)-\mcl{M}\right|\\
			\leq C3^{14m}\bigg(\sum_{z_1,z_2\in\Z_m}\left(\mcl{N}_2(z_1,z_2)+\mcl{N}_3(z_1,z_2)+\mcl{N}_4(z_1,z_2)+\mcl{N}_5(z_1,z_2)\right)\\
			+\sum_{z\in\Z_m}\left(\sum_{z'\in\Z_m}\left(\tilde{\D}_{(3^{m+1})}g_\lambda(z,z')\right)^2\right)^2\bigg),
		\end{multline*}
		where
		\begin{equation*}
			\mcl{M}=\sum_{x_i,y_i\in \Zo} \prod_{i=1}^{2}\E_\rho\left[c_{x_i,x_i+1}\left(\pi_{x_i,x_i+1}(\ell_e+\phi_m^{z_i})\right)^2\right]\prod_{i,j=1}^{2}\D_1\bar{g}_\lambda(x_i,y_j).
		\end{equation*}
		Using \eqref{prop.QuarticFourBlock}	$(\text{a})$ for $f(x)=\E_\rho\left[c_{x,x+1}\left(\pi_{x,x+1}(\ell_e+\phi_m^{z})\right)^2\right]$ with $x\in z+\cu_m$ in Proposition~\ref{prop.GreenQuartic}, we have
		\begin{equation*}
			\mcl{M}\leq C3^{12m}\left(\log(\lambda^{-1})+1\right).
		\end{equation*}
		Taking $r=3^{m+1}$ in Proposition~\ref{prop.GreenQuartic}, we also have
		\begin{multline*}
			\sum_{z_1,z_2\in\Z_m}\left(\mcl{N}_1(z_1,z_2)+\mcl{N}_2(z_1,z_2)+\mcl{N}_3(z_1,z_2)+\mcl{N}_4(z_1,z_2)+\mcl{N}_5(z_1,z_2)\right)\\
			+\sum_{z\in\Z_m}\left(\sum_{z'\in\Z_m}\left(\tilde{\D}_{(3^{m+1})}g_\lambda(z,z')\right)^2\right)^2\leq C3^{8m}\left(\left(\log(\lambda^{-1})^{\frac{3}{2}}\right)+1\right).
		\end{multline*}
		Combining the above estimates we can obtain
		\begin{equation}\label{eq.NonDig}
			\left|\sum_{\substack{z_1,z_2\in\Z_m\\ |z_1-z_2|> 3^{m+1}}}\expec{\bar{X_m^{z_1}},\bar{X_m^{z_2}}}\right|\leq C3^{38m}\left(\left(\log(\lambda^{-1})\right)^{\frac{3}{2}}+1\right).
		\end{equation}
		Therefore combining \eqref{eq.NearDig} and \eqref{eq.NonDig} we can obtain the desired estimate \eqref{eq.H1TwoScaleVarDeg2}.
		
	\end{proof}
	
	In the last part of this subsection, we provide a uniform bound for $\pi_b\tilde{G}_\lambda$, which is the difference after one jump for $\tilde{G}_\lambda$.
	\begin{corollary}\label{Cor.JumpDeg2}
		There exists a constant $C(c_{-},c_{+})$ such that, for any bond $b\in(\Zo)^*$ and $\lambda\in(0,1)$,  we have the following uniform bound for $\pi_b\tilde{G}_\lambda$:
		\begin{equation*}
			\norm{\pi_b\tilde{G}_\lambda}_{L^\infty}\leq C m 3^{3m}\lambda^{-\frac{1}{2}},\qquad\forall b\in(\Zo)^*,\ \lambda\in(0,1),\quad\forall m\in\N
		\end{equation*}
		\begin{proof}
			Suppose $b=\{x,x+1\}\in\ov{(z+\cu_m)^*}$, we can compute $\pi_b\tilde{G}_\lambda(\eta)$ explicitly as:
			\begin{equation*}
				\pi_{b}\tilde{G}_\lambda(\eta) = 2\sum_{y\in\Zo}\D_1\bar{g}_\lambda(x,y)\bar{\eta}_y(\bar{\eta}_{x}-\bar{\eta}_{x+1})+\sum_{z'\in\Z_m}\pi_{b}[\D\bar{G}_\lambda]_m^{z'}\cdot \phi_m^{z'}+[\D\bar{G}_\lambda]_m^z\pi_b\phi_m^z.
			\end{equation*}
			Taking supreme of $\eta\in\X$ for each terms yields:
			\begin{equation*}
				\norm{\pi_b\tilde{G}_\lambda}_{L^\infty}\leq 2\sum_{y\in\Zo}|\D_1\bar{g}_\lambda(x,y)|+\sum_{z'\in\Z_m}\norm{\pi_{b}[\D\bar{G}_\lambda]_m^{z'}}_{L^\infty}\norm{\phi_m^{z'}}_{L^\infty}+\norm{[\D\bar{G}_\lambda]_m^z}_{L^\infty}\norm{\pi_b\phi_m^z}_{L^\infty}.
			\end{equation*}
			Utilizing \eqref{lem.CorrectorLp} in Lemma~\ref{lem.Homogenization} we have
			\begin{equation*}
				\norm{\phi_{m}^z}_{L^\infty}\leq Cm3^{3m},\quad \norm{\pi_b\phi_m^z}_{L\infty}\leq C m3^{3m},\qquad \forall z\in\Z_m.
			\end{equation*}
			From the computation \eqref{eq.PibDG} we can derive
			\begin{align*}
				\sum_{z'\in\Z_m}\norm{\pi_{b}[\D\bar{G}_\lambda]_m^{z'}}_{L^\infty}&\leq \frac{2}{|\cu_m|}\sum_{y\in\Zo}\left(|\D_1\D_2\bar{g}_\lambda(y,x)|+|\D_1\bar{g}_\lambda(y,x)|+|\D_1\bar{g}_\lambda(y,x+1)|\right)\\
				&\leq\frac{4}{|\cu_m|}\sum_{y\in\Zo}\left(|\D_1\bar{g}_\lambda(y,x)|+|\D_1\bar{g}_\lambda(y,x+1)|\right).
			\end{align*}
			By the definition of the averaged slope $[\D\bar{G}_\lambda]_m^z$ in \eqref{eq.AveSlopeDeg2}, we can obtain
			\begin{equation*}
				\norm{[\D\bar{G}_\lambda]_m^z}_{L^\infty}\leq \sum_{y\in\Zo}\frac{2}{|\cu_m|}\sum_{x\in z+\cu_m}|\D_1\bar{g}_\lambda(x,y)|.
			\end{equation*}
			Combining the above estimates, we can obtain
			\begin{align*}
				\norm{\pi_b\tilde{G}_\lambda}_{L^\infty}&\leq Cm3^{3m}\sup_{x\in\Zo}\Ll\{\sum_{y\in\Zo}\left(|\D_1\bar{g}_\lambda(x,y)|+|\D_1\bar{g}_\lambda(y,x)|\right)\Rr\}\\
				&\leq Cm3^{3m}\sup_{x\in\Zo}\Ll\{\sum_{y\in\Zo}\left(|\bar{g}_\lambda(x,y)|+|\bar{g}_\lambda(y,x)|\right)\Rr\}\\
				&=Cm3^{3m}\sup_{x\in\Zo}\sum_{y\in\Zo}\bar{g}_\lambda(\{x,y\}),
			\end{align*}
			where in the last line, we use the symmetric assumption \eqref{eq.SymmetricAssumption}.
			
			Using \eqref{lem.TwoParticleTransition} in Lemma~\ref{lem.TransitionGreen}, we have
			\begin{align*}
				\sum_{y\in\Zo}\bar{g}_\lambda(\{x,y\}) 
				& =\int_0^\infty e^{-\lambda t}\sum_{y\in\Zo}\bar{p}_t(\{0,1\},\{x,y\})\, \d t\\
				& \leq C\int_{C_1}^{\infty}\frac{e^{-\lambda t}}{(1+t)}\sum_{y\in\Zo}\left(\exp\left(-\tfrac{y^2}{C_2t}\right)+\exp\left(-\tfrac{|y|}{C_2\log t}\right)\1_{|y|>\frac{C_3et}{\log t}}\right)\, \d t+C.
			\end{align*}
			Fix $t\geq C_1$, and a simple computation yields 
			\begin{equation*}
				\sum_{y\in\Zo}\left(\exp\left(-\tfrac{y^2}{C_2t}\right)+\exp\left(-\tfrac{|y|}{C_2\log t}\right)\1_{|y|>\frac{C_3 et}{\log t}}\right) \leq C\left(t^{\frac{1}{2}}+|\log t|\right),
			\end{equation*}
			and therefore we have 
			\begin{equation*}
				\sum_{y\in\Zo}\bar{g}_\lambda(\{x,y\})\leq C
				\int_0^\infty \frac{e^{-\lambda t}}{1+t}\left(t^{\frac{1}{2}}+|\log t|\right)\, \d t+C\leq C\lambda^{-\frac{1}{2}}.
			\end{equation*}
			Combining the above estimates, we can conclude Corollary~\ref{Cor.JumpDeg2}.
		\end{proof}
	\end{corollary}
	
	\subsection{Proof of Theorem~\ref{thm.main_next}}\label{subsec.ConvergenceDeg2}
	
	After establishing Proposition~\ref{prop.GreenEstimateDeg2}, Proposition~\ref{prop.L2FluxDeg2}, Proposition~\ref{prop.H1TwoScaleDeg2}, and Corollary~\ref{Cor.JumpDeg2} in previous two subsections, the proof for Theorem~\ref{thm.main_next} is almost identical to $d=2$ case in Section~\ref{subsec.MTildeCLT}, Section~\ref{subsec.tightness}, and Section~\ref{subsec.Convergence}. 
	
	The only difference is that for the tightness part, since the argument in Proposition~\ref{prop.tightGamma} for $d=2$ no longer apply because the cost of the flow connecting from $\delta_0$ to $p_\ell$ has order $\ell$ in $d=1$ instead of $\log \ell$; see lemma \ref{lem flow}. Therefore, we propose a different approach via the forward-backward martingale, which takes advantage of the reversibility of the process.
	
	We first confirm the martingale CLT of $\tilde{M}^N$ in $d=1$:
	
	\begin{proof}[Proof of Proposition~\ref{prop.CLT_MDeg2}]
		The proof is the same as the proof of Proposition~\ref{prop.CLT_M} utilizing Proposition~\ref{prop.GreenEstimateDeg2}, Proposition~\ref{prop.H1TwoScaleDeg2}, and Corollary~\ref{Cor.JumpDeg2} under the chosen scale \eqref{eq.scaleMDeg2}.
	\end{proof}
	
	This next part is devoted to the tightness of the occupation time.
	\begin{proposition}\label{prop.tightGammaDeg2}
		For $d = 1$, the occupation time $\{\Gamma^N\}_{N \in \N}$ defined in \eqref{eq.OccupationDeg2} is tight in $C(\R_+, \R)$.
	\end{proposition}
	
	\begin{proof}[Proof of Proposition~\ref{prop.tightGammaDeg2}]
		By Dynkin's martingale formula, 
		\[M_t^{N} = \frac{1}{\sqrt{N \log N}} G_{\lambda} (\eta_{Nt}) - \frac{1}{\sqrt{N \log N}} G_{\lambda} (\eta_{0}) - \frac{1}{\sqrt{N \log N}} \int_0^{Nt} \mathcal{L} G_\lambda (\eta_s) \, \d s\]
		is a martingale. Consider the reversed process $\{\eta_{NT-t}\}_{0\leq t \leq NT}$, then we have the following backward martingale,
		\[\mathfrak{m}_t^{N} = \frac{1}{\sqrt{N \log N}} G_{\lambda} (\eta_{NT-Nt}) - \frac{1}{\sqrt{N \log N}} G_{\lambda} (\eta_{NT}) - \frac{1}{\sqrt{N \log N}} \int_0^{Nt} \mathcal{L} G_\lambda (\eta_{NT-s}) \, \d s.\]
		After a change of variables,
		\[\mathfrak{m}_T^{N} - \mathfrak{m}_{T-t}^{N} =  \frac{1}{\sqrt{N \log N}} G_{\lambda} (\eta_{0})  - \frac{1}{\sqrt{N \log N}} G_{\lambda} (\eta_{Nt})  -   \frac{1}{\sqrt{N \log N}} \int_0^{Nt} \mathcal{L} G_\lambda (\eta_s) \, \d s.\]
		Thus,
		\[M_t^{N}  + \mathfrak{m}_T^{N} - \mathfrak{m}_{T-t}^{N} = -   \frac{2}{\sqrt{N \log N}} \int_0^{Nt} \mathcal{L} G_\lambda (\eta_s) \, \d s.\]
		Viewing the definition \eqref{eq.ResolventDeg2}, we obtain the following decomposition for $\Gamma^N_t$ defined in \eqref{eq.OccupationDeg2}
		\begin{equation}\label{decom tightness}
			\Gamma^N_t = \frac{1}{2} \big(M_t^{N}  + \mathfrak{m}_T^{N} - \mathfrak{m}_{T-t}^{N} \big) + \frac{1}{\sqrt{N \log N}} \int_0^{Nt} \lambda G_\lambda (\eta_s) \, \d s.
		\end{equation}
		Next, we prove the tightness of the terms on the {\rhs} respectively.
		
		Concerning the martingale, we use Doob's inequality and isometry of quadratic variation to get
		\begin{align*}
			\E_\rho \big[\sup_{0 \leq t \leq T} \big(M_t^{N} - \widetilde{M}_t^{N}\big)^2 \big ] &\leq 4 \E_\rho \big[ \big(M_T^{N} - \widetilde{M}_T^{N}\big)^2 \big ]\\
			&= 4\E_\rho \Ll[ \bracket{M^{N} - \widetilde{M}^{N}}_T \Rr]\\
			&= \frac{4T}{\log N} \langle \tilde{G}_\lambda - G_\lambda, - \mathcal{L} (\tilde{G}_\lambda - G_\lambda) \rangle.
		\end{align*}
		Recall $\lambda = 1/N$. By Proposition~\ref{prop.HomoResolvent} and the estimates in Proposition~\ref{prop.GreenEstimateDeg2}, the last term has order 
		\begin{align*}
			&\frac{4T}{\log N} \expec{\tilde{G}_\lambda-G_\lambda,-\L(\tilde{G}_\lambda-G_\lambda)}\\
			&\leq \frac{4CT}{\log N}\left(\left(3^{-2\alpha m}+\lambda3^{3m}\right)\expec{\bar{G}_\lambda,-\Lb\bar{G}_\lambda} + 3^{14m}\expec{-\Lb\bar{G}_\lambda,-\Lb\bar{G}_\lambda}\right) \\
			&\leq CT \Ll(3^{-2\alpha m} +  3^{3m}/N + 3^{14m} / \log N\Rr),
		\end{align*}
		which converges to zero as $N \rightarrow \infty$ with the choice 
		\begin{align}\label{eq.m_choice}
			1 \ll m \ll (\log N)^{\frac{1}{14}}.
		\end{align}
		The tightness of $M_\cdot^{N}$ follows from the convergence of $\widetilde{M}_\cdot^{N}$ immediately.
		
		The tightness of the reversed martingales on the right-hand side of \eqref{decom tightness} can be proved following the same argument.
		
		For the last term on the right-hand side of \eqref{decom tightness}, by using Cauchy--Schwarz inequality, we obtain a bound
		\begin{align*}
			&\E_\rho \Big[ \sup_{0 \leq t \leq T}\Big( \frac{1}{\sqrt{N \log N}} \int_0^{Nt} \lambda G_\lambda (\eta_s) \, \d s \Big)^2\Big] \\
			&\leq \frac{\lambda^2 N^2 T^2}{N\log N} \|G_\lambda\|_{L^2}^2\\
			&\leq \frac{3\lambda^2 N T^2}{\log N} \Big( \|G_\lambda - \tilde{G}_\lambda\|_{L^2}^2 + \| \tilde{G}_\lambda - \bar{G}_\lambda\|_{L^2}^2 + \|\bar{G}_\lambda\|_{L^2}^2 \Big).
		\end{align*}
		By Proposition \ref{prop.HomoResolvent}, Proposition \ref{prop.GreenEstimateDeg2} and Proposition \ref{prop.L2FluxDeg2}, the last bound also has order $3^{-2\alpha m} +  3^{3m}/N + 3^{14m} / \log N$, which will also vanish under the condition \eqref{eq.m_choice}. This concludes the proof.
	\end{proof}
	
	In the following proposition in parallel with Proposition~\ref{prop.RemainderConverge}, we identify the limit of the remainder.
	\begin{proposition}
		For $d=1$, the remainder $\{\Gamma^N-\tilde{M}^N\}_{N\in\N}$ converges weakly to the zero process.
	\end{proposition}
	\begin{proof}
		The proof is the same as the proof of Proposition~\ref{prop.RemainderConverge} utilizing Proposition~\ref{prop.GreenEstimateDeg2}, Proposition~\ref{prop.L2FluxDeg2}, and Proposition~\ref{prop.tightGammaDeg2} under the chosen scale \eqref{eq.scaleMDeg2}.
	\end{proof}
	
	Finally, we prove Theorem~\ref{thm.main_next}.
	
	\begin{proof}[Proof of Theorem~\ref{thm.main_next}]
		Notice that the case $\deg(f)>2/d$ is essentially Proposition~\ref{prop.KVnonGradient}. Thus it suffices to prove the case $d=1,\ \deg(f)=2$. 
		Note that for any centered local function $f$,
		\begin{equation*}
			\deg \left(f - \bar{f}^\prime (\rho)\bar{\eta}(0) - \frac{\bar{f}'' (\rho)}{2}\bar{\eta}(0)\bar{\eta}(1)\right)\geq 3.
		\end{equation*}
		Since $\deg(f)=2$, we have $\bar{f}^\prime (\rho)=0$, and therefore by Proposition~\ref{prop.KVnonGradient}, $\big(\Gamma^N_t (f) - \frac{\bar{f}''(\rho)}{2} \Gamma^N_t)\big)_{0 \leq t \leq T}$ converges to the zero process. Combining Proposition~\ref{prop.CLT_MDeg2} we conclude the proof.
	\end{proof}

	\appendix
	\section{Homogenization of the resolvent equation}\label{appendix}

    Throughout the appendix $G_\lambda$, $\bar{G}_\lambda$ represent the solutions of the resolvent equation \eqref{eq.resolvent} with $f=\bar{\eta}(0)$, \eqref{eq.ResolventSEP} for the case $\deg=1$ in all dimensions, and represent the solutions of the resolvent equation \eqref{eq.ResolventDeg2}, \eqref{eq.ResolventSSEP} for the special case $\deg=2$ in one dimension. $\tilde{G}_\lambda$ represents the corresponding two-scale expansion, \eqref{eq.TwoScaleExpan} for the case $\deg=1$ in all dimensions, and \eqref{eq.TwoscaleExDeg2} for the special case $\deg=2$ in one dimension.
	
	The following estimate is an analogue of that in \cite[(5.22)]{gu2024quantitative}, which controls the approximation error of the two-scale expansion in $\norm{\cdot}_{1,\lambda}$ norm.
	\begin{proposition}\label{prop.HomoResolvent}
		There exists a constant $C(d,c_{-},c_{+},\r,\rho)$ such that for any $\lambda>0$, the following estimate holds:
		\begin{equation}\label{eq.HomoResolvent}
			\begin{multlined}
				\lambda \norm{\tilde{G}_\lambda - G_\lambda}^2_{L^2} + \expec{\tilde{G}_\lambda-G_\lambda,-\L(\tilde{G}_\lambda-G_\lambda)}\\
				\leq C\left(\left(3^{-2\alpha m}+\lambda3^{3m}\right)\expec{\bar{G}_\lambda,-\Lb\bar{G}_\lambda} + 3^{14m}\expec{-\Lb\bar{G}_\lambda,-\Lb\bar{G}_\lambda}\right).
			\end{multlined}
		\end{equation}
	\end{proposition}
	
	\begin{proof}
		We multiply $\tilde{G}_\lambda-G_\lambda$ for both sides of the identity
		\begin{equation*}
			(\lambda-\L)G_\lambda=(\lambda-\Lb)\bar{G}_\lambda,
		\end{equation*}
		after taking expectation with respect to $\P_\rho$ and rearranging the terms we can obtain:
		\begin{multline}\label{eq.HomH1Estimate}
			\expec{\tilde{G}_\lambda-G_\lambda, -\L \left(\tilde{G}_\lambda-G_\lambda\right)} + \lambda \norm{\tilde{G}_\lambda-G_\lambda}_{L^2}^{2} \\
			= \expec{\tilde{G}_\lambda-G_\lambda,-\L\tilde{G}_\lambda+\Lb\bar{G}_\lambda} + \lambda\expec{\tilde{G}_\lambda-G_\lambda,\tilde{G}_\lambda-\bar{G}_\lambda}.
		\end{multline}
		We can estimate the first term of the right-hand side of \eqref{eq.HomH1Estimate} by \eqref{prop.FluxErr} in Proposition~\ref{prop.L2FluxErr} (we use \eqref{prop.FluxErrDeg2} in Proposition~\ref{prop.L2FluxDeg2} for the special case $d=1, \deg=2$):
		\begin{equation}\label{eq.FluxEstimate}
			\begin{multlined}
				\left|\expec{\tilde{G}_\lambda-G_\lambda,-\L\tilde{G}_\lambda+\Lb\bar{G}_\lambda}\right|\\
				\leq C\expec{\tilde{G}_\lambda-G_\lambda,-\Lb\left(\tilde{G}_\lambda-G_\lambda\right)}^{\frac{1}{2}}\left(3^{-\alpha m}\expec{\bar{G}_\lambda,-\Lb\bar{G}_\lambda}^{\frac{1}{2}}+3^{7m}\expec{-\Lb\bar{G}_\lambda,-\Lb\bar{G}_\lambda}^{\frac{1}{2}}\right)
			\end{multlined}
		\end{equation}
		By \cite[Corollary~2.3]{gu2025relaxation}, we can compare the Dirichlet forms for $\tilde{G}_\lambda-G_\lambda$ associated with generators $\L$ and $\Lb$:
		\begin{equation}
			\expec{\tilde{G}_\lambda-G_\lambda,-\Lb\left(\tilde{G}_\lambda-G_\lambda\right)} \leq C\expec{\tilde{G}_\lambda-G_\lambda, -\L \left(\tilde{G}_\lambda-G_\lambda\right)}.
		\end{equation}
		Therefore, we can use AM--GM inequality in \eqref{eq.FluxEstimate}. A carefully chosen weight can let $\expec{\tilde{G}_\lambda-G_\lambda, -\L \left(\tilde{G}_\lambda-G_\lambda\right)}$ be reabsorbed in the right-hand side of \eqref{eq.HomH1Estimate}:
		\begin{multline}\label{eq.FluxEsR}
			\left|\expec{\tilde{G}_\lambda-G_\lambda,-\L\tilde{G}_\lambda+\Lb\bar{G}_\lambda}\right|\\
			\leq \frac{1}{2}\expec{\tilde{G}_\lambda-G_\lambda, -\L \left(\tilde{G}_\lambda-G_\lambda\right)}+C\left(3^{-2\alpha m}\expec{\bar{G}_\lambda,-\Lb\bar{G}_\lambda}+3^{14m}\expec{-\Lb\bar{G}_\lambda,-\Lb\bar{G}_\lambda}\right).
		\end{multline} 
		We can use Cauchy--Schwarz inequality to handle the second term in \eqref{eq.HomH1Estimate}:
		\begin{equation}\label{eq.HomRemainder}
			\lambda\expec{\tilde{G}_\lambda-G_\lambda, \tilde{G}_\lambda-\bar{G}_\lambda} \leq \frac{\lambda}{2}\norm{\tilde{G}_\lambda-G_\lambda}^{2}_{L^2}+\frac{\lambda}{2}\norm{\tilde{G}_\lambda-\bar{G}_\lambda}^{2}_{L^2}.
		\end{equation}
		We can use \eqref{prop.L2Err} in Proposition~\ref{prop.L2FluxErr} (we use \eqref{prop.L2ErrDeg2} in Proposition~\ref{prop.L2FluxDeg2} for the special case $d=1, \deg=2$) to control the second term in \eqref{eq.HomRemainder}:
		\begin{equation}\label{eq.HomL2}
			\norm{\tilde{G}_\lambda-\bar{G}_\lambda}^{2}_{L^2} \leq C3^{3m}\expec{\bar{G}_\lambda,-\Lb\bar{G}_\lambda}.
		\end{equation}
		We can combine the estimates \eqref{eq.FluxEsR}, \eqref{eq.HomRemainder}, and \eqref{eq.HomL2} to obtain the desired result \eqref{eq.HomoResolvent}.
	\end{proof}
	
	To show $\tilde{G}_\lambda$ is a reasonable substitution of $G_\lambda$, we take the $d=2$ case as an example. We first use the reversibility of the process and Proposition~\ref{prop.relaxation} to obtain
	\begin{align*}
		\E_\rho \Ll[G_\lambda^2\Rr] &= \int_0^\infty \, \d t \int_0^\infty \, \d s \, e^{-\lambda (t+s)} \E_\rho \Ll[P_t (\bar{\eta}(0)) P_s(\bar{\eta}(0))\Rr] \\
		&= \int_0^\infty te^{-\lambda t} \E_\rho \Ll[\Ll(P_{t/2} (\bar{\eta}(0))\Rr)^2 \Rr]\, \d t  \leq C \lambda^{-1}.
	\end{align*}
	By \eqref{eq.resolvent} we have the relation
	\begin{equation*}
		\expec{G_\lambda,-\L G_\lambda}=\expec{G_\lambda,\bar{\eta}(0)}-\lambda\E_\rho[G_\lambda^2].
	\end{equation*}
	Since $G_\lambda (\eta) = \int_0^{\infty} e^{-\lambda t} P_t (\bar{\eta}(0)) \, \d t$, we can estimate the first term $\expec{G_\lambda,\bar{\eta}(0)}$ as:
	\begin{equation*}
		\expec{G_\lambda,\bar{\eta}(0)}=\int_{0}^{\infty}e^{-\lambda t}\expec{P_t(\bar{\eta}(0)),\bar{\eta}(0)}\,\d t=\int_{0}^{\infty}e^{-\lambda t}\E_{\rho}\left[\left(P_{t/2}(\bar{\eta}(0))\right)^2\right]\sim\log(\lambda^{-1}),
	\end{equation*}
	where the last estimate utilizes Proposition~\ref{prop.relaxation}. Therefore we obtain the asymptotic for $\expec{G_\lambda,-\L G_\lambda}$:
	\begin{equation*}
		\expec{G_\lambda,-\L G_\lambda}\sim\log(\lambda^{-1}),\quad\text{as}\ \lambda\rightarrow 0.
	\end{equation*}
	
	Plugging in the estimates for $\bar{G}_\lambda$ in dimension $d=2$ developed in Section~\ref{subsec.HomH1}, the upper bound in Proposition~\ref{prop.HomoResolvent} becomes:
	\begin{equation*}
		C((3^{-2\alpha m}+\lambda 3^{3m})\log(\lambda^{-1})+3^{14m})\ll\log(\lambda^{-1}),
	\end{equation*}
	when we choose the mesoscopic scale $m$ appropriately:
	\begin{equation*}
		1\ll3^{14m}\ll\log(\lambda^{-1}).
	\end{equation*}
	
	The above computations justify the heuristic $\tilde{G}_\lambda\simeq G_\lambda$ discussed in the introduction in $\norm{\cdot}_{1,\lambda}$ sense.

    \section*{Statement on A.I. use}
    The manuscript was written by the authors. During the preparation of Proposition~\ref{prop.GreenQuartic}, the second author used a proof sketch generated by ChatGPT 5.4 Plus as a reference; the final proof and verification are due to the authors.

	\section*{Acknowledgements}
	The research of CG and LY is supported by the National Natural Science Foundation of China (Nos. 12301166, 12595280, 12595284). LZ thanks the financial support from the National Natural Science Foundation of China (Nos. 12401168, 12371142).
	We thank Claudio Landim for the comments on the preliminary version of the manuscript.

	\bibliographystyle{plain}
	\bibliography{KawasakiRef}
	
\end{document}